\documentclass[11pt,subeqn,subarray]{article}
\usepackage{amssymb,amsmath,amsthm}
\usepackage{graphicx}
\usepackage{flafter}
\usepackage[hidelinks]{hyperref}
\usepackage{multirow}
\usepackage{float}
\newtheorem{theorem}{Theorem}[section]
\newtheorem{lemma}[theorem]{Lemma}
\newtheorem{definition}[theorem]{Definition}
\newtheorem{remark}[theorem]{Remark}
\newtheorem{corollary}[theorem]{Corollary}

\def\sqr#1#2{{\vcenter{\vbox{\hrule height.#2pt
              \hbox{\vrule width.#2pt height#1pt \kern#1pt \vrule width.#2pt}
              \hrule height.#2pt}}}}

\def\be{\begin{equation}}
\def\ee{\end{equation}}
\def\bea{\begin{eqnarray}}
\def\eea{\end{eqnarray}}
\def\ca{{\cal A}}

\def\cd{{\cal D}}

\def\l{\lambda}

\def\g{\gamma}
\def\t{\tau}

\def\m{\mu}

\def\a{\alpha}
\def\b{\beta}

\def\d{\delta}

\def\dbC{\mathbb{C}}

\def\dbR{\mathbb{R}}

\def\cA{{\cal A}}

\def\cD{{\cal D}}

\def\cH{{\cal H}}

\def\Si{\Sigma}

\def\ss{\smallskip}
\def\ms{\medskip}

\def\q{\quad}
\def\qq{\qquad}
\def\3n{\negthinspace \negthinspace \negthinspace }
\def\2n{\negthinspace \negthinspace }
\def\1n{\negthinspace }

\def\limsup{\mathop{\overline{\rm lim}}}

\def\lan{\mathop{\langle}}
\def\ran{\mathop{\rangle}}

\def\cd{\cdot}

\def\Re{{\mathop{\rm Re}\,}}
\def\Im{{\mathop{\rm Im}\,}}

\def\({\Big (}
\def\){\Big )}
\def\[{\Big[}
\def\]{\Big]}
\def\bde{\begin{definition}}
\def\ede{\end{definition}}
\def\be{\begin{equation}}
\def\bel{\begin{equation}\label}
\def\ee{
\end{equation}}
\def\bex{\begin{example}}
\def\eex{\end{example}}
\def\bt{\begin{theorem}}
\def\et{\end{theorem}}
\def\bc{\begin{corollary}}
\def\ec{\end{corollary}}
\def\bl{\begin{lemma}}
\def\el{\end{lemma}}
\def\bp{\begin{proposition}}
\def\ep{\end{proposition}}
\def\bas{\begin{assumption}}
\def\eas{\end{assumption}}
\def\br{\begin{remark}}
\def\er{\end{remark}}
\def\ba{\begin{array}}
\def\ea{\end{array}}
\def\ed{\end{document}}
\def\ds{\displaystyle}

\def\ns{\noalign{\ss}}

\theoremstyle{remark}
\theoremstyle{plain}
\newtheorem{proposition}[theorem]{Proposition}

\usepackage{color}

\usepackage{enumitem}

\newcommand{\curly}[1]{\left\{#1\right\}}

\usepackage{caption}
\usepackage{subcaption}
\usepackage{cleveref}
\allowdisplaybreaks

\usepackage[markup=underlined]{changes}
\definechangesauthor[color=blue]{ck}

\begin{document}
\title
{\bf Regularity Analysis for Two Coupled Second Order Evolution Equations
\thanks{CHXD and QZ was supported by the National Natural Science Foundation of China (grants No. 12271035, 12131008) and Beijing Municipal Natural Science Foundation (grant No. 1232018).}}
\author{
		Chenxi Deng\thanks{School of Mathematics and Statistics,
			Beijing Institute of Technology, Beijing 100081, China (email: chenxideng@bit.edu.cn).}
		\and Zhaobin Kuang\thanks{Computer Science Department, Stanford University, Stanford 94305, U.S.A.(email: zhaobin.kuang@gmail.com).}
        \and Zhuangyi Liu\thanks{Department of Mathematics and Statistics, University of Minnesota, Duluth, MN 55812, U.S.A.(email: zliu@d.umn.edu).}
        \and
		Qiong Zhang\thanks{  School of Mathematics and Statistics, Beijing Key Laboratory on MCAACI,
			Beijing Institute of Technology, Beijing 100081, China (email: zhangqiong@bit.edu.cn)}
	}

\maketitle
\date{}
{\small {\bf Abstract.} We investigate the regularity of the $C_0$-semigroup associated with a system of two coupled second order evolution equations with indirect damping, whose stability was recently studied in \cite{Hao-Kuang-Liu-Yong 2025}.  By deriving the asymptotic expression of the generator’s eigenvalues, we partition the parameter space into several disjoint regions, where the semigroup exhibits either analyticity or Gevrey class regularity. Together with the estimate  of the resolvent of the generator on the imaginary axis, we give a complete and sharp regularity characterization for this system. 
\vskip 2mm

{\small {\it Keywords. Coupled second order evolution equations; Indirectly damping; Analyticity; Gevrey class regularity }

\ms

{\bf MSC (2010)}: 35B65, 35K90, 35L90, 47D03, 93D05 \hspace{4 in}
\section{Introduction}
\label{sec:intro}
\setcounter{equation}{0}
\setcounter{theorem}{0}

Let $H$ be a complex Hilbert space with the inner product
$\lan\cd\,,\cd\ran$ and the induced norm $\|\cd\|$. We consider the
following coupled second order evolution equations:
\bel{1.1}\left\{\2n\ba{ll}
\ns\ds u_{tt}=-aA^\g u+ bA^\a y_t,\\
\ns\ds y_{tt}=-Ay - bA^\a u_t - k A^\b y_t,\\
\ns\ds u(0)=u_0,\quad u_t(0)=v_0, \quad y(0)=y_0, \quad y_t(0)=w_0 \ea\right.\ee
in the space $H\times H$, where $A$ is a self-adjoint and positive definite operator; $b\ne0$, $a,k>0$. The parameters 
$(\a,\b,\g)\in E:= [0, \frac{\g+1}{2}]\times[0,1]\times [\frac{1}{2}, 2]$ represent the order of coupling, damping, and indirectly damped equation, with respect to the operator $A$. The well-posedness and asymptotic stability of this system were investigated in \cite{Hao-Kuang-Liu-Yong 2025}. They proved that the system generates a strongly continuous semigroup of contractions $e^{\mathcal{A}_{\a,\b,\g}t}$ on the Hilbert space $D(A^\frac{\gamma}{2})\times H \times D(A^\frac{1}{2})\times H$. When $(a,\gamma)\ne (1,1)$, the parameter region $E$ was partitioned into five sub-regions $S_1(\gamma), S_2(\gamma), \cdots, S_5(\gamma)$  (see Figure \ref{fig:r0} in next Section). By the frequency domain method and spectral analysis, they showed that the system is exponentially stable in $S_1(\gamma)$; polynomially stable with optimal order in $S_2(\gamma),S_3(\gamma), S_4(\gamma)$; and strongly stable in $S_5(\gamma)$. When $(a,\gamma)=(1,1)$, the exponentially stable region $S_1(1)$ expands due to the same wave speed phenomenon. 
With a different partition, this case was treated separately. 

Our goal is to add a complete regularity analysis to this system. If $(a,\gamma)\neq (1,1)$, let $R(\gamma)$ be the interior of the exponentially stable region $S_1(\gamma)$ given in \cite{Hao-Kuang-Liu-Yong 2025}, together with its boundary $\beta=1$, i.e., 

\begin{align}\label{eq:R}
    R(\gamma):= \left\{ \begin{array}{ll} \Big\{(\a,\beta,\gamma)\in S_1(\g)\bigm| 0 < 2\a-\beta < \g, \; 2\a + \beta - \g>0,\; 0< \beta\le 1\Big\},  & 1< \g \le 2 \\
    \Big\{(\a,\beta,\gamma)\in S_1(\g)\bigm| 0 < 2\a-\beta < \g, \; 2\a + \beta + \g > 2,\; 0< \beta\le 1\Big\} \bigcup \Big\{(\frac{\g+1}{2}, 1, \g)\Big\} & \frac{1}{2} \le \g \le 1.
    \end{array} \right.
\end{align}  
If $a=\gamma=1$, denote

\begin{align}\label{eq:tildeR}
    \widetilde{R}:=\Big\{(\a,\b)\in [0,1]^2\bigm| 0 < 2\alpha-\beta < 1, 0<\beta\le 1 \Big\}\bigcup \{ (1,1)\}.
\end{align} 

We only need to consider the regularity in regions  $R(\gamma)$ and $ \tilde{R}$, since there is always a sequence of system eigenvalues with a vertical asymptote if $(\a,\b,\g)\notin R(\gamma)$ or $(\a,\b)\notin \tilde{R}$, which does not meet the necessary spectrum condition for a differentiable semigroup \cite{Hao-Kuang-Liu-Yong 2025, Hao-Liu-Yong 2015}.  

Let us first give a brief review about the development in the study of stability and regularity of an abstract system.  Around the late 1980s and early 1990s, the model of directly damped system
\bel{1.2}
u_{tt} = -Au - Bu_t 
\ee
was investigated where $A$ is a self-adjoint positive definite operator in a Hilbert space, $B$ is equivalent to $A^\alpha$, and the parameter $\alpha$ represents the relative order of damping. When $\alpha=0, \frac{1}{2}, 1$, it models the viscous damping, the structural damping, and the Kelvin-Voigt damping, respectively. The system was shown to be exponentially stable for $\alpha \in [0, 1]$, analytic for $\alpha \in [\frac{1}{2}, 1]$; of Gevrey class $\delta>\frac{1}{2\alpha}$ for $\alpha \in (0, \frac{1}{2})$; not differentiable for $\alpha=0$, see \cite{Chen-Triggiani 1989, Chen-Triggiani 1990, Huang 1985, Huang 1988, Huang-Liu 1988} and the references therein. Subsequently, the optimal polynomial stability of order $\frac{1}{2|\alpha|}$ for $\alpha<0$ was added to this list \cite{Liu-Zhang 2015}.

In 1993, David Russell \cite{Russell 1993} introduced an abstract framework of indirectly damped system,
i.e., a conservative system coupled with a directly damped system. He pointed out that achieving a stability and regularity profile similar to that of the directly damped system is desirable, but remains a challenging task. He wrote in that paper, ``At the present time it does not appear possible to give a result for indirect mechanisms that even approaches the known direct damping results just listed in mathematical generality''.
 Since then, significant progress has been made in this direction, largely due to advances in the frequency domain characterization of semigroup properties and its applications \cite{ Batty-Chill-Tomilov 2016, Borichev-Tomilov 2010, Liu-Rao 2005, Liu-Yong 1998, Taylor 1989}
In particular, extensive studies have been conducted on the following thermoelastic system:
\bel{1.3}\left\{\2n\ba{ll}
\ns\ds (I + A^\gamma)u_{tt}=-aA^\g u + bA^\a y,\\
\ns\ds y_{t}= - bA^\a u_t+ k A^{\frac{\beta}{2}} q,\\
\ns\ds \tau q_{t}= -q-A^{\frac{\beta}{2}}y,\\
\ns\ds u(0)=u_0,\quad u_t(0)=v_0, \quad y(0)=y_0, \quad y(0)=q_0, \ea\right.\ee
where $A, a, b, k, \alpha, \beta$ are as in \eqref{1.1} and $\gamma \in [0,1]$.
When the heat conduction follows Fourier's law (i.e., $\tau = 0$), a complete analysis of stability and regularity for the abstract coupled hyperbolic–parabolic system \eqref{1.3} has been carried out in  \cite{Ammar-Bader-Benabdallah 1999, Hao-Liu 2013, Hao-Liu-Yong 2015,Rivera-Racke 1996} for $\gamma = 0$ and in \cite{DellOro-Rivera-Pata 2013, Sare-Liu-Racke 2019, Kuang-Liu-Sare 2021} for $\gamma\in (0,1]$. In the case where $\tau \neq 0$, the model transitions to Cattaneo's law of heat conduction, and stability results are { obtained} in \cite{Deng-Han-Kuang-Zhang 2024, Han-Kuang-Zhang 2023}. For { other} related studies on directly damped or weakly coupled indirectly damped second-order evolution systems, we refer to \cite{ Alabau-Boussouira 2002, Alabau-Cannarsa-Guglielmi 2011,Alabau-Cannarsa-Komornik 2002, Ammari-Shel-Tebou 2021, Ammari-Shel-Tebou  2023, Kuang-Liu-Tebou 2022} and the references therein.

Motivated by the above system, this paper advances the regularity analysis of the indirectly damped coupled hyperbolic system \eqref{1.1}, moving a step closer to Russell’s wishlist. We consider two distinct cases depending on the wave speeds: (i) the case of different wave speeds ( $a$ and $\g$ { do} not simultaneously equal to 1), and (ii) the case of the same wave speeds ($a = \gamma = 1$). In the case of different wave speeds, the region $R(\gamma)$ is partitioned into five disjoint regions $R_1(\g)$–$R_5(\g)$ (see the definition in \eqref{2.4}). We show that the $C_0$-semigroup generated by \eqref{1.1} is analytic in $R_1(\g)$, while it belongs to a Gevrey class with different orders in regions $R_2(\g)$–$R_5(\g)$ (see Theorem \ref{thm:different}). Similarly, in the identical wave speed case, the region $\tilde{R}$ is divided into four disjoint regions $\tilde{R}_1$–$\tilde{R}_4$ (see the definition in \eqref{2.5}), where the semigroup is analytic in $\tilde{R}_1$ and exhibits Gevrey class regularity with appropriate indices in $\tilde{R}_2$–$\tilde{R}_4$ (see Theorem \ref{thm:same}). Compared to the case of different wave speeds, the identical wave speed setting yields improved regularity results. This improvement stems from the fact that, under the same wave speeds, the damping mechanism is transferred more effectively between the two coupled equations.

Our approach begins by establishing a resolvent estimate for the generator of the semigroup associated with \eqref{1.1} in each region via a contradiction argument. We then apply Lemma \ref{lemma:method} to derive regularity results based on the obtained estimate. In addition,  we also analyze the asymptotic expansion of eigenvalues associated with the semigroup generator,  which provides key information on the partition of the regularity region and the relation between the Gevrey class order and the parameters $(\a, \b, \g)$. Based on the analysis of the eigenvalues, we can further prove that the orders of Gevrey class previously acquired are optimal.

The rest of the paper is organized as follows. In Section \ref{sec:pre}, we introduce some technical theorems and lemmas that will be used throughout this paper and state our main results. In Sections \ref{sec:different} and \ref{sec:same}, we establish our main results for the cases where the two coupled equations have different and identical wave speeds, respectively. In Section \ref{sec:Asymptotic}, we explain the symbolic manipulation technique used to determine the asymptotic behavior of the eigenvalues of the system. Finally, we will give some applications of our main theorems in Section \ref{sec:applications}. A summary of the stability and regularity results of system \eqref{1.1} is included in Section \ref{sec:conclusion}.

\ms

\section{Preliminaries and main results}\label{sec:pre}
\setcounter{equation}{0}
\setcounter{theorem}{0}
Define the Hilbert space 
$$\cH:=\cD(A^{\g\over2})\times H\times D(A^{1\over2})\times H$$
with the inner product 
$$\lan U_1,U_2\ran{}\1n_\cH=a\lan A^{\g\over2}u_1,A^{\g\over2}u_2\ran+\lan
v_1,v_2\ran+\lan A^{1\over2}y_1,A^{1\over2}y_2\ran+\lan w_1,w_2\ran,\qq \ U_i=\begin{pmatrix}u_i\\ v_i\\
y_i\\w_i\end{pmatrix}\in\cH,~i=1,2.$$
 By denoting $v=u_t, w=y_t$ and
$U_0=(u_0,v_0,y_0,w_0)^\top$, system (\ref{1.1}) can be written as an
abstract linear evolution equation on the space $\cH$,
\bel{2.1}\left\{\2n\ba{ll}
\ns\ds\frac{dU(t)}{dt}=\cA_{\a,\b,\g}U(t),\qq t\ge0, \\
\ns\ds U(0)=U_0,\ea\right.\ee
where the operator $\cA_{\a,\b,\g}:\cD(\cA_{\a,\b})\subseteq\cH\to\cH$
is defined by
\bel{2.2}\cA_{\a,\b,\g}=\begin{pmatrix}
0 & I & 0 & 0\\
-aA^\g & 0 & 0 & bA^\a \\
0 & 0 & 0 & I \\
0 & -bA^\a & -A & -kA^\b \end{pmatrix},\ee
with the domain
\begin{align*}
    \cD(\cA_{\a,\b,\g})=\{U\in \cH\,|\,\cA_{\a,\b,\g}U\in \cH\}.
\end{align*}
 It is known  that
$\ca_{\a,\b,\g}$ generates a
$C_0$-semigroup $e^{\ca_{\a,\b,\g}t}$ of contractions on $\cH$
(\cite{Hao-Kuang-Liu-Yong 2025}). Then the solution to the evolution equation
(\ref{1.2}) admits the following representation:
$$U(t)=e^{\ca_{\a,\b,\g}t}U_0,\qq t\ge0.$$

\begin{remark}
 It was pointed out in \cite{Hao-Kuang-Liu-Yong 2025}  that the set
$$
 \cD_0:=\cD(A^\g)\times\cD(A^{\a\vee
{\frac{\gamma}{2}}})\times\cD(A)\times\cD(A^{\a\vee\beta\vee{\frac{1}{2}}}) \subseteq \cD(\cA_{\a,\b,\g})
$$ 
is dense in $\cH$,  where $a \vee b=\max\{a, b\}$  for  $a, b\in\mathbb{R} $. For convenience of presentation, we still write the operator $\cA_{\a,\b,\g}$ in the matrix form \eqref{2.2}, since only $U\in \cD_0$ will be used in the subsequent analysis. Another approach was given in \cite{Sare-Liu-Racke 2019, Kuang-Liu-Sare 2021}
by writing the system operator in factorized form.  
\end{remark}

Before going further, let us recall some definitions relevant to the regularity of $C_0$-semigroups.

\ms

\begin{definition}
    Let $e^{\mathcal{A} t}$ be a $C_0$-semigroup on
a Hilbert space $\cH$ with the 
\begin{itemize}
    \item  Semigroup $e^{\cA  t}$ is said to be {\rm analytic} if $e^{\cA t}$ admits an extension $T(z)$ for $z\in\Delta_{\theta}:=\{z\in\mathbb{C}\,|\,|\arg z|<\theta\}$ for some $\theta\in(0,\frac{\pi}{2}]$ such that
    \begin{enumerate}[label=(\roman*)]
        \item $z\mapsto T(z)$ is analytic;
        \item $T(z_1+z_2)=T(z_1)T(z_2)$ for any $z_1,z_2\in \Delta_{\theta}$;
        \item $\underset{\Delta_{\theta'}\ni z\to0}{\lim}\|T(z)U-U\|=0$ for all $U\in \mathcal{H}$ and $0<\theta'<\theta$.
    \end{enumerate}

    \item Semigroup $e^{\cA t}$ is said to be of {\rm Gevrey class $\delta$ (with $\delta > 1$)} if it is infinitely differentiable and for any compact set $\mathcal{K}\subseteq(0,\infty)$, and each $r>0$, there exists a constant $C(\mathcal{K},r)$ such that
    \begin{align}\label{2.3}
        \|\cA^n e^{\cA t}\|\leq Cr^n(n!)^{\delta},\qq t\in\mathcal{K},n\ge0.
    \end{align}

    \item  Semigroup $e^{\cA t}$ is said to be {\rm differentiable} if
for any $x\in\cH$, $t\mapsto e^{\cA t}x$ is differentiable on
$(0,\infty)$.
\end{itemize}
\end{definition}
\noindent Note that an analytic semigroup satisfies \eqref{2.3} with $\delta=1$. A semigroup which is analytic or is of Gevrey class $\delta>1$ is always infinitely differentiable.

Recall the definitions of $R(\g)$ and $\tilde{R}$ in \eqref{eq:R} and \eqref{eq:tildeR}. As mentioned in the introduction, if $(\alpha,\beta,\gamma)\notin R(\g)$ or $(\alpha,\beta)\notin \tilde{R}$, the system admits a sequence of eigenvalues with a vertical asymptote, thereby violating the spectral criterion for differentiability of
$e^{\ca_{\a,\b,\g}t}$. Therefore, it suffices to study the regularity of system \eqref{2.1} only within the regions $R(\g)$ and $ \tilde{R}$, respectively.

\begin{figure}[h]
\centering
\centering
\begin{subfigure}{0.19\textwidth}
\centering
\includegraphics[width=\textwidth]{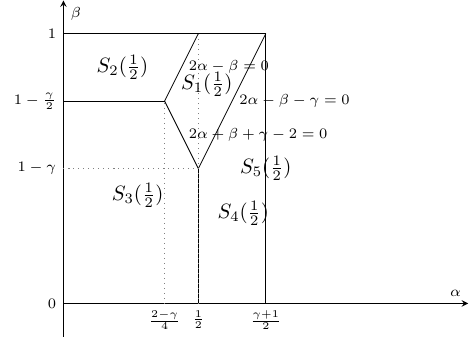}
\end{subfigure}
\begin{subfigure}{0.19\textwidth}
\centering
\includegraphics[width=\textwidth]{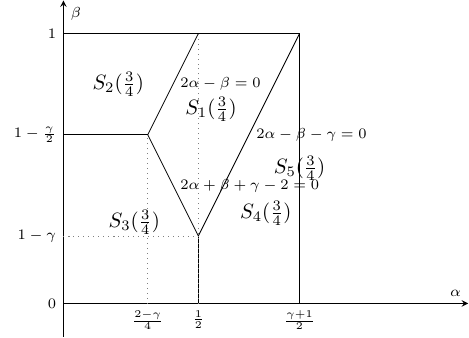}
\end{subfigure}
\centering
\begin{subfigure}{0.19\textwidth}
\centering
\includegraphics[width=\textwidth]{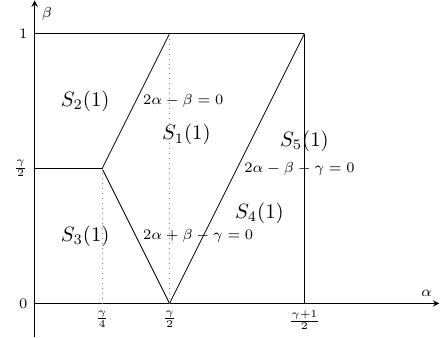}
\end{subfigure}
\begin{subfigure}{0.19\textwidth}
\centering
\includegraphics[width=\textwidth]{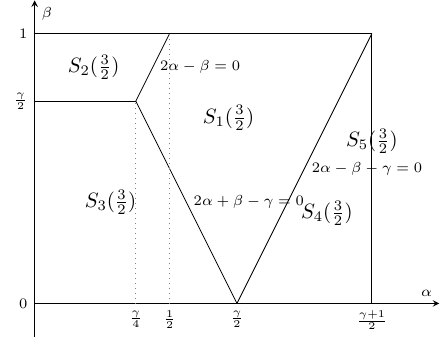}
\end{subfigure}
\begin{subfigure}{0.19\textwidth}
\centering
\includegraphics[width=\textwidth]{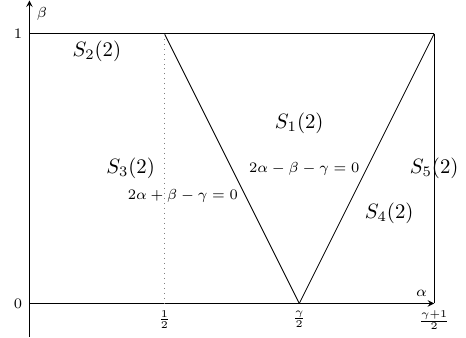}
\end{subfigure}\\
\begin{subfigure}{0.19\textwidth}
\centering
\includegraphics[width=\textwidth]{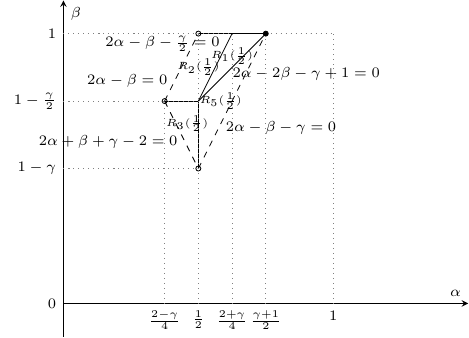}
\caption{$\gamma=\frac{1}{2}$}
\end{subfigure}
\begin{subfigure}{0.19\textwidth}
\centering
\includegraphics[width=\textwidth]{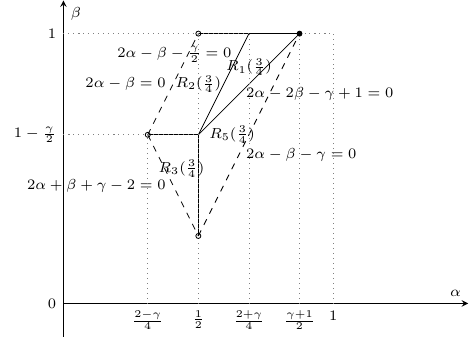}
\caption{$\gamma=\frac{3}{4}$}
\end{subfigure}
\centering
\begin{subfigure}{0.19\textwidth}
\centering
\includegraphics[width=\textwidth]{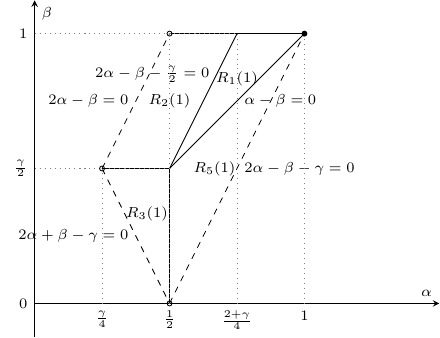}
\caption{$\gamma=1$, $a\ne 1$}
\end{subfigure}
\begin{subfigure}{0.19\textwidth}
\centering
\includegraphics[width=\textwidth]{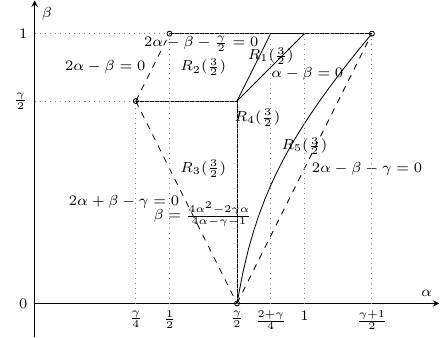}
\caption{$\gamma=\frac{3}{2}$}
\end{subfigure}
\begin{subfigure}{0.19\textwidth}
\centering
\includegraphics[width=\textwidth]{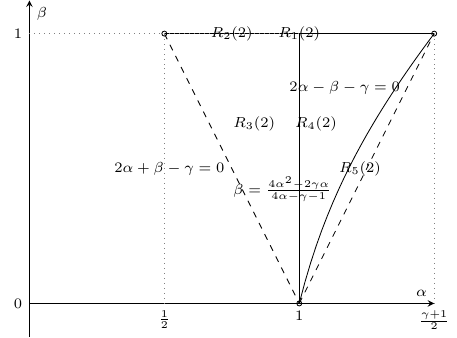}
\caption{$\gamma=2$}
\end{subfigure}
\caption{Regions of stability (the first row) and regularity (the second row). The figures of regions of stability are from \cite{Hao-Kuang-Liu-Yong 2025}.}
\label{fig:r0}
\end{figure}
When $(a,\gamma)\ne (1, 1)$, we partition the 3-d region $R(\gamma)$ into five disjoint subregions, see Figure \ref{fig:r0}: 
\begin{align}\label{2.4}
\begin{array}{lll}
    R_1(\gamma)&:=
    \Big\{(\a,\b,\g)\in R(\g) \bigm|4\alpha-2\beta-\gamma \ge 0, 2\alpha-2\beta+((1-\gamma)\vee0)\le 0 \Big\}, \quad  \frac{1}{2} \le \g \le 2\\   
    R_2(\gamma)&:=\Big\{(\a,\b,\g)\in R(\g) \bigm|2\alpha-\beta>0,4\alpha-2\beta-\gamma<0, \beta  \ge (1-\frac{\gamma}{2})\vee\frac{\g}{2} \Big\}, \quad \quad \frac{1}{2} \le \g \le 2\\
    R_3(\gamma)&:=\Big\{(\a,\b,\g)\in R(\g) \bigm|2\alpha+\beta-(\gamma\vee (2-\gamma))>0,\alpha  \le \frac{1\vee \g}{2}, \beta<(1-\frac{\gamma}{2})\vee\frac{\g}{2} \Big\}, \quad \frac{1}{2} \le \g \le 2\\
    R_4(\gamma)&:=\Big\{(\a,\b,\g)\in R(\g) \bigm|\a>\beta,\a>\frac{\g}{2},\beta>\frac{4\alpha^2-2\alpha\gamma}{4\alpha-\gamma-1}\Big\}, \quad 1 < \gamma\le 2\\
    R_5(\gamma)&:=\left\{
\begin{array}{l}
\big\{(\a,\b,\g)\in R(\g) \bigm| 2\alpha-\beta-\gamma<0,\frac{1}{2}<\alpha<\frac{\gamma+1}{2},2\alpha-2\beta-\gamma+1>0\big\}, \quad \frac{1}{2}\le\g\le 1\\
 \big\{(\a,\b,\g)\in R(\g) \bigm| 2\alpha-\beta - \g <0,\beta\le\frac{4\alpha^2-2\gamma\alpha}{4\alpha-\gamma-1}\big\}, \quad 1 < \g\le 2
\end{array}
\right.
\end{array}
\end{align}
 Note that the edge $\{(\frac{\g+1}{2}, 1, \g),\; | \;\g\in [\frac{1}{2},1]\}$ in $R(\g)$ is contained in $R_1(\g)$.

When $(a,\gamma)= (1, 1)$, we partition the 2-d region $\tilde{R}$ 
into the following four disjoint subregions, see Figure \ref{fig:regularity-same}:
\begin{align}\label{2.5}
\begin{array}{lll}
      \tilde{R}_1&:=\Big\{(\a,\b)\in\widetilde{R} \bigm| \alpha\leq \beta,2\alpha-\beta\geq \frac{1}{2}\Big\};\\
     \tilde{R}_2&:=\Big\{(\a,\b)\in\widetilde{R} \bigm| \alpha\leq \beta,0<2\alpha-\beta<\frac{1}{2}\Big\};\\
      \tilde{R}_3&:=\Big\{(\a,\b)\in \widetilde{R} \bigm| \alpha>\beta,\alpha\leq \frac{1}{2}\Big\};\\
       \tilde{R}_4&:=\Big\{(\a,\b)\in\widetilde{R} \bigm| \alpha>\beta,\a>\frac{1}{2},2\alpha-\beta<1\Big\};
\end{array}
\end{align}

\begin{figure}[h]
\centering
\begin{subfigure}{0.32\textwidth} 
\centering
\includegraphics[width=\textwidth]{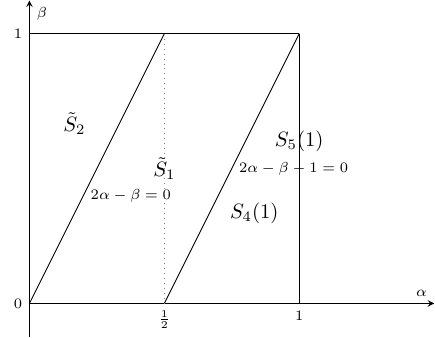} 
\caption{$\tilde{S}$}
\label{fig:stability-same-wave}
\end{subfigure}
\begin{subfigure}{0.32\textwidth} 
\centering
\includegraphics[width=\textwidth]{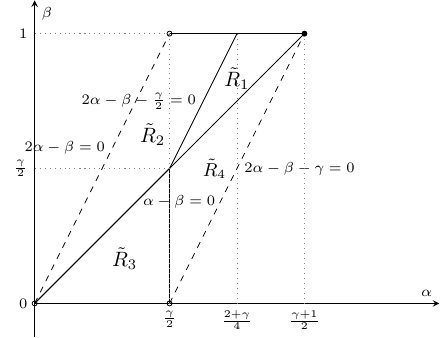} 
\caption{$\tilde{R}$}
\label{fig:regularity-same-wave}
\end{subfigure}
\begin{subfigure}{0.32\textwidth}
\centering
\includegraphics[width=\textwidth]{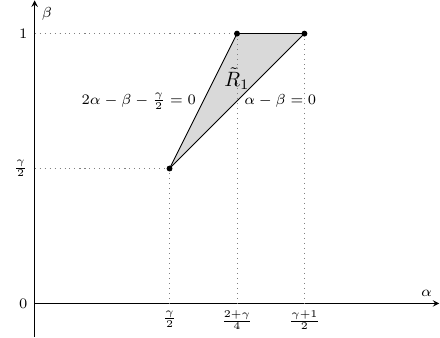}
\caption{$\tilde{R}_1$}
\label{fig:tilde-R1}
\end{subfigure}\\
\begin{subfigure}{0.32\textwidth}
\centering
\includegraphics[width=\textwidth]{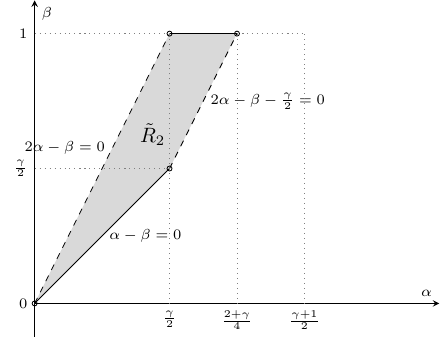}
\caption{$\tilde{R}_2$}
\end{subfigure}
\label{fig:tilde-R2}
\begin{subfigure}{0.32\textwidth}
\centering
\includegraphics[width=\textwidth]{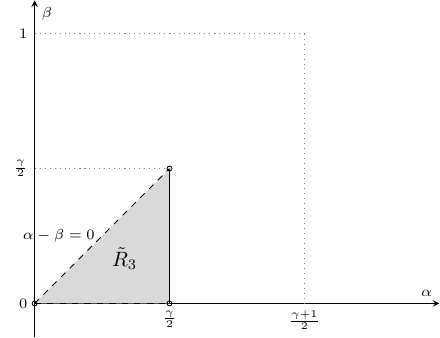}
\caption{$\tilde{R}_3$}
\end{subfigure}
\label{fig:tilde-R3}
\begin{subfigure}{0.32\textwidth}
\centering
\includegraphics[width=\textwidth]{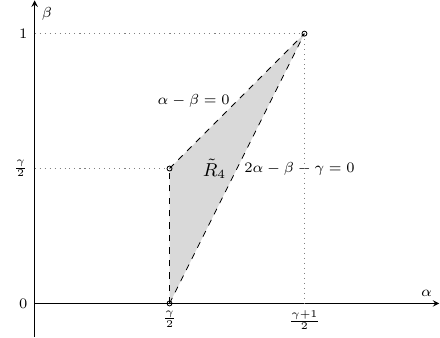}
\caption{$\tilde{R}_4$}
\label{fig:tilde-R4}
\end{subfigure}
\caption{Regions of stability (\Cref{fig:stability-same-wave}) and regularity (\Cref{fig:regularity-same-wave}-\Cref{fig:tilde-R4}) when $\gamma=1$ and the system of the two coupled equations are of the same speed ($a=1$). \Cref{fig:stability-same-wave} is from \cite{Hao-Kuang-Liu-Yong 2025}.}
\label{fig:regularity-same}
\end{figure}

\begin{remark}
\label[remark]{remark:eigenvalues}
    The above partition of the regularity region is based on the asymptotic expression of the eigenvalues of the operator $\cA_{\a,\b,\g}$, which will be presented in Section 5. This is a crucial step in our complete regularity analysis. Since the associated characteristic equation is a quartic polynomial equation, its roots can be calculated for fixed $(\a,\b,\g)$ values in terms of the eigenvalues $\mu_n$ of the operator $A$. However, for all $(\a,\b,\g)\in E$, this becomes a great challenge. We must identify the real and imaginary parts of the eigenvalues as functions of $\mu_n$ and $(\a,\b,\g)$, and the region where the form of the eigenvalues remains the same. In the current case, one of the regions even has a curved surface as its boundary, whose form is unexpected. Therefore, it is almost impossible to do this by hand computation. In response, Z. Kuang proposed a symbolic manipulation algorithm, which has been successfully applied in \cite{Deng-Han-Kuang-Zhang 2024,  Han-Kuang-Zhang 2023, Hao-Kuang-Liu-Yong 2025, Hao-Liu-Yong 2015, Kuang-Liu-Sare 2021, Kuang-Liu-Tebou  2022}, that enables  us to partition the regularity region  and obtain the asymptotic expressions of the eigenvalues in each subregion that provide candidates for the order of the resolvent operator $(i\lambda  - \cA_{\a,\b,\g})^{-1}$ on the imaginary axis. We then proceed to show that these candidates are actually the orders we are looking for by
    the frequency domain method. 
\end{remark}

Our main result for the case of different wave speeds is as follows.
\begin{theorem}\label{thm:different}
  Let $(a,\gamma)\neq (1, 1)$. The semigroup $e^{\cA_{\alpha,\beta,\gamma} t}$ is analytic in $R_1(\gamma)$ and is
   of Gevrey class $\delta>\frac{1}{\mu}$ for 
   \begin{equation*}
\mu=\left\{
\begin{aligned}
&\frac{2(2\a-\b)}{\g} & \quad \text{ in } R_2(\gamma),\\
&\frac{2(2\alpha+\beta-(\gamma\vee (2-\gamma)))}{\g} & \quad \text{ in } R_3(\gamma),\\
&\frac{\beta}{\a} & \quad \text{ in } R_4(\gamma),\\
&\frac{2(-2\a+\b+\g)}{-2\a+\g+1}&\quad \text{ in } R_5(\gamma).\\
\end{aligned}
\right.
\end{equation*} 
\end{theorem}

The following theorem presents the result for the identical wave speed case.
\begin{theorem}\label{thm:same}
  Let $(a,\gamma)= (1, 1)$. The  semigroup $e^{\cA_{\alpha,\beta,1} t}$ is analytic in $\tilde{R}_1$ and is
   of Gevrey class $\delta>\frac{1}{\mu}$ for 
   \begin{equation*}
\mu=\left\{
\begin{aligned}
&2(2\alpha-\beta) & \quad \text{ in } \tilde{R}_2,\\
&2\beta & \quad \text{ in } \tilde{R}_3,\\
&\frac{-2\alpha+\beta+1}{1-\alpha}&\quad \text{ in } \tilde{R}_4.\\
\end{aligned}
\right.
\end{equation*} 
\end{theorem}

\begin{remark}\label{7.19}
Note that the identical wave speed case yields better stability and regularity properties than the different wave speed case, due to a more efficient transfer of damping between the coupled equations. The improved stability result is reflected in the expansion of the exponentially stable region under the identical wave speed setting \cite{Hao-Kuang-Liu-Yong 2025}, while the improved regularity is observed from the following two aspects. 

First, the regularity region is enlarged. As illustrated in Figure~\ref{Comparison}, when $a = \gamma = 1 $, the region $R_3(\gamma)$ expands to include $\tilde{R}_2  (\mbox{with} \; \beta \leq \frac{1}{2})$ and  $\tilde{R}_3 $. In addition, the regions $\tilde{R}_1$, $\tilde{R}_2 (\mbox{with} \; \beta>\frac{1}{2})$ and $\tilde{R}_4$ correspond to ${R}_1(\gamma)$, $R_2(\gamma)$ and $R_5(\gamma)$, respectively. 
Second, the Gevrey class order is improved. In  $\tilde{R}_3$, where  $\alpha,\beta \leq \frac{1}{2}$, the regularity improves from Gevrey order $ \delta > \frac{1}{2(2\alpha + \beta - 1)}$ to $\delta > \frac{1}{2(2\alpha - \beta)}$ in $\tilde{R}_2$  and $\delta > \frac{1}{2\beta}$  in $\tilde{R}_3$.
\end{remark}

\begin{figure}[h]
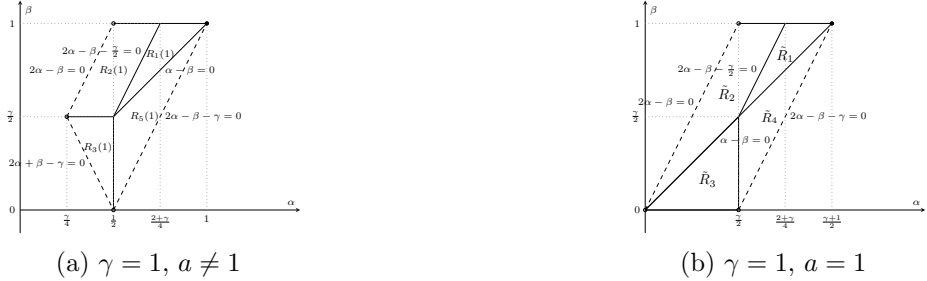

\centering
\begin{subfigure}{0.49\textwidth}
\centering
\includegraphics[width=0.5\textwidth]{./support/R-all-id-gamma-1.pdf}
\caption{$\gamma=1$, $a\ne 1$}
\end{subfigure}~
\begin{subfigure}{0.49\textwidth}
\centering
\includegraphics[width=0.5\textwidth]{./support/tilde-R-all-id-gamma-1.pdf}
\caption{$\gamma=1$, $a=1$}
\end{subfigure}~
\caption{Comparison between different wave speeds and identical wave speeds.}
\label{Comparison}
\end{figure}

At the end of this section, we recall some key lemmas that will be used in our regularity analysis.
The following lemma characterizes the spectral location and resolvent growth on the imaginary axis for infinitesimal generators of analytic and Gevrey class semigroups. For details, see \cite[Chapter 1]{Liu-Zheng 1999}, \cite{Pazy 1983} and \cite[Theorem 4]{Taylor 1989}.
\begin{lemma}\label[lemma]{lemma:method}  
Let $\cA:\cD(\cA_{\a,\b,\g})\subseteq\cH\to\cH$ generate a
$C_0$-semigroup $e^{\cA t}$ on $\cH$ such that
\begin{align*}
    \|e^{\cA t}\|\le M,\qq t\ge0, 
\end{align*}
for some $M\ge1$ and
\begin{align*}
    i\l\in\rho(\cA),\qq\l\in\dbR.
\end{align*}
Then the following hold:

\ms

{\rm(i)} Semigroup $e^{\cA t}$ is analytic if and only if for some
$a\in\dbR$ and $b,C>0$ such that
\begin{align*}
    \rho(\cA)\supseteq\Si(a,b){:=}\Big\{\l\in\dbC\bigm|\Re\l>a-b|\Im\l|\Big\},
\end{align*}
and
\begin{align*}
    \|(i\l-\cA)^{-1}\|\le{C\over1+|\l|},\qq\l\in\Si(a,b).
\end{align*}
This is the case if and only if
\bel{2.6}\limsup_{\l\in\dbR,\,|\l|\to\infty}|\l|\,\|(i\l-\cA)^{-1}\|<\infty.\ee

{\rm(ii)} Semigroup $e^{\cA t}$ is of {\it Gevrey} class $\d>1$ if
and only if for any $b,\t>0$, there are constants $a\in\dbR$ and
$C>0$ depending on $b,\t,\d$ such that
\begin{align*}
    \rho(\cA)\supseteq\Si_b(\d):=\Big\{\l\in\dbC\bigm|\Re\l>a-b|\Im\l|^{1\over\d}\Big\},
\end{align*}
and
\begin{align*}
    \|(i\l-\cA)^{-1}\|\le
C\(e^{-\t\Re\l}+1\),\qq\l\in\Si_b(\d).
\end{align*}
This is the case, in particular, if for some $\m\in(\d^{-1},1)$,
\bel{2.7}\limsup_{\l\in\dbR,\,|\l|\to\infty}|\l|^\m\|(i\l-\cA)^{-1}\|<\infty.
\ee

{\rm(iii)} Semigroup $e^{\cA t}$ is differentiable if and only if
for any $b>0$, there are constants $a_b\in\dbR$ and $C_b>0$ such
that
\begin{align*}
\rho(\cA)\supseteq\Si_b:=\Big\{\l\in\dbC\bigm|\Re\l>a_b-b\log|\Im\l|\Big\},\end{align*}
and
\begin{align*}
\|(i\l-\cA)^{-1}\|\le
C_b|\Im\l|,\qq\forall\l\in\Si_b,~\Re\l\le0.\end{align*}
This is the case, in particular, if
\bel{2.8}\limsup_{\l\in\dbR,\,|\l|\to\infty}\log|\l|\|(i\l-\cA)^{-1}\|=0.\ee

\end{lemma}

We also recall the interpolation technique that will be used frequently in later sections.
\begin{lemma}
    Let $A:\cD(A)\subseteq H$ be self-adjoint and
positive definite. Then
\begin{align*}
    \|A^px\|\le\|A^qx\|^{p-r\over
q-r}\|A^rx\|^{q-p\over q-r},\qq 0\le r\le p\le
q,~x\in\cD(A^q).
\end{align*}
\end{lemma}

\section{Different  wave speeds}\label{sec:different}
\setcounter{equation}{0}
\setcounter{theorem}{0}
\rm 
This section is devoted to the proof of Theorem \ref{thm:different}. We shall use a contradiction argument to estimate the order $\mu$ of the resolvent operator on the imaginary axis. For simplicity of presentation, we will take $b=k=1$ throughout the rest of the paper.

\begin{lemma}\label{lemma:important}
Given $(\a,\b,\g)\in R(\g)$ and $\m\in (0,1]$. Suppose that the following is not true:
$$\limsup_{\l\in\dbR,\,|\l|\to\infty}|\l|^{\m}\|(i\l-\cA_{\a,\b,\g})^{-1}\|
<\infty.$$ Then there exists a sequence
$(U_n:=(u_n,v_n,y_n,w_n)^T)_{n\ge1}\subseteq \cD_0$ with
\begin{align}\label{5.7}
    \begin{array}{lll}
        &&\lim_{n\to\infty}|\l_n|=\infty,  \\
       \noalign{\medskip}   
         && \|U_n\|_\cH^2=a\|A^{\g\over2}u_n\|^2 + \|v_n\|^2 + \|A^{1\over2}y_n\|^2 + \|w_n\|^2=1,\q
n\ge1.
    \end{array}
\end{align} 
Moreover, 
\begin{align}
    \label{5.8a}
	& a\|A^{\frac{\gamma}{2}}u_n\|^2+\|w_n\|^2=\frac{1}{2}+o(1).\\
	\label{5.8b}
	&\|A^{\frac{1}{2}}y_n\|^2+\|v_n\|^2=\frac{1}{2}+o(1).\\
    \label{5.8c}
    &\|\lambda_n^{-\frac{\mu}{2}}A^{\frac{\beta}{2}}w_n\|=o(1).\\
    \label{5.8d}
    &A^{\frac{\gamma}{2}}u_n=-i\lambda_n^{-1}A^{\frac{\gamma}{2}}v_n+o(1).\\
    \label{5.8e}
&A^{\frac{1}{2}}y_n=-i\lambda_n^{-1}A^{\frac{1}{2}}w_n+o(1).
\end{align}
\end{lemma}
\begin{proof}
By the assumption, there exists a sequence
$\{(\l_n,U_n)\bigm|n\ge1\}\subseteq\dbR\times\cD_0$ with
$U_n:=(u_n,v_n,y_n,w_n)^T$ and \eqref{5.7} holds.
Moreover, we have
\bel{3.28}\lim_{n\to\infty}|\l_n|^{-\m}\|(i\l_n-\cA_{\a,\b,\g})
U_n\|_\cH=0,\ee
i.e.,
\begin{subequations}
	\bel{3.30a}\3n\3n\3n\3n\3n\3n\3n\3n\3n\3n
	i\l_n^{-\m+1}A^{\g\over2}u_n-\l_n^{-\m}A^{\g\over 2}v_n=o(1),\ee
	\bel{3.30b}\3n
	i\l_n^{-\m+1}v_n+\l_n^{-\m}aA^\g u_n-\l_n^{-\m}A^{\a}w_n=o(1),\ee
	\bel{3.30c}\3n\3n\3n\3n\3n\3n\3n\3n\3n\3n
	i\l_n^{-\m+1}A^{1\over2}y_n-\l_n^{-\m}A^{1\over 2}w_n=o(1),\ee
	\bel{3.30d}
	i\l_n^{-\m+1}w_n + \l_n^{-\m}A y_n + \l_n^{-\m}A^\a v_n + \l_n^{-\m}A^\b w_n=o(1).\ee
\end{subequations}

Note that \eqref{5.8d} and \eqref{5.8e} are clear from \eqref{3.30a} and \eqref{3.30c}, respectively.     Taking the inner product of \eqref{3.30a} with $aA^{\frac{\gamma}{2}}u_n$, \eqref{3.30b} with $v_n$, \eqref{3.30c} with $A^{\frac{1}{2}}y_n$, and \eqref{3.30d} with $w_n$ gives
     \begin{align}
         &ia\lambda_n^{-\mu+1}\|A^{\frac{\gamma}{2}}u_n\|^2-a\l_n^{-\mu}\lan  A^{\gamma}v_n, u_n\ran=o(1).\label{5.81}\\
        &i\lambda_n^{-\mu+1}\|v_n\|^2+a\l_n^{-\mu}\lan  A^{\gamma}u_n, v_n\ran- \l_n^{-\mu} \lan A^{\alpha}w_n, v_n\ran=o(1).\label{5.82}\\
         &i\lambda_n^{-\mu+1}\|A^{\frac{1}{2}}y_n\|^2-\l_n^{-\mu} \lan Aw_n, y_n\ran=o(1).\label{5.83}\\
         &i\lambda_n^{-\mu+1}\|w_n\|^2+ \l_n^{-\mu}\lan Ay_n, w_n\ran+ \l_n^{-\mu} \lan A^{\alpha}v_n, w_n\ran+\lambda_n^{-\mu}\|A^{\frac{\beta}{2}}w_n\|^2=o(1).\label{5.84}
     \end{align}
    Adding \eqref{5.81}-\eqref{5.84} yields
       \begin{align}\label{5.85}
       \begin{array}{lll}
        &\lambda_n^{-\mu}\|A^{\frac{\beta}{2}}w_n\|^2+i\lambda_n^{-\mu+1}(a\|A^{\frac{\gamma}{2}}u_n\|^2+\|v_n\|^2+\|A^{\frac{1}{2}}y_n\|^2+\|w_n\|^2)\\  \noalign{\medskip}   
     & -2i\lambda_n^{-\mu}(a \Im\lan  A^{\gamma}v_n,u_n\ran+\Im \lan A^{\alpha}v_n, w_n\ran+\Im \lan Ay_n, w_n\ran)=o(1).
       \end{array}    
     \end{align}
    Then \eqref{5.8c} follows by taking the real part of \eqref{5.85}. 
    On the other hand, taking conjugate of \eqref{5.82} and \eqref{5.83}, then multiplying both by (-1) respectively, we get
    \begin{align}
        &&i\lambda_n^{-\mu+1}\|v_n\|^2-a\l_n^{-\mu}\lan  A^{\gamma}v_n, u_n\ran+ \l_n^{-\mu} \lan A^{\alpha}v_n, w_n\ran=o(1).\label{5.86}\\
         &&i\lambda_n^{-\mu+1}\|A^{\frac{1}{2}}y_n\|^2+\l_n^{-\mu} \lan Ay_n, w_n\ran=o(1).\label{5.87}
    \end{align}
    Then combining \eqref{5.8c}, \eqref{5.81}, \eqref{5.84}, \eqref{5.86} and \eqref{5.87}, we obtain
    \begin{align}\label{5.88}
        i\lambda_n^{-\mu+1}(\|A^{\frac{\gamma}{2}}u_n\|^2+\|w_n\|^2-\|v_n\|^2-\|A^{\frac{1}{2}}y_n\|^2)=o(1).
    \end{align}
    Therefore, \eqref{5.8a} and \eqref{5.8b} follow from \eqref{5.7}, \eqref{5.88} and the fact $\mu\leq 1$. The proof is complete.
\end{proof}

To facilitate the proof, we decompose Theorem~\ref{thm:different} into five propositions, presented as Proposition~\ref{main1} and Propositions \ref{main2}–\ref{main5}. We next establish them one by one.
\begin{proposition}\label{main1}
  The  semigroup $e^{\cA t}$ is analytic in 
 $$R_1(\gamma)=\Big\{(\a,\b,\g)\in R(\g) \bigm|4\alpha-2\beta-\gamma \ge 0, \; 2\alpha-2\beta+((1-\gamma)\vee0)\le 0 \Big\},\;\;  {1\over2}\le\gamma \le 2.$$
\end{proposition}
\begin{figure}[t]
\centering
\begin{subfigure}{0.19\textwidth}
\centering
\includegraphics[width=\textwidth]{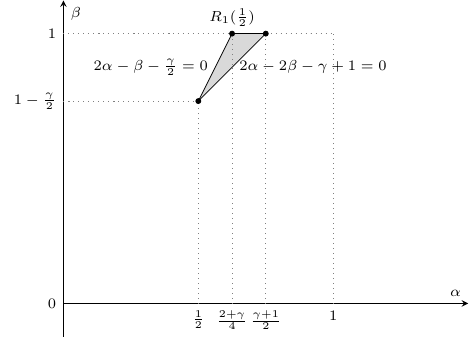}
\caption{$\gamma=\frac{1}{2}$}
\end{subfigure}~
\begin{subfigure}{0.19\textwidth}
\centering
\includegraphics[width=\textwidth]{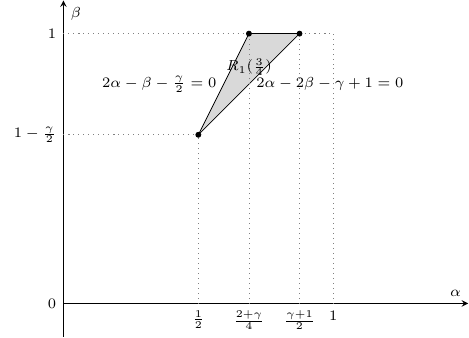}
\caption{$\gamma=\frac{3}{4}$}
\end{subfigure}
\centering
\begin{subfigure}{0.19\textwidth}
\centering
\includegraphics[width=\textwidth]{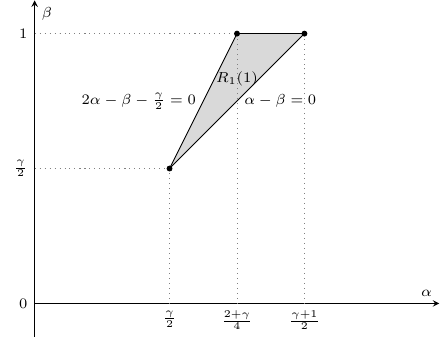}
\caption{$\gamma=1$, $a\ne 1$}
\end{subfigure}~
\begin{subfigure}{0.19\textwidth}
\centering
\includegraphics[width=\textwidth]{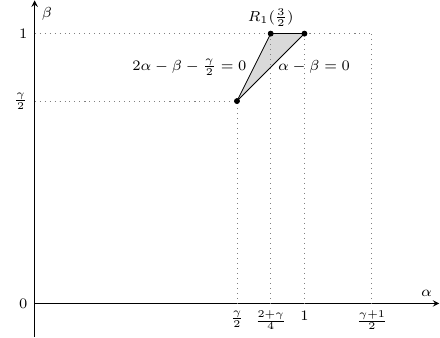}
\caption{$\gamma=\frac{3}{2}$}
\end{subfigure}~
\begin{subfigure}{0.19\textwidth}
\centering
\includegraphics[width=\textwidth]{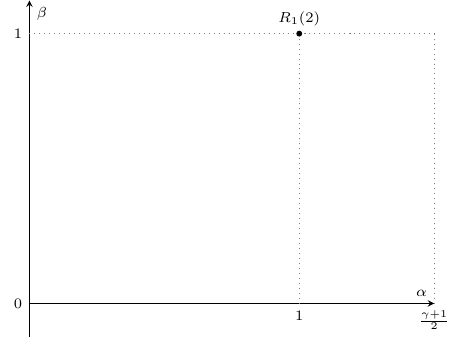}
\caption{$\gamma=2$}
\end{subfigure}
\caption{Visualization of $R_1(\gamma)$ when $\gamma=$  $\frac{1}{2},\frac{3}{4},1,\frac{3}{2}$, and $2$.}
\label{fig:r1}
\end{figure}

\begin{proof}
   By Lemma \ref{lemma:method}, it suffices to prove \eqref{2.6}. Let $(\a,\b,\g)\in R_1(\gamma)$, as shown in Figure~\ref{fig:r1}. By Lemma \ref{lemma:important}, if \eqref{2.6} fails, we have \eqref{5.7}-\eqref{5.8e}.
Since $\alpha - \beta \le 0$ in $R_1(\gamma)$, applying $A^{\alpha - \beta}$ to (\ref{3.30d}) and adding the result to (\ref{3.30b}) yields 
\be\label{3.40}
iv_n +a \l_n^{-1} A^\g u_n + iA^{\a-\b}w_n + \l_n^{-1}A^{2\a-\b}v_n + \l_n^{-1}A^{1 + \a-\b}y_n = o(1).
\ee
Taking the inner product of (\ref{3.40}) with $v_n$ provides
\be\label{3.41}
i\|v_n\|^2 + a\lan \l_n^{-1} A^\g u_n,v_n \ran + i\lan  A^{\a-\b}w_n,v_n \ran +  \|\l_n^{-\frac{1}{2}}A^{\a-\frac{\b}{2}}v_n\|^2 + \lan \l_n^{-1} A^{1+\a-\b}y_n,v_n \ran = o(1).
\ee 
The first and third terms of (\ref{3.41}) are bounded due to \eqref{5.7}. Moreover, by \eqref{5.8d},
$$ \lan \l_n^{-1} A^\g u_n,v_n \ran = \lan A^\frac{\g}{2}u_n, \l_n^{-1}A^\frac{\g}{2}v_n   \ran\leq \|A^\frac{\g}{2}u_n\|\|\l_n^{-1}A^\frac{\g}{2}v_n \| = O(1). $$
Since $2\alpha-2\beta+((1-\gamma)\vee0)\le0$, then again by \eqref{5.8d} and \eqref{5.7},
\begin{equation}\label{6.18}
 \lan \l_n^{-1} A^{1+\a-\b}y_n,v_n \ran =\left\{
\begin{aligned}
\lan  A^{\frac{1}{2}}y_n,\l_n^{-1}A^{\frac{1}{2}+\a-\b} v_n \ran = O(1), & \quad \frac{1}{2}\le\g<1,\\
 \lan  A^{1+\a-\b-\frac{\g}{2}}y_n,\l_n^{-1}A^\frac{\g}{2} v_n \ran = O(1), & \quad 1<\g\leq 2.
\end{aligned}
\right.
\end{equation}
Hence, we can conclude from \eqref{3.41} that $ \|\l_n^{-\frac{1}{2}}A^{\a-\frac{\b}{2}}v_n\|=O(1). $ Then by \eqref{5.8c},
\begin{align}\label{5.89}
     \lan \l_n^{-1} A^\a w_n,v_n \ran = \lan \l_n^{-\frac{1}{2}} A^\frac{\b}{2} w_n, \l_n^{-\frac{1}{2}}A^{\a-\frac{\b}{2}}v_n \ran = o(1).
\end{align}

Next, note that $\frac{1}{2}\le \frac{\g}{2}\vee (1-\frac{\g}{2})\le \b$. Acting $A^{\frac{1}{2}-\b}$ on (\ref{3.30d}) gives
\be\label{3.42} 
	iA^{\frac{1}{2}-\b}w_n + \l_n^{-1}A^{\frac{3}{2}-\b} y_n + \l_n^{-1}A^{\frac{1}{2}-\b+\a} v_n + \l_n^{-1}A^\frac{1}{2} w_n=o(1). 
\ee
Since $\alpha-\beta+((\frac{1}{2}-\frac{\g}{2})\vee0)\le 0$, we conclude from \eqref{5.8d} and \eqref{5.8e} that the first, third and fourth term of (\ref{3.42}) are bounded, which implies   $\|\l_n^{-1}A^{\frac{3}{2}-\b} y_n\| = O(1)$. Then 
\be\label{3.43} \| \l_n^{-\frac{1}{2}}A^{1 - \frac{\b}{2}} y_n \|^2 = \lan \l_n^{-1}A^{\frac{3}{2}-\b} y_n, A^\frac{1}{2}y_n \ran = O(1). 
\ee
It follows again from \eqref{5.8c} that
\be\label{3.44} \lan \l_n^{-1} Ay_n,w_n \ran = \lan\l_n^{-\frac{1}{2}} A^{1-\frac{\b}{2}}y_n,\l_n^{-\frac{1}{2}}A^\frac{\b}{2} w_n \ran = o(1). 
\ee
Therefore, one can deduce from \eqref{5.83} and \eqref{3.44} that 
\be\label{5.71}  \|A^\frac{1}{2}y_n\|=o(1). \ee 
Then by \eqref{6.18}, we immediately have 
\begin{align}\label{5.73}
     \lan \l_n^{-1} A^{1+\a-\b}y_n,v_n \ran =o(1).
\end{align}
Moreover, by \eqref{5.8c}, \eqref{5.84}, \eqref{5.89}, and \eqref{3.44}, we obtain
\begin{equation}\label{3.45}
\|w_n\| = o(1).
\end{equation}
In addition, \eqref{5.82} together with \eqref{5.89} implies that
\begin{equation}\label{5.72}
i\|v_n\|^2 + a\langle \lambda_n^{-1} A^\gamma u_n,v_n \rangle = o(1).
\end{equation}
We now substitute \eqref{5.73}, \eqref{3.45}, and \eqref{5.72} into \eqref{3.41}, which leads to
\begin{equation}\label{3.46}
 \|\l_n^{-\frac{1}{2}}A^{\a-\frac{\b}{2}}v_n\|=o(1).
\end{equation}
This, in turn, yields
\be\label{3.47} 
 \lan \l_n^{-1} A^\g u_n,v_n \ran = \lan \l_n^{-\frac{1}{2}} A^{\g-\a+\frac{\b}{2}} u_n,\l_n^{-\frac{1}{2}}A^{\a-\frac{\b}{2}}v_n \ran = o(1)
\ee
provided that
\begin{align}\label{5.74}
    \|\l_n^{-\frac{1}{2}} A^{\g-\a+\frac{\b}{2}} u_n\|=O(1).
\end{align} 

Finally, we show \eqref{5.74} holds. In fact, since $\frac{1}{2}\leq \alpha$, acting $A^{\frac{1}{2}-\alpha}$ on (\ref{3.30b}) yields
$$iA^{\frac{1}{2}-\alpha}v_n+a\l_n^{-1}A^{\gamma+\frac{1}{2}-\alpha} u_n-\l_n^{-1}A^{\frac{1}{2}}w_n=o(1).$$
Note that (\ref{5.8e}) and (\ref{5.71}) implies $\l_n^{-1}A^{\frac{1}{2}}w_n=o(1)$. This together with the above equation shows
\begin{align}\label{5.3}
   \| \l_n^{-1}A^{\gamma+\frac{1}{2}-\alpha} u_n\|=O(1).
\end{align}
Since $\g-2\a+\b \le \frac{\g}{2}$ in $R_1(\gamma)$, by taking the inner product of (\ref{3.40}) with 
$A^{\g-2\a+\b}u_n$ in $H$, we obtain
$$ i\lan v_n + A^{\a-\b}w_n,A^{\g-2\a+\b} u_n \ran  + \lan \l_n^{-1} A^\frac{\g}{2} v_n,A^\frac{\g}{2}u_n \ran + a\|\l_n^{-\frac{1}{2}} A^{\g-\a+\frac{\b}{2}} u_n\|^2+ \lan A^\frac{1}{2}y_n,\l_n^{-1}A^{\g+\frac{1}{2}-\a}u_n \ran = o(1).
$$
The boundedness of the first two inner product terms is immediate, and the third term is bounded by \eqref{5.3}. Therefore, \eqref{5.74} holds. Combining this with \eqref{5.71}, \eqref{5.72}, and \eqref{3.47} yields $\|v_n\|^2+\|A^{1\over2}y_n\|^2=o(1)$, which contradicts to \eqref{5.8b}. 
\end{proof}

Recall that $$R_2(\gamma)=\Big\{(\a,\b,\g)\in  R(\g) \bigm|2\alpha-\beta>0,4\alpha-2\beta-\gamma<0, \beta  \ge (1-\frac{\gamma}{2})\vee\frac{\g}{2} \Big\},\;\; {1\over2}\le\gamma \le 2.$$
Before proving Theorem \ref{main2}, we first prove the following lemma which will be used later. Note that  if \eqref{2.7} is not true, by Lemma \ref{lemma:important}, we have \eqref{5.7}-\eqref{5.8e} hold. Then we can deduce the following estimates using interpolation.
\begin{lemma}
  Let $(\a,\b,\g)\in R_2(\gamma)$, and $\m=\frac{2(2\a-\b)}{\g}$. If \eqref{2.7} is not true, then the following estimates hold:
    \begin{align}
     &\|\l_n^{-1+\frac{\m}{2}}A^{\g -\a+\frac{\b}{2}}u_n\| =O(1),\label{3.50} \\
 &  \| \lambda_n^{-\frac{\m}{2}}A^{\frac{\g}{2}+\alpha-\frac{\beta}{2}}u_n\| =O(1),  \label{5.22}\\
 &    \|\lambda_n^{-\frac{\m}{2}}A^{1+\alpha-\frac{3\b}{2}}u_n\| =O(1).\label{5.224}
\end{align}
\end{lemma}
\begin{proof}
Note that $\mu \in (0,1)$ and $\frac{1}{2} \le \left(1 - \frac{\gamma}{2}\right) \vee \frac{\gamma}{2} \le \beta$ also holds in the region $R_2(\gamma)$. Thus, both \eqref{3.42} and \eqref{3.43} remain valid. Since we still have 
\begin{align}\label{6.181}
    2\alpha - 2\beta + \big((1 - \gamma) \vee 0\big) < 0
\end{align}
 in  $R_2(\gamma)$, we apply $\lambda_n^{\mu - 1} A^{\alpha - \beta}$ to \eqref{3.30d} and  $\lambda_n^{\mu - 1} $ to \eqref{3.30b}; adding the resulting once again yields \eqref{3.40}.

Define
$$\theta_n := u_n + A^{1 + \alpha - \beta - \gamma} y_n.$$
Then it follows that
$$\|A^{\frac{\gamma}{2}} \theta_n\| = O(1).$$
Using \eqref{5.8d}, \eqref{5.7}, and \eqref{3.40}, we further obtain
$$\|\lambda_n^{-1} A^\gamma \theta_n\| =O(1).$$

Since $\frac{\g}{2} < \g - \a +\frac{\b}{2} < \g$, by interpolation,
$$
\|\l_n^{-1+\frac{\m}{2}}A^{\g -\a+\frac{\b}{2}}\theta_n\|\le \|\l_n^{-1} A^\g \theta_n\|^{1-\frac{2\a-\b}{\g}} \|A^{\frac{\g}{2}}\theta_n\|^\frac{2\a-\b}{\g} = O(1), 
$$
then \eqref{3.50} holds since 
\begin{align*}
    \|\l_n^{-1+\frac{\m}{2}}A^{\g -\a+\frac{\b}{2}}u_n\|\le \|\l_n^{-1+\frac{\m}{2}}A^{\g -\a+\frac{\b}{2}}\theta_n\| + \|\l_n^{-1+\frac{\m}{2}}A^{1-\frac{\b}{2}}y_n\|=O(1).
\end{align*}
The boundedness of $\|\l_n^{-1+\frac{\m}{2}}A^{1-\frac{\b}{2}}y_n\|$ follows from \eqref{3.43} and $-1+\frac{\mu}{2} {<} -\frac{1}{2}$. 

   Since $\frac{\gamma}{2}<\frac{\g}{2}+\alpha-\frac{\beta}{2}< \gamma -\alpha+\frac{\beta}{2}$,  one concludes from \eqref{3.50} that
\begin{align*}
   \| \lambda_n^{-\frac{\m}{2}}A^{\frac{\g}{2}+\alpha-\frac{\beta}{2}}u_n\|\leq \|A^{\frac{\g}{2}}u_n\|^{1-\frac{2\alpha-\beta}{\gamma+\beta-2\alpha}}\|\l_n^{-1+\frac{\mu}{2}}A^{\g-\alpha+\frac{\beta}{2}}u_n\|^{\frac{2\alpha-\beta}{\gamma+\beta-2\alpha}}=O(1).
\end{align*}
Moreover, if $1+\alpha-\frac{3\b}{2}{\le }\frac{\gamma}{2}$, then \eqref{5.224} holds clearly; otherwise, by \eqref{6.181}, we have $\frac{\gamma}{2} {<} 1+\alpha-\frac{3\b}{2}\leq \gamma-\alpha+\frac{\beta}{2}$. Then
 $$\|\lambda_n^{-\frac{2+2\alpha-3\beta-\gamma}{\gamma}}A^{1+\alpha-\frac{3\b}{2}}u_n\|\leq \|\lambda_n^{-1+\frac{\m}{2}}A^{\gamma-\alpha+\frac{\beta}{2}}u_n\|^{\frac{2+2\alpha-3\beta-\gamma}{\gamma+\b-2\a}} \|A^{\frac{\gamma}{2}}u_n\|^{1-\frac{2+2\alpha-3\beta-\gamma}{\gamma+\b-2\a}}=O(1).$$
Thus we have \eqref{5.224} since \eqref{3.50} and $\frac{2+2\alpha-3\beta-\gamma}{\gamma}\leq \frac{\mu}{2}$, which is from  $\beta>(1-\frac{\gamma}{2})\vee\frac{\g}{2}$.
\end{proof}

\begin{figure}[t]
\centering
\begin{subfigure}{0.19\textwidth}
\centering
\includegraphics[width=\textwidth]{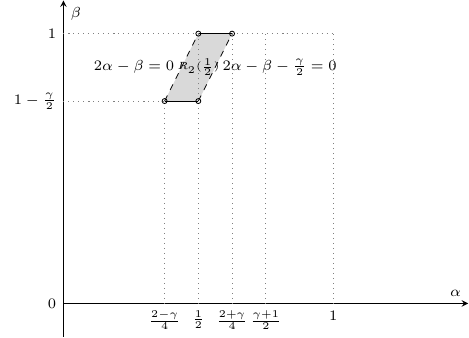}
\caption{$\gamma=\frac{1}{2}$}
\end{subfigure}~
\begin{subfigure}{0.19\textwidth}
\centering
\includegraphics[width=\textwidth]{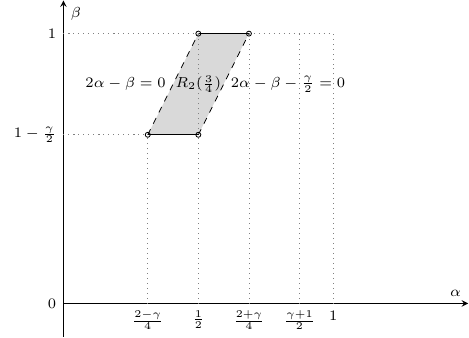}
\caption{$\gamma=\frac{3}{4}$}
\end{subfigure}
\centering
\begin{subfigure}{0.19\textwidth}
\centering
\includegraphics[width=\textwidth]{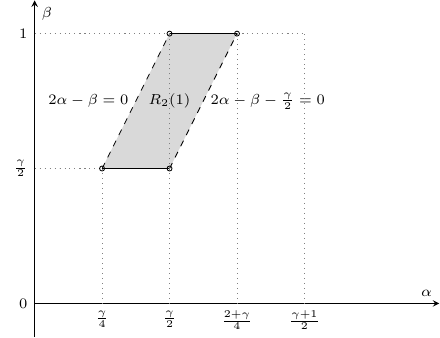}
\caption{$\gamma=1$, $a\ne 1$}
\end{subfigure}~
\begin{subfigure}{0.19\textwidth}
\centering
\includegraphics[width=\textwidth]{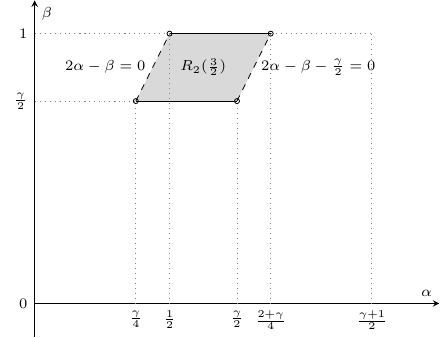}
\caption{$\gamma=\frac{3}{2}$}
\end{subfigure}~
\begin{subfigure}{0.19\textwidth}
\centering
\includegraphics[width=\textwidth]{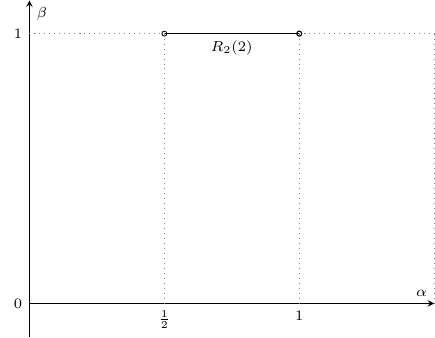}
\caption{$\gamma=2$}
\end{subfigure}
\caption{Visualization of $R_2(\gamma)$ when $\gamma=\frac{1}{2},\frac{3}{4},1,\frac{3}{2},2$.}
\label{fig:r2}
\end{figure}
The regularity in $R_2(\gamma)$ is stated as follows.
\begin{proposition}
    \label{main2}
  The  semigroup $e^{\cA t}$ is of Gevrey class $\delta>\frac{\g}{2(2\alpha-\beta)}$ in $R_2(\gamma)$.
\end{proposition}
\begin{proof}
 By Lemma \ref{lemma:method}, it suffices to prove \eqref{2.7} for $\m=\frac{2(2\a-\b)}{\g}\in (0,1)$. Let $(\a,\b,\g)\in R_2(\gamma)$, as shown in Figure~\ref{fig:r2}. By Lemma \ref{lemma:important}, if \eqref{2.7} fails,  we have \eqref{5.7}-\eqref{5.8e}. 
It follows from \eqref{5.8c}, \eqref{3.43} and $-1+\frac{\mu}{2} {<} -\frac{1}{2}$ that
\be\label{7.22} \lan \l_n^{-1} Ay_n, w_n \ran = \lan\l_n^{-1+\frac{\mu}{2}} A^{1-\frac{\b}{2}}y_n, \l_n^{-\frac{\mu}{2}}A^\frac{\b}{2} w_n \ran = o(1). 
\ee
Therefore,  \eqref{5.83} and \eqref{7.22} yield that 
\be\label{7.221}  \|A^\frac{1}{2}y_n\|=o(1). \ee 
On the other hand, since $0<\a-\frac{\b}{2} < \frac{\g}{2}$, by interpolation
\begin{align*}
    \|\l_n^{-\frac{\m}{2}}A^{\a-\frac{\b}{2}}v_n\| \le \|\l_n^{-1}A^\frac{\g}{2}v_n\|^\frac{2\a-\b}{\g} \|v_n\|^{1-\frac{2\a-\b}{\g}}=O(1).
\end{align*}
By \eqref{5.8c} and the above estimation, we then obtain
\begin{align}\label{5.202}
    \lan \l_n^{-\mu} A^\a w_n, v_n \ran = \lan \l_n^{-\frac{\m}{2}} A^\frac{\b}{2} w_n,
 \l_n^{-\frac{\m}{2}}A^{\a-\frac{\b}{2}}v_n \ran = o(1).
\end{align}
Moreover, \eqref{5.82} and \eqref{5.202} imply 
\begin{align}\label{7.222}
    i\|v_n\|^2 + a\lan \l_n^{-1} A^\g u_n, v_n \ran=o(1).
\end{align}
If
\begin{align}\label{5.19}
    \|\l_n^{-\frac{\m}{2}}A^{\a-\frac{\b}{2}}v_n\|=o(1),
\end{align}
then by \eqref{3.50} we obtain
\begin{align*}
    \lan \l_n^{-1} A^\g u_n, v_n \ran = \lan \l_n^{-1+\frac{\m}{2}}A^{\g -\a+\frac{\b}{2}}u_n,
 \l_n^{-\frac{\m}{2}}A^{\a-\frac{\b}{2}}v_n \ran = o(1).
\end{align*}
 This, together with \eqref{7.221} and \eqref{7.222} implies $\|v_n\|^2+\|A^{1\over2}y_n\|^2=o(1)$, which contradicts to \eqref{5.8b}.

We now show \eqref{5.19} holds. Since $\a - \b \le 0$ in $R_2(\gamma)$, taking the inner product of \eqref{3.30d} with $A^{\alpha-\beta}v_n$, we get 
\begin{align}\label{5.222}
    i\lan\lambda_n^{-\m+1}w_n, A^{\alpha-\beta}v_n\ran+\lan\lambda_n^{-\m}Ay_n, A^{\alpha-\beta}v_n\ran+\|\l_n^{-\frac{\m}{2}}A^{\a-\frac{\b}{2}}v_n\|^2+\lan\lambda_n^{-\m}A^{\beta}w_n, A^{\alpha-\beta}v_n\ran=o(1).
\end{align}
Taking the inner product of \eqref{3.30b} with $A^{\alpha-\beta }w_n$ on $H$, we have
$$i\lan  \l_n^{-\m+1}v_n,  A^{\a-\b}w_n\ran+a\lan\l_n^{-\mu}A^{\gamma}u_n, A^{\alpha-\b}w_n\ran-\|\lambda_n^{-\frac{\m}{2}}A^{\alpha-\frac{\beta}{2}}w_n\|^2=o(1).$$
It is easy to see that $\|\lambda_n^{-\frac{\m}{2}}A^{\alpha-\frac{\beta}{2}}w_n\|=o(1)$ in the above equation due to $\a-\frac{\beta}{2}\leq \frac{\beta}{2}$ and \eqref{5.8c}. Since  $\beta>(1-\frac{\gamma}{2})\vee\frac{\g}{2}$, we also have $\|\lambda_n^{-\frac{\m}{2}}A^{((1-\frac{\gamma}{2})\vee\frac{\g}{2})-\frac{\beta}{2}}w_n\|=o(1)$. Note that ${\gamma}-((1-\frac{\gamma}{2})\vee\frac{\g}{2})+\alpha-\frac{\beta}{2}\leq \frac{\gamma}{2}+\alpha-\frac{\beta}{2}$, then  \eqref{5.8c} and \eqref{5.22} yields
$$\lan\l_n^{-\mu}A^{\gamma}u_n, A^{\alpha-\b}w_n\ran\leq \|\lambda_n^{-\frac{\m}{2}}A^{{\gamma}-((1-\frac{\gamma}{2})\vee\frac{\g}{2})+\alpha-\frac{\beta}{2}}u_n\|\|\lambda_n^{-\frac{\m}{2}}A^{((1-\frac{\gamma}{2})\vee\frac{\g}{2})-\frac{\beta}{2}}w_n\|=o(1).$$
Thus,
\begin{align}\label{5.205}
   i \lan  \l_n^{-\m+1}v_n,  A^{\a-\b}w_n\ran=o(1).
\end{align}

On the other hand, by \eqref{6.181}, we have $\frac{1}{2}+\alpha-\beta-\frac{\gamma}{2}\leq 0$. Then by \eqref{5.8c}, \eqref{3.30a}, \eqref{3.30c}  and \eqref{5.224}, we obtain
\begin{align}\label{5.223}
\begin{array}{lll}
     \lan\lambda_n^{-\m}Ay_n, A^{\alpha-\beta}v_n\ran&=\lan A^{1+\alpha-\beta-\frac{\gamma}{2}}y_n, \lambda_n^{-\m}A^{\frac{\gamma}{2}}v_n\ran\\   \noalign{\medskip}  
    &=-i\lan A^{1+\alpha-\beta-\frac{\gamma}{2}}y_n, \lambda_n^{1-\m}A^{\frac{\gamma}{2}}u_n\ran+o(1)\\ \noalign{\medskip}
    &=-i\lan \lambda_n^{1-\m}A^{\frac{1}{2}}y_n, A^{\frac{1}{2}+\alpha-\beta}u_n\ran+o(1)\\ \noalign{\medskip}
    &=-\lan \lambda_n^{-\frac{\m}{2}}A^{\frac{\beta}{2}}w_n, \lambda_n^{-\frac{\m}{2}}A^{1+\alpha-\frac{3\b}{2}}u_n\ran+o(1)\\ \noalign{\medskip}
   & \leq \|\lambda_n^{-\frac{\m}{2}}A^{\frac{\beta}{2}}w_n\|\|\lambda_n^{-\frac{\m}{2}}A^{1+\alpha-\frac{3\b}{2}}u_n\|+o(1)\\ \noalign{\medskip}
    &=o(1). 
\end{array} 
\end{align}
Therefore, \eqref{5.19} follows from \eqref{5.202}, \eqref{5.222}, \eqref{5.205} and \eqref{5.223}.
\end{proof}

\begin{figure}[t]
\centering
\begin{subfigure}{0.19\textwidth}
\centering
\includegraphics[width=\textwidth]{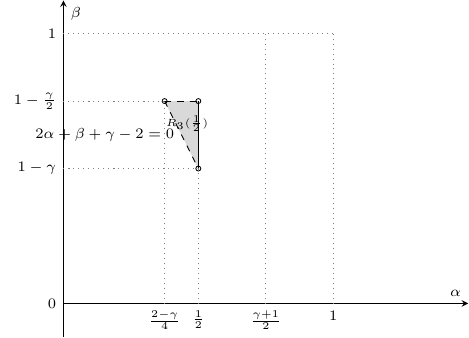}
\caption{$\gamma=\frac{1}{2}$}
\end{subfigure}~
\begin{subfigure}{0.19\textwidth}
\centering
\includegraphics[width=\textwidth]{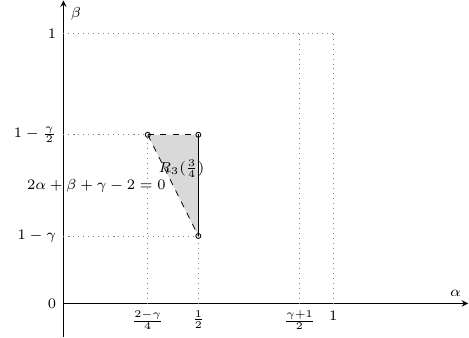}
\caption{$\gamma=\frac{3}{4}$}
\end{subfigure}
\centering
\begin{subfigure}{0.19\textwidth}
\centering
\includegraphics[width=\textwidth]{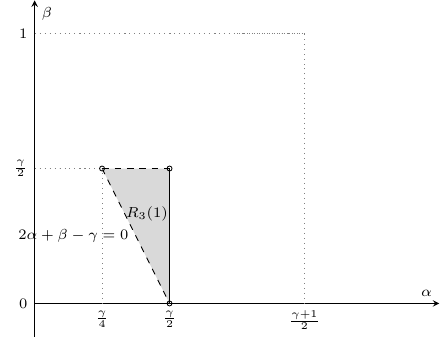}
\caption{$\gamma=1$, $a\ne 1$}
\end{subfigure}~
\begin{subfigure}{0.19\textwidth}
\centering
\includegraphics[width=\textwidth]{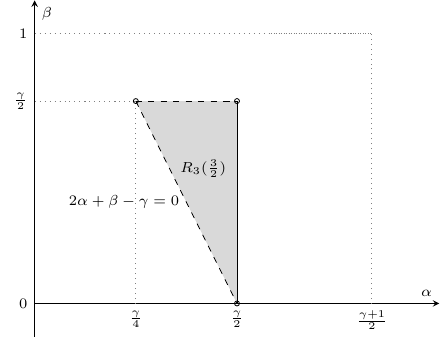}
\caption{$\gamma=\frac{3}{2}$}
\end{subfigure}~
\begin{subfigure}{0.19\textwidth}
\centering
\includegraphics[width=\textwidth]{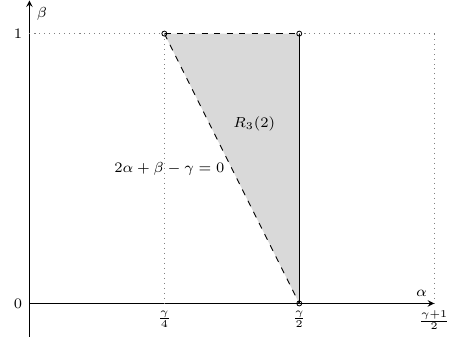}
\caption{$\gamma=2$}
\end{subfigure}
\caption{Visualization of $R_3(\gamma)$ when $\gamma=\frac{1}{2},\frac{3}{4},1,\frac{3}{2},2$.}
\label{fig:r3}
\end{figure}
\begin{proposition}\label{main3}
 The  semigroup $e^{\cA t}$ is of Gevrey class $\delta>\frac{\g}{2(2\alpha+\beta-(\gamma\vee (2-\gamma)))}$  in 
  $$R_3(\gamma)=\Big\{(\a,\b,\g)\in  R(\g)  \bigm|2\alpha+\beta-(\gamma\vee (2-\gamma))>0, \; \alpha  \le \frac{1\vee \g}{2}, \;  \beta<(1-\frac{\gamma}{2})\vee\frac{\g}{2} \Big\},\;\; {1\over2}\le\gamma \le 2.$$
\end{proposition}

\begin{proof}
 By Lemma \ref{lemma:method}, it suffices to prove \eqref{2.7} for $\m=\frac{2(2\alpha+\beta-(\gamma\vee (2-\gamma)))}{\g}\in (0,1)$. Let $(\a,\b,\g)\in R_3(\gamma)$, as shown in Figure~\ref{fig:r3}. If \eqref{2.7} fails, by Lemma \ref{lemma:important}, we have \eqref{5.7}-\eqref{5.8e}. Next, we consider two cases, respectively.

 \textbf{Case (i): When $1<\gamma\le 2$.} Let $\m=\frac{2(2\alpha+\beta-\gamma)}{\gamma}\in (0,1)$ and 
$$   R_3(\gamma)=\Big\{(\a,\b,\g)\in  R(\g)\bigm| 2\a+\b - \g > 0, \; \a \le  \frac{\gamma}{2},\;  \b<\frac{\gamma}{2} \Big\}. $$
Since $0<\alpha-\frac{\beta}{2}<\frac{\g}{2}, \;\alpha {\color{red} \le } \frac{\gamma}{2}$, by interpolation we have
\begin{align}\label{9.261}
   \|\lambda_n^{-1+\frac{\m}{2}}A^{\alpha-\frac{\beta}{2}}v_n\| \leq  \|\lambda_n^{\frac{\beta-2\alpha}{\gamma}}A^{\alpha-\frac{\beta}{2}}v_n\|\leq \|\lambda_n^{-1}A^{\frac{\g}{2}}v_n\|^{\frac{2\alpha-\beta}{\gamma}}\|v_n\|^{1-\frac{2\alpha-\beta}{\gamma}}=O(1).
\end{align}
Then \begin{align}\label{10.163}
    \lan \l_n^{-1}  A^{\alpha}w_n,v_n\ran\leq \|\lambda_n^{-1+\frac{\m}{2}}A^{\alpha-\frac{\beta}{2}}v_n\|\|\lambda_n^{-\frac{\m}{2}}A^{\frac{\beta}{2}}w_n\|=o(1).
\end{align}
Since $\alpha\leq \frac{\g}{2}$, it follows from \eqref{3.30a} that 
\begin{align}\label{10.16}
    \|A^{\alpha}u_n\|, \;\; \|\lambda_n^{-1}A^{\alpha}v_n\|=O(1).
\end{align}
If $\beta \ge \frac{1}{2}$, then \eqref{3.42}–\eqref{3.43} remain valid. In this case, combining \eqref{5.8c} and \eqref{3.43}, we obtain
\begin{align*}
    \lan \l_n^{-1} Ay_n,w_n \ran = \lan\l_n^{-1+\frac{\mu}{2}} A^{1-\frac{\b}{2}}y_n,\l_n^{-\frac{\mu}{2}}A^\frac{\b}{2} w_n \ran = o(1). 
\end{align*}
If instead $\beta < \frac{1}{2}$, then \eqref{5.8e} implies that $\|\lambda_n^{-1}A^{\beta}w_n\|\le O(1)$. Using this, together with \eqref{5.8c}, \eqref{3.30d} and \eqref{10.16}, we deduce that
\begin{align}\label{10.161}
    \|\lambda_n^{-1}Ay_n\|=O(1).
\end{align}
By interpolation, it follows that
\begin{align}\label{10.162}
    \|\lambda_n^{-1+\frac{\mu}{2}}A^{1-\frac{\mu}{4}}y_n\|\leq \|\lambda_n^{-1}Ay\|^{1-\frac{\mu}{2}}\|A^{\frac{1}{2}}y_n\|^{\frac{\mu}{2}}=O(1).
\end{align}
Thus, we obtain
\be\label{9.26} \lan \l_n^{-1} Ay_n,w_n \ran = \lan\l_n^{-1+\frac{\mu}{2}} A^{1-\frac{\mu}{4}}y_n,\;\l_n^{-\frac{\mu}{2}}A^\frac{\mu}{4} w_n \ran = o(1), 
\ee
where we use $\frac{\m}{4}\leq \frac{\beta}{2}$ and \eqref{5.8c}.
Therefore,  combining \eqref{5.83} and \eqref{9.26}, we conclude that 
\be\label{9.25}  \|A^\frac{1}{2}y_n\|=o(1). \ee

Next, let $\theta_n:=au_n-A^{\alpha-\gamma}w_n$. By $\alpha  \le\frac{\g}{2}$ and \eqref{3.30b}, we have
\begin{equation}
\label{4.100}
\|A^{\frac{\g}{2}}\theta_n\|,\;\;
\|\lambda_n^{-1}A^{\g}\theta_n\|=O(1).
\end{equation}
Then by interpolation,
\begin{align}\label{4.10}
    \|\lambda_n^{-\frac{\mu}{2}}A^{\alpha+\frac{\beta}{2}}\theta_n\|, \;\;\|\l_n^{-1+\frac{\mu}{2}}A^{\gamma-\frac{\mu\gamma}{4}}\theta_n\| =O(1).
\end{align}
Recalling that \eqref{5.82} yields
\begin{align*}
    i\|v_n\|^2+a\lan \l_n^{-1} A^{\gamma}\theta_n,v_n\ran=o(1).
\end{align*}
If we have \begin{align}\label{4.112}
    \|\lambda_n^{-\frac{\mu}{2}}A^{\frac{\mu\gamma}{4}}v_n\|=o(1),
\end{align}
then by \eqref{4.10}, we get 
$$\lan \l_n^{-1} A^{\gamma}\theta_n,v_n\ran=\lan\l_n^{-\frac{\mu}{2}}A^{\frac{\mu\gamma}{4}}v_n,\l_n^{-1+\frac{\mu}{2}}A^{\gamma-\frac{\mu\gamma}{4}}\theta_n \ran=o(1),$$
which implies
\begin{align*}
    \|v_n\|=o(1),
\end{align*}
This, together with \eqref{9.25} yields $\|v_n\|^2+\|A^{1\over2}y_n\|^2=o(1)$, which contradicts to \eqref{5.8b}.

Finally, we prove that \eqref{4.112} holds.
Acting $A^{\alpha+\beta-\gamma}v_n$ on \eqref{3.30d} yields
\begin{align}\label{9.262}
    \lan i\lambda_n^{-\mu+1}w_n,A^{\alpha+\beta-\gamma}v_n \ran+ \lan \lambda_n^{-\mu}Ay_n,A^{\alpha+\beta-\gamma}v_n \ran+\lan \lambda_n^{-\mu}A^{\beta}w_n,A^{\alpha+\beta-\gamma}v_n \ran+\|\lambda_n^{-\frac{\mu}{2}}A^{\frac{\mu\gamma}{4}}v_n\|^2=o(1).
\end{align}
We now show that the first three terms on the left-hand side of \eqref{9.262} vanish as $n\to\infty$, which will imply \eqref{4.112}.

For the first term of \eqref{9.262}, note that $\alpha+\beta-\gamma <0$, taking the inner product of \eqref{3.30b} with $A^{\alpha+\beta-\gamma }w_n$ on $H$ yields
\begin{align*}
    i\lan\lambda_n^{-\mu+1}v_n,A^{\alpha+\beta-\gamma }w_n\ran+\lan\lambda_n^{-\frac{\mu}{2}}A^{\alpha+\frac{\beta}{2}}\theta_n,\lambda_n^{-\frac{\mu}{2}}A^{\frac{\beta}{2}}w_n\ran=o(1).
\end{align*}
By \eqref{5.8c} and \eqref{4.10}, we see 
$$\lan\lambda_n^{-\frac{\mu}{2}}A^{\alpha+\frac{\beta}{2}}\theta_n,\lambda_n^{-\frac{\mu}{2}}A^{\frac{\beta}{2}}w_n\ran\le\|\lambda_n^{-\frac{\mu}{2}}A^{\alpha+\frac{\beta}{2}}\theta_n\|\|\lambda_n^{-\frac{\mu}{2}}A^{\frac{\beta}{2}}w_n\|=o(1).$$
Thus, we obtain
\begin{align}\label{4.9}
    \lan\lambda_n^{-\mu+1}v_n,A^{\alpha+\beta-\gamma }w_n\ran=o(1).
\end{align}

For the second term of \eqref{9.262}, let $$p:=\frac{\mu}{2}+\frac{2-2\gamma}{\g}.$$ If $\alpha+\frac{\b}{2}-\g+1\le \frac{\gamma}{2}$, then by \eqref{4.100},
\begin{align}\label{4.11}
    \|\l_n^{-p}A^{\alpha+\frac{\b}{2}-\g+1}\theta_n\| =O(1).
\end{align}
 If instead $\frac{\gamma}{2}<\alpha+\frac{\b}{2}-\g+1<\g$,  then \eqref{4.11} still holds by interpolation and \eqref{4.100}.
Now, observe that $1+\alpha+\beta-\frac{3\gamma}{2}<\frac{1}{2}$ and $\alpha+\beta-\gamma+\frac{1}{2}<\frac{\gamma}{2}$, by\eqref{5.8c}, \eqref{3.30a}, \eqref{3.30c} and \eqref{4.11}, we get
\begin{align}\label{4.113}
\begin{array}{lll}
      \lan \lambda_n^{-\mu}Ay_n,A^{\alpha+\beta-\gamma}v_n \ran&=-i \lan A^{1+\alpha+\beta-\frac{3\gamma}{2}}y_n,\lambda_n^{-\mu+1}A^{\frac{\gamma}{2}}u_n \ran+o(1)\\  \noalign{\medskip}
     &=-i\lan \lambda_n^{-\mu+1}A^{\frac{1}{2}}y_n,A^{\alpha+\beta-\gamma+\frac{1}{2}}u_n \ran+o(1)\\  \noalign{\medskip}
      &=-\lan \lambda_n^{-\mu}A^{\frac{1}{2}}w_n,A^{\alpha+\beta-\gamma+\frac{1}{2}}u_n \ran+o(1)\\ \noalign{\medskip}
       &=-\lan \lambda_n^{-\mu+p}A^{\frac{\beta}{2}}w_n,\lambda_n^{-p}A^{\alpha+\frac{\beta}{2}-\gamma+1}u_n \ran+o(1)\\ \noalign{\medskip}
       &\displaystyle
       =- {1\over a}\lan \lambda_n^{-\mu+p}A^{\frac{\beta}{2}}w_n,\lambda_n^{-p}A^{\alpha+\frac{\beta}{2}-\gamma+1}\theta_n \ran+ {1\over a}\|\lambda_n^{-\frac{\mu}{2}}A^{\a+\frac{\beta+1}{2}-\gamma}w_n\|^2+o(1)\\ \noalign{\medskip}
       &\displaystyle
       \leq {1\over a} \|\lambda_n^{-\mu+p}A^{\frac{\beta}{2}}w_n\|\|\lambda_n^{-p}A^{\alpha+\frac{\beta}{2}-\gamma+1}\theta_n\|+ {1\over a}\|\lambda_n^{-\frac{\mu}{2}}A^{\a+\frac{\beta+1}{2}-\gamma}w_n\|^2+o(1)\\ \noalign{\medskip}
       &=o(1).
\end{array}
\end{align}
In the last step of \eqref{4.113}, we use \eqref{5.8c} and \eqref{4.11} since $-\mu+p<-\frac{\mu}{2}$ and $\alpha+\frac{\beta+1}{2}-\gamma<\frac{\b}{2}$.

For the third term of \eqref{9.262}, note that if $\alpha+\frac{3\beta}{2}-\gamma<0$, then clearly 
\begin{align}\label{4.115}
    \| \lambda_n^{-\frac{\mu}{2}}A^{\alpha+\frac{3\beta}{2}-\gamma}v_n\|=O(1).
\end{align}
If instead $0<\alpha+\frac{3\beta}{2}-\gamma<\frac{\gamma}{2}$, by interpolation we have
\begin{align*}
    \|\lambda_n^{-\frac{2\alpha+3\beta-2\gamma}{\g}}A^{\alpha+\frac{3\beta}{2}-\gamma}v_n\|\leq \|\lambda^{-1}A^{ \color{red}\frac{\g}{2}} v_n\|^{\frac{2\alpha+3\beta-2\gamma}{\g}}\| {\color{red}v_n}\|^{1-\frac{2\alpha+3\beta-2\gamma}{\g}}=O(1).
\end{align*}
Since $-\frac{2\alpha+3\beta-2\gamma}{\g}>-\frac{\mu}{2}$, then \eqref{4.115} also holds and
\begin{align}\label{4.114}
    \lan \lambda_n^{-\mu}A^{\beta}w_n,A^{\alpha+\beta-\gamma}v_n \ran\le \|\lambda_n^{-\frac{\mu}{2}}A^{\frac{\beta}{2}}w_n\|\|\lambda_n^{-\frac{\mu}{2}}A^{\alpha+\frac{3\beta}{2}-\gamma}v_n\|=o(1).
\end{align}
 Therefore, we conclude \eqref{4.112} from \eqref{9.262}, \eqref{4.9}, \eqref{4.113}, and \eqref{4.114}.

 \textbf{Case (ii): When $\frac{1}{2}\le \gamma  \le  1$.} Let $\m=\frac{2(2\alpha+\beta+\gamma-2)}{\g}$ and 
$$   R_3(\gamma)=\Big\{(\a,\b,\g)\in  R(\g) \bigm|\a\le \frac{1}{2},\; \beta<1-\frac{\gamma}{2},\;2\alpha+\beta+\g-2>0\Big\}. $$
Note that in this case we also have  $-1+\frac{\m}{2}\le\frac{\beta-2\alpha}{\gamma}$ and $0<\alpha-\frac{\beta}{2}<\frac{\g}{2}$, then \eqref{9.261} and \eqref{10.163} still hold. 
Since $\gamma\le 1$, applying $A^{\frac{\gamma}{2}-\frac{1}{2}}$ to \eqref{3.30d} yields
$$iA^{\frac{\gamma}{2}-\frac{1}{2}}w_n+\lambda_n^{-1}A^{\frac{\gamma}{2}+\frac{1}{2}}y_n+\lambda_n^{-1}A^{\alpha+\frac{\gamma}{2}-\frac{1}{2}}v_n+\lambda_n^{-1}A^{\beta+\frac{\gamma}{2}-\frac{1}{2}}w_n=o(1).$$
Since $\alpha\le \frac{1}{2}, \;\beta < 1-\frac{\g}{2} $, it follows that 
\begin{align}
    \|\lambda_n^{-1}A^{\frac{\g}{2}+\frac{1}{2}}y_n\|=O(1).
\end{align}
Note that $\frac{1}{2}\le \frac{3}{2}-\alpha-\frac{\beta}{2}\le \frac{\g}{2}+\frac{1}{2}$, so by interpolation we obtain $\|\lambda_n^{-1+\frac{\mu}{2}}A^{\frac{3}{2}-\alpha-\frac{\beta}{2}}y_n\|=O(1)$ and thus
\begin{align*}
    \lan \l_n^{-1} Ay_n,\;w_n \ran = \lan\l_n^{-1+\frac{\mu}{2}} A^{\frac{3}{2}-\alpha-\frac{\beta}{2}}y_n,\;\l_n^{-\frac{\mu}{2}}A^{\alpha+\frac{\beta}{2}-\frac{1}{2}} w_n \ran = o(1),
\end{align*}
Therefore, by \eqref{5.83}, we conclude that
\be \label{5.321} \|A^\frac{1}{2}y_n\|=o(1). \ee

Since $\alpha\le \frac{1}{2}$, by \eqref{3.30b} we see $\|\lambda_n^{-1}A^{\gamma}u_n\|=O(1)$. Note that $\frac{\g}{2}{<} \alpha+\frac{\beta}{2}+\g-1{<} \gamma$, by interpolation we obtain the following estimates:
\begin{align}
   \|\l_n^{-1+\frac{\mu}{2}} A^{1-\alpha-\frac{\beta}{2}+\frac{\g}{2}}u_n\|&=O(1),\label{5.23}\\
   \|\l_n^{-\frac{\mu}{2}} A^{\alpha+\frac{\beta}{2}+\g-1}u_n\|&=O(1).\label{5.24}
\end{align}
If
\begin{align}\label{5.19*}
    \|\l_n^{-\frac{\m}{2}}A^{\frac{\g}{2}+\alpha+\frac{\beta}{2}-1}v_n\|=o(1),
\end{align}
then by \eqref{5.23},
\begin{align*}
    \lan \l_n^{-1} A^\g u_n,\;v_n \ran = \lan \l_n^{-1+\frac{\m}{2}} A^{1-\alpha-\frac{\beta}{2}+\frac{\g}{2}}u_n,
\;\l_n^{-\frac{\m}{2}} A^{\frac{\g}{2}+\alpha+\frac{\beta}{2}-1}v_n \ran = o(1).
\end{align*}
 This, together with \eqref{5.82} and \eqref{10.163} implies $\|v_n\|=o(1)$, 
 which contradicts to \eqref{5.8b} by \eqref{5.321}.

It suffices to show \eqref{5.19*} holds. 
Since $\a <1-\frac{\g}{2},\;\b <1-\frac{\g}{2}$, taking the inner product of \eqref{3.30d} with $A^{\alpha+\beta+\gamma-2}v_n$, we get 
\begin{align}\label{5.222*}
  &  i\lan\lambda_n^{-\m+1}w_n,\;A^{\alpha+\beta+\gamma-2}v_n\ran+\lan\lambda_n^{-\m}y_n,\;A^{\alpha+\beta+\gamma-1}v_n\ran+\|\l_n^{-\frac{\m}{2}}A^{\frac{\g}{2}+\alpha+\frac{\beta}{2}-1}v_n\|^2 \notag
    \\ & +\lan\lambda_n^{-\m}w_n,\;A^{\alpha+2\beta+\gamma-2}v_n\ran=o(1).
\end{align}
Taking the inner product of \eqref{3.30b} with $A^{\alpha+\beta+\gamma-2 }w_n$ on $H$, we have
$$i\lan  \l_n^{-\m+1}v_n, A^{\alpha+\beta+\gamma-2}w_n\ran+a\lan\l_n^{-\mu}A^{\gamma}u_n,A^{\alpha+\beta+\gamma-2}w_n\ran-\|\lambda_n^{-\frac{\m}{2}}A^{\alpha+\frac{\beta}{2}+\frac{\g}{2}-1}w_n\|^2=o(1).$$
It is easy to see that $\|\lambda_n^{-\frac{\m}{2}}A^{\alpha+\frac{\beta}{2}+\frac{\g}{2}-1}w_n\|=o(1)$ in the above equation since $\alpha+\frac{\beta}{2}+\frac{\g}{2}-1\leq \frac{\beta}{2}$. By \eqref{5.8c}, \eqref{5.24} and $\frac{\b}{2}+\g-1\le\frac{\beta}{2}$,
$$\lan\l_n^{-\mu}A^{\gamma}u_n,\;A^{\alpha+\beta+\gamma-2}w_n\ran\leq \|\lambda_n^{-\frac{\m}{2}}A^{\alpha+\frac{\beta}{2}+\g-1}u_n\|\|\lambda_n^{-\frac{\m}{2}}A^{\frac{\b}{2}+\g-1}w_n\|=o(1).$$
Thus,
\begin{align}\label{5.205*}
   i\lan\l_n^{-\m+1}A^{\alpha+\beta+\gamma-2} v_n,\; w_n\ran=o(1).
\end{align}

On the other hand, since $\alpha+\beta+\gamma-\frac{3}{2} \le {\gamma \over2}$,
by \eqref{5.8c}, \eqref{3.30a}, \eqref{3.30c}  and \eqref{5.24}, we obtain
\begin{align}\label{5.223*}
\begin{array}{lll}
     \lan\lambda_n^{-\m}Ay_n,\;A^{\alpha+\beta+\gamma-2}v_n\ran&=\lan A^{\frac{1}{2}}y_n,\;\lambda_n^{-\m}A^{\alpha+\beta+\gamma-\frac{3}{2}}v_n\ran\\ \noalign{\medskip}
    &=-i\lan A^{\frac{1}{2}}y_n,\;\lambda_n^{1-\m}A^{\alpha+\beta+\gamma-\frac{3}{2}}u_n\ran+o(1)\\ \noalign{\medskip}
    &=-\lan \lambda_n^{-\m}A^{\frac{1}{2}}w_n,\;A^{\alpha+\beta+\gamma-\frac{3}{2}}u_n\ran+o(1)\\ \noalign{\medskip}
   & \leq \|\lambda_n^{-\frac{\m}{2}}A^{\frac{\beta}{2}}w_n\|\|\lambda_n^{-\frac{\m}{2}}A^{\alpha+\frac{\beta}{2}+\gamma-1}u_n\|+o(1)\\  \noalign{\medskip}
    &=o(1). 
\end{array} 
\end{align}
Since  
$\alpha+\frac{3\beta}{2}+\gamma-2\le \alpha+\frac{\beta}{2}+\frac{\g}{2}-1 $ and  
$0\le  \alpha+\frac{\beta}{2}+\frac{\g}{2}-1\leq \frac{\gamma}{2}$, 
by interpolation and \eqref{5.8b}, \eqref{3.30a}, we get $$\|\lambda_n^{-\frac{\m}{2}}A^{\alpha+\frac{3\beta}{2}+\gamma-2}v_n\|,\;\;
\|\lambda_n^{-\frac{\m}{2}}A^{\alpha+\frac{\beta}{2}+\frac{\g}{2}-1}v_n\|=O(1).$$
Then
\begin{align*}
    \lan\lambda_n^{-\m}w_n,\;A^{\alpha+2\beta+\gamma-2}v_n\ran\leq \|\lambda_n^{-\frac{\m}{2}}A^{\frac{\beta}{2}}w_n\|\|\lambda_n^{-\frac{\m}{2}}A^{\alpha+\frac{3\beta}{2}+\gamma-2}v_n\|=o(1).
\end{align*}
This, together with \eqref{5.222*}, \eqref{5.205*}, and \eqref{5.223*} implies \eqref{5.19*} holds.
\end{proof}

\begin{proposition}\label{main4}
 The  semigroup $e^{\cA t}$ is of Gevrey class $\delta>\frac{\a}{\beta}$  in 
  $$R_4(\gamma)=\Big\{(\a,\b,\g)\in  R(\g) \bigm|\a>\beta,\;\a>\frac{\g}{2},\;\beta>\frac{4\alpha^2-2\alpha\gamma}{4\alpha-\gamma-1}\Big\}, \quad 1<\gamma\le 2.$$
\end{proposition}

\begin{figure}[t]
\centering
\begin{subfigure}{0.49\textwidth}
\centering
\includegraphics[width=0.5\textwidth]{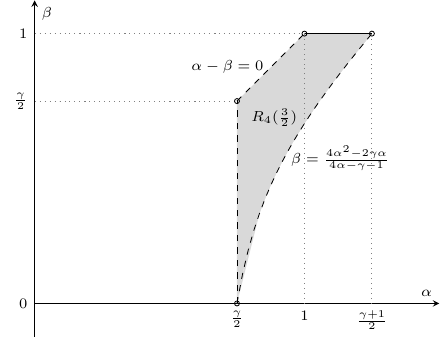}
\caption{$\gamma=\frac{3}{2}$}
\end{subfigure}~
\begin{subfigure}{0.49\textwidth}
\centering
\includegraphics[width=0.5\textwidth]{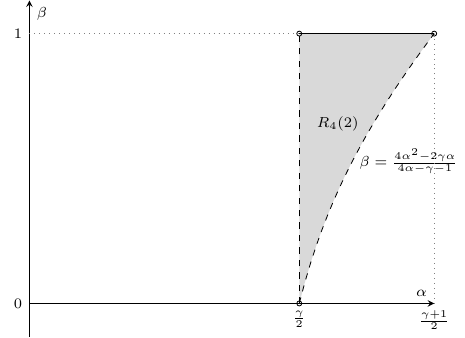}
\caption{$\gamma=2$}
\end{subfigure}
\caption{Visualization of $R_4(\gamma)$ when $\gamma=\frac{3}{2},2$.}
\label{fig:r4}
\end{figure}
\begin{proof}
 By Lemma \ref{lemma:method}, it suffices to prove \eqref{2.7} for $\mu=\frac{\beta}{\a}\in (0,1)$.   Let $(\a,\b,\g)\in R_4(\gamma)$, as shown in Figure~\ref{fig:r4}.  By Lemma \ref{lemma:important},  if \eqref{2.7} fails, then we have \eqref{5.7}-\eqref{5.8e}.  
Since $\alpha>\frac{\gamma}{2}$, it follows from \eqref{3.30d} that
\begin{align}\label{5.161}
    iA^{-\a+\frac{\gamma}{2}}w_n+\lambda_n^{-1}A^{1-\alpha+\frac{\gamma}{2}}y_n+\lambda_n^{-1}A^{\frac{\gamma}{2}}v_n+\lambda_n^{-1}A^{\beta-\alpha+\frac{\gamma}{2}}w_n=o(1).
\end{align}
{If $\beta \leq \frac{1}{2}$}, let $\eta_n:=A^{\frac{1}{2}}y_n+A^{\beta-\frac{1}{2}}w_n$,  then $$\|\eta_n\|,\;\;\|\lambda_n^{-1}A^{\frac{1+\gamma}{2}-\alpha}\eta_n\|=O(1).$$
Since $\alpha-\frac{\beta}{2}<\frac{\gamma}{2}$, we observe that $0\leq \frac{1}{2}-\frac{\beta}{2}\leq \frac{1+\gamma}{2}-\alpha$,  by interpolation we have 
$\|\lambda_n^{-\frac{1-\beta}{1+\g-2\a}}A^{\frac{1}{2}-\frac{\beta}{2}}\eta_n\|=O(1)$.
Let $\xi_n:=A^{1-\frac{\beta}{2}}y_n+A^{\frac{\beta}{2}}w_n$, so that
\begin{align}\label{5.163}
    \|\lambda_n^{-\frac{1-\beta}{1+\g-2\a}}\xi_n\|=\|\lambda_n^{-\frac{1-\beta}{1+\g-2\a}}A^{\frac{1}{2}-\frac{\beta}{2}}\eta_n\|=O(1).
\end{align}
Furthermore, since $\beta>\frac{4\alpha^2-2\gamma\alpha}{4\alpha-\gamma-1}$, we deduce that $-1+\frac{1-\beta}{1+\g-2\a}<-\frac{\mu}{2}$. Hence, by \eqref{5.8c} and \eqref{5.163},
\begin{align}\label{6.20}
\begin{array}{lll}
    \lan\lambda_n^{-1}Ay_n,w_n \ran&=   \lan\lambda_n^{-\frac{1-\beta}{1+\g-2\a}}A^{1-\frac{\beta}{2}}y_n,\lambda_n^{-1+\frac{1-\beta}{1+\g-2\a}}A^{\frac{\beta}{2}}w_n \ran\\ \noalign{\medskip}
    &=\lan\lambda_n^{-\frac{1-\beta}{1+\g-2\a}}\xi_n,\lambda_n^{-1+\frac{1-\beta}{1+\g-2\a}}A^{\frac{\beta}{2}}w_n \ran-\lan\lambda_n^{-1+\frac{\mu}{2}}A^{\frac{\beta}{2}}w_n,\lambda_n^{-\frac{\mu}{2}}A^{\frac{\beta}{2}}w_n \ran\\ \noalign{\medskip}
    &=o(1),
\end{array}
\end{align}
where we use $-1+\frac{\mu}{2}<-\frac{\mu}{2}$.

If $\beta>\frac{1}{2}$, applying $A^{\frac{1}{2}-\b}$ to \eqref{3.30d} yields
\begin{align}\label{5.16}
    iA^{\frac{1}{2}-\b}w_n+ \l_n^{-1}A^{\frac{3}{2}-\b} y_n + \l_n^{-1}A^{\frac{1}{2}-\b+\a} v_n + \l_n^{-1}A^\frac{1}{2} w_n=o(1). 
\end{align}
If $\frac{1}{2}-\b+\a>\frac{\gamma}{2}$, acting $A^{\frac{\gamma}{2}-(\frac{1}{2}-\b+\a)}$ on \eqref{5.16} yields 
\begin{align}\label{5.211}
    \| \l_n^{-1}A^{1+\frac{\g}{2}-\a} y_n\|=O(1).
\end{align}
Note that $\frac{1}{2}<1-\frac{\beta}{2}<1+\frac{\g}{2}-\a$, by interpolation we have
\begin{align*}
    \|\lambda_n^{-\frac{1-\beta}{1+\g-2\a}}A^{1-\frac{\beta}{2}}y_n\|=O(1).
\end{align*}
Moreover, it follows from $\beta>\frac{4\alpha^2-2\gamma\alpha}{4\alpha-\gamma-1}$ that $-1+\frac{1-\beta}{1+\g-2\a}<-\frac{\mu}{2}$ and thereby
$$ \lan\lambda_n^{-1}Ay_n,w_n \ran\leq \|\lambda_n^{-\frac{1-\beta}{1+\g-2\a}}A^{1-\frac{\beta}{2}}y_n\|\|\lambda_n^{-1+\frac{1-\beta}{1+\g-2\a}}A^{\frac{\beta}{2}}w_n \|=o(1).$$
If instead $\frac{1}{2}-\b+\a\le\frac{\gamma}{2}$, then it follows from \eqref{5.8d}, \eqref{5.8e} and \eqref{5.16} that $\|\l_n^{-1}A^{\frac{3}{2}-\b} y_n\|=O(1)$. Since $\frac{1}{2}<1-\frac{\beta}{2}<\frac{3}{2}-\b$, by interpolation one has 
$\|\lambda_n^{-\frac{1}{2}}A^{1-\frac{\beta}{2}}y_n\|=O(1)$.
Note that $-1+\frac{\mu}{2}\le -\frac{1}{2}$, then
$$\lan\lambda_n^{-1}Ay_n,w_n \ran\leq \|\lambda_n^{-1+\frac{\mu}{2}}A^{1-\frac{\beta}{2}}y_n\|\|\lambda_n^{-\frac{\mu}{2}}A^{\frac{\beta}{2}}w_n\|=o(1).$$
By the above analysis and \eqref{5.83}, we  obtain
\begin{align}\label{5.162}
    \|A^{\frac{1}{2}}y_n\|=o(1).
\end{align}

On the other hand, let $\theta_n:=aA^{\gamma-\a}u_n-w_n$ then $ \|\theta_n\|=O(1)$, and \eqref{3.30b} can be rewritten as 
\begin{align}\label{5.152}
    iv_n+\lambda^{-1}_nA^{\a}\theta_n=o(1),
\end{align}
which implies $\|\lambda^{-1}_nA^{\a}\theta_n\|=O(1)$. By interpolation we get
\begin{align}
    \|\lambda^{-1+\frac{\mu}{2}}_nA^{\a-\frac{\beta}{2}}\theta_n\|=O(1),\label{5.153}\\
    \|\lambda^{-\frac{\mu}{2}}_nA^{\frac{\beta}{2}}\theta_n\|=O(1).\label{5.154}
\end{align}
Taking the inner product of \eqref{5.152} with $v_n$ gives
$$i\|v_n\|^2+\lan\lambda_n^{-1}A^{\alpha}\theta_n,v_n\ran=o(1).$$
If \begin{align}\label{5.155}
   \|\lambda_n^{-\frac{\mu}{2}}A^{\frac{\beta}{2}}v_n\|=o(1), 
\end{align}
then by \eqref{5.153}, $\lan\lambda_n^{-1}A^{\alpha}\theta_n,v_n\ran\leq \|\lambda_n^{-1+\frac{\mu}{2}}A^{\alpha-\frac{\beta}{2}}\theta_n\|\|\lambda_n^{-\frac{\mu}{2}}A^{\frac{\beta}{2}}v_n\|=o(1).$
Therefore,
\begin{align*}
    \|v_n\|=o(1).
\end{align*}
This, together with \eqref{5.8b} and \eqref{5.162} yields a contradiction.

Finally, we prove \eqref{5.155} holds.
Taking the inner product of \eqref{3.30d} with $A^{\beta-\alpha}v_n$, we obtain
\begin{align}\label{5.157}
    \lan i\lambda_n^{1-\mu} w_n,A^{\beta-\alpha}v_n\ran+\lan\lambda_n^{-\mu}Ay_n, A^{\beta-\alpha}v_n\ran+\lan\lambda_n^{-\mu}A^{\beta}w_n, A^{\beta-\alpha}v_n\ran+\|\lambda_n^{-\frac{\mu}{2}}A^{\frac{\beta}{2}}v_n\|^2=o(1).
\end{align}
Taking the inner product of \eqref{3.30b} with $A^{\beta-\alpha}w_n$ gives
$$\lan i\lambda_n^{1-\mu}v_n,A^{\beta-\a}w_n\ran+\lan\lambda_n^{-\mu}A^{\alpha}\theta_n,A^{\beta-\a}w_n\ran=o(1).$$
Note that by \eqref{5.154}, $\lan\lambda_n^{-\mu}A^{\alpha}\theta_n,A^{\beta-\a}w_n\ran\leq \|\lambda_n^{-\frac{\mu}{2}}A^{\frac{\beta}{2}}\theta_n\|\|\lambda_n^{-\frac{\mu}{2}}A^{\frac{\beta}{2}}w_n\|=o(1)$,
then\begin{align}\label{5.158}
    \lan i\lambda_n^{1-\mu}v_n,A^{\beta-\a}w_n\ran=o(1).
\end{align} 

On the other hand, since $0<\frac{1}{2}+\beta-\alpha<\frac{1}{2}<\frac{\gamma}{2}$, by interpolation we have
\begin{align*}
    \|\lambda_n^{-\frac{1}{\gamma}}A^{\frac{1}{2}}v_n\|,\;\;\|\lambda_n^{-\frac{1+2\beta-2\alpha}{\gamma}}A^{\frac{1}{2}+\beta-\a}v_n\|=O(1)
\end{align*}
Since $\beta<\alpha\leq \frac{\alpha(2\alpha-1)}{2\alpha-\gamma}$, we have $-\mu+\frac{1+2\beta-2\alpha}{\gamma}<0$.
Thus, by \eqref{5.162}, we obtain $$\|\lambda_n^{-\mu+\frac{1+2\beta-2\alpha}{\gamma}}A^{\frac{1}{2}}y_n\|=o(1),$$
and
\begin{align}\label{5.159}
    \lan\lambda_n^{-\mu}A^{\frac{1}{2}}y_n,A^{\frac{1}{2}+\beta-\a}v_n\ran\leq \|\lambda_n^{-\mu+\frac{1+2\beta-2\alpha}{\gamma}}A^{\frac{1}{2}}y_n\|\|\lambda_n^{-\frac{1+2\beta-2\alpha}{\gamma}}A^{\frac{1}{2}+\beta-\a}v_n\|=o(1).
\end{align}
Next, by  \eqref{5.157}, \eqref{5.158} and \eqref{5.159}, one has
\begin{align*}
    \|\lambda_n^{-\frac{\mu}{2}}A^{\frac{\beta}{2}}v_n\|^2&=-\lan\lambda_n^{-\mu}A^{\beta}w_n, A^{\beta-\alpha}v_n\ran+o(1)\\
   & \leq \|\lambda_n^{-\frac{\mu}{2}}A^{\frac{3\beta}{2}-\alpha}w_n\|\|\lambda_n^{-\frac{\mu}{2}}A^{\frac{\beta}{2}}v_n\|+o(1)\\
   &\leq \frac{1}{2}\|\lambda_n^{-\frac{\mu}{2}}A^{\frac{3\beta}{2}-\alpha}w_n\|^2+\frac{1}{2}\|\lambda_n^{-\frac{\mu}{2}}A^{\frac{\beta}{2}}v_n\|^2+o(1).
\end{align*}
Therefore, by \eqref{5.8c} and $\frac{3\beta}{2}-\alpha<\frac{\beta}{2}$, we finally obtain \eqref{5.155}, finishing the proof.

\end{proof}

\begin{proposition}\label{main5}
 The  semigroup $e^{\cA t}$ is of Gevrey class $\delta>\frac{-2\a+\g+1}{2(-2\a+\b+\g)}$ in 
\begin{equation*}
R_5(\gamma)=\left\{
\begin{aligned}
&\Big\{(\a,\b,\g)\in  R(\g) \bigm| 2\alpha-\beta-\gamma<0,\;\frac{1}{2}<\alpha<\frac{\gamma+1}{2},\;2\alpha-2\beta-\gamma+1>0\Big\}, & \quad \frac{1}{2}\le\g\le 1,\\
& \Big\{(\a,\b,\g)\in  R(\g) \bigm| 2\alpha-\beta-\gamma<0,\;\beta\le\frac{4\alpha^2-2\gamma\alpha}{4\alpha-\gamma-1}\Big\}, & \quad 1<\g\leq 2.
\end{aligned}
\right.
\end{equation*}
\end{proposition}

\begin{figure}[t]
\begin{subfigure}{0.19\textwidth}
\centering
\includegraphics[width=\textwidth]{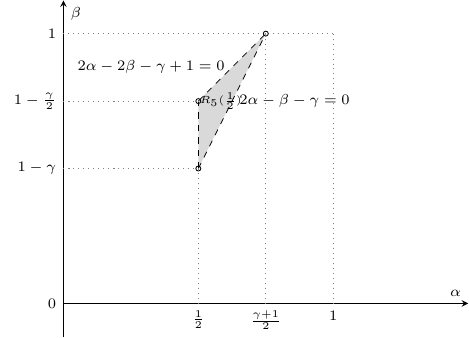}
\caption{$\gamma=\frac{1}{2}$}
\end{subfigure}
\begin{subfigure}{0.19\textwidth}
\centering
\includegraphics[width=\textwidth]{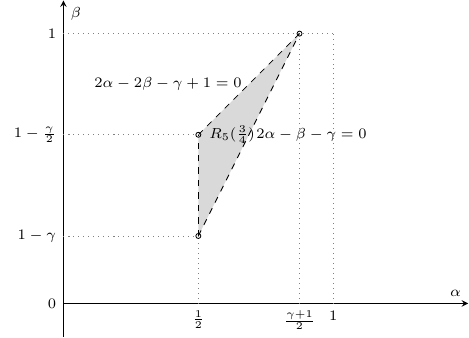}
\caption{$\gamma=\frac{3}{4}$}
\end{subfigure}
\centering
\begin{subfigure}{0.19\textwidth}
\centering
\includegraphics[width=\textwidth]{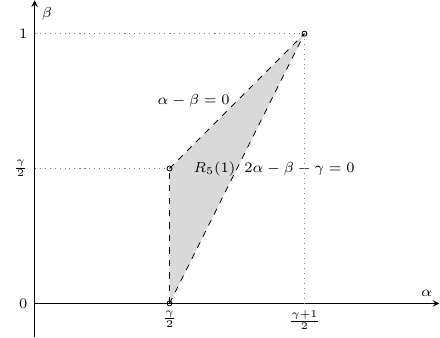}
\caption{$\gamma=1$, $a\ne 1$}
\end{subfigure}
\begin{subfigure}{0.19\textwidth}
\centering
\includegraphics[width=\textwidth]{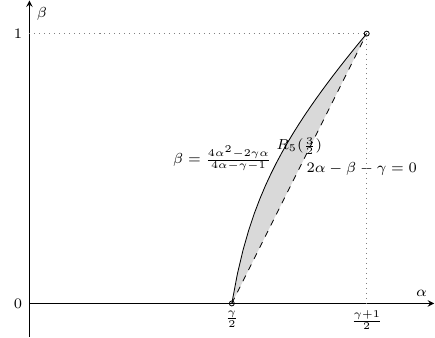}
\caption{$\gamma=\frac{3}{2}$}
\end{subfigure}
\begin{subfigure}{0.19\textwidth}
\centering
\includegraphics[width=\textwidth]{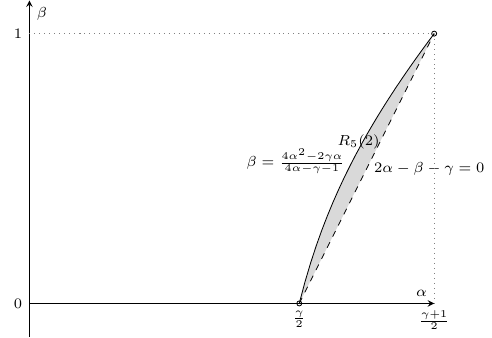}
\caption{$\gamma=2$}
\end{subfigure}
\caption{$R_5(\gamma)$ when $\gamma=\frac{1}{2},\frac{3}{4},1,\frac{3}{2},2$.}
\label{fig:r5}
\end{figure}

\begin{proof}
 By Lemma \ref{lemma:method}, it suffices to prove \eqref{2.7} for $\mu=\frac{2(-2\a+\b+\g)}{-2\a+\g+1}\in (0,1)$.  Let $(\a,\b,\g)\in R_5(\gamma)$, as shown in Figure~\ref{fig:r5}.  By Lemma \ref{lemma:important},  if \eqref{2.7} fails, then we have \eqref{5.7}-\eqref{5.8e}.  We consider two cases, respectively.

  \textbf{Case (i): When $1<\g\leq 2$.} 
Since we also have $\alpha-\beta>0,\alpha>\frac{\gamma}{2}$ in $R_5(\gamma)$ (This is evident from Figure \ref{fig:r5}), the proof of $$\|A^{\frac{1}{2}}y_n\|=o(1)$$ in $R_5(\gamma)$ can be treated in the same way as in $R_4(\gamma)$. 
The only difference is that now $\mu=\frac{2(-2\a+\b+\g)}{-2\a+\g+1}$, so if $\beta\leq \frac{1}{2}$,  we directly obtain \eqref{6.20} using \eqref{5.8c} and \eqref{5.163}, without relying on the condition $\beta<\frac{4\alpha^2-2\gamma\alpha}{4\alpha-\gamma-1}$. In other words,
\begin{align*}
    \lan\lambda_n^{-1}Ay_n,w_n \ran&=  \lan\lambda_n^{-1+\frac{\mu}{2}}A^{1-\frac{\beta}{2}}y_n,\lambda_n^{-\frac{\mu}{2}}A^{\frac{\beta}{2}}w_n \ran\\
    &=\lan\lambda_n^{-1+\frac{\mu}{2}}\xi_n,\lambda_n^{-\frac{\mu}{2}}A^{\frac{\beta}{2}}w_n \ran-\lan\lambda_n^{-1+\frac{\mu}{2}}A^{\frac{\beta}{2}}w_n,\lambda_n^{-\frac{\mu}{2}}A^{\frac{\beta}{2}}w_n \ran\\
    &\leq \|\lambda_n^{-1+\frac{\mu}{2}}\xi_n\|\|\lambda_n^{-\frac{\mu}{2}}A^{\frac{\beta}{2}}w_n \|+\|\lambda_n^{-1+\frac{\mu}{2}}A^{\frac{\beta}{2}}w_n\|\|\lambda_n^{-\frac{\mu}{2}}A^{\frac{\beta}{2}}w_n \|\\
    &=o(1).
\end{align*}

Next, let $\theta_n:=aA^{\gamma-\a}u_n-w_n$, then by \eqref{3.30b} we still have 
$$ \|\theta_n\|,\|\lambda^{-1}_nA^{\a}\theta_n\|=O(1).$$
 By interpolation we get
\begin{align}
    \|\lambda^{-1+\frac{\beta}{2\alpha}}_nA^{\a-\frac{\beta}{2}}\theta_n\|=O(1),\label{5.153*}
\end{align}
Since $\beta<\frac{4\alpha^2-2\gamma\alpha}{4\alpha-\gamma-1}$, then $\mu<\frac{\beta}{\alpha}$. Thus, 
$$ \|\lambda^{-1+\frac{\mu}{2}}_nA^{\a-\frac{\beta}{2}}\theta_n\|=O(1).$$
Taking the inner product of \eqref{3.30b}  with $v_n$ gives
$$i\|v_n\|^2+\lan\lambda_n^{-1}A^{\alpha}\theta_n,v_n\ran=o(1).$$
If \begin{align}\label{5.155*}
   \|\lambda_n^{-\frac{\mu}{2}}A^{\frac{\beta}{2}}v_n\|=o(1), 
\end{align}
then $\lan\lambda_n^{-1}A^{\alpha}\theta_n,v_n\ran\leq \|\lambda_n^{-1+\frac{\mu}{2}}A^{\alpha-\frac{\beta}{2}}\theta_n\|\|\lambda_n^{-\frac{\mu}{2}}A^{\frac{\beta}{2}}v_n\|=o(1).$
Therefore,
\begin{align}\label{5.156*}
    \|v_n\|=o(1).
\end{align}
This, together with \eqref{5.8b} and $\|A^{\frac{1}{2}}y_n\|=o(1)$ yields a contradiction.

Finally, we prove \eqref{5.155*} holds.
Taking the inner product of \eqref{3.30d} with $A^{\beta-\alpha}v_n$, we obtain
\begin{align}\label{5.157*}
    \lan i\lambda_n^{1-\mu} w_n,A^{\beta-\alpha}v_n\ran+\lan\lambda_n^{-\mu}A^{\beta}w_n+\lambda_n^{-\mu}Ay_n, A^{\beta-\alpha}v_n\ran+\|\lambda_n^{-\frac{\mu}{2}}A^{\frac{\beta}{2}}v_n\|^2=o(1).
\end{align}
Taking the inner product of \eqref{3.30b} with $A^{\beta-\alpha}w_n$ gives
\begin{align}\label{5.21}
    \lan i\lambda_n^{1-\mu}v_n,A^{\beta-\a}w_n\ran+a\lan\lambda_n^{-\mu}A^{\g}u_n,A^{\beta-\a}w_n\ran-\|\lambda_n^{-\frac{\mu}{2}}A^{\frac{\beta}{2}}w_n\|^2=o(1).
\end{align}
Acting $A^{-\a+\frac{1}{2}}$ on \eqref{3.30b} gives
\begin{align}\label{6.27}
    iA^{-\a+\frac{1}{2}}v_n+a\lambda_n^{-1}A^{\g+\frac{1}{2}-\a}u_n-\lambda_n^{-1}A^{\frac{1}{2}}w_n=o(1).
\end{align}
Thus, \begin{align}\label{7.5}
    \|\lambda_n^{-1}A^{\g+\frac{1}{2}-\a}u_n\|=O(1).
\end{align} 
Since $\frac{\gamma}{2}\le\g+\frac{\b}{2}-\a\le \gamma+\frac{1}{2}-\a$, by interpolation we have 
\begin{align}\label{6.281}
    \|\lambda_n^{-\frac{\mu}{2}}A^{\g+\frac{\b}{2}-\a}u_n\|=O(1).
\end{align}
Then  
 $$\lan\lambda_n^{-\mu}A^{\g}u_n,A^{\beta-\a}w_n\ran\leq \|\lambda_n^{-\frac{\mu}{2}}A^{\g+\frac{\b}{2}-\a}u_n\|\|\lambda_n^{-\frac{\mu}{2}}A^{\frac{\beta}{2}}w_n\|=o(1).$$ 
Together with \eqref{5.21}, we obtain
\begin{align}\label{5.158*}
    \lan i\lambda_n^{1-\mu}v_n,A^{\beta-\a}w_n\ran=o(1).
\end{align} 

Next, acting $A^{-\a+\frac{1}{2}}$ on \eqref{3.30d} gives
$$iA^{-\a+\frac{1}{2}}w_n+\lambda_n^{-1}A^{\frac{3}{2}-\a}y_n+\lambda_n^{-1}A^{\frac{1}{2}}v_n+\lambda_n^{-1}A^{\frac{1}{2}-\alpha+\beta}w_n=o(1).$$
 Since $\a>\beta$, by \eqref{5.8d} and \eqref{5.8e} we see $\|\lambda_n^{-1}A^{\frac{1}{2}}v_n\|,\|\lambda_n^{-1}A^{\frac{1}{2}-\alpha+\beta}w_n\|=O(1).$
 Thus, $$\|\lambda_n^{-1}A^{\frac{3}{2}-\a}y_n\|=O(1).$$
 On the other hand, by $\|A^{\frac{1}{2}}y_n\|=o(1)$ and $1-\frac{\gamma}{2}\leq\frac{1}{2}$, we see $\|A^{1-\frac{\gamma}{2}}y_n\|=o(1)$. 
 Since $1-\frac{\gamma}{2}\leq 1+\frac{\beta}{2}-\a<\frac{3}{2}-\a$,  by interpolation,  we have $$\|\lambda_n^{-\frac{\mu}{2}}A^{1+\frac{\beta}{2}-\alpha}y_n\|=o(1).$$
We then conclude from \eqref{5.157*} and \eqref{5.158*} that
\begin{align*}
    \|\lambda_n^{-\frac{\mu}{2}}A^{\frac{\beta}{2}}v_n\|^2&=-\lan\lambda_n^{-\mu}A^{\beta}w_n+\lambda_n^{-\mu}Ay_n, A^{\beta-\alpha}v_n\ran+o(1)\\
   & \leq \|\lambda_n^{-\frac{\mu}{2}}A^{\frac{3\beta}{2}-\alpha}w_n+\lambda_n^{-\frac{\mu}{2}}A^{1+\frac{\beta}{2}-\alpha}y_n\|\|\lambda_n^{-\frac{\mu}{2}}A^{\frac{\beta}{2}}v_n\|+o(1)\\
   &\leq \frac{1}{2}\|\lambda_n^{-\frac{\mu}{2}}A^{\frac{3\beta}{2}-\alpha}w_n+\lambda_n^{-\frac{\mu}{2}}A^{1+\frac{\beta}{2}-\alpha}y_n\|^2+\frac{1}{2}\|\lambda_n^{-\frac{\mu}{2}}A^{\frac{\beta}{2}}v_n\|^2+o(1).
\end{align*}
Since $\frac{3\beta}{2}-\alpha<\frac{\beta}{2}$, we finally obtain
$$\|\lambda_n^{-\frac{\mu}{2}}A^{\frac{\beta}{2}}v_n\|\le \|\lambda_n^{-\frac{\mu}{2}}A^{\frac{3\beta}{2}-\alpha}w_n+\lambda_n^{-\frac{\mu}{2}}A^{1+\frac{\beta}{2}-\alpha}y_n\|\leq \|\lambda_n^{-\frac{\mu}{2}}A^{\frac{3\beta}{2}-\alpha}w_n\|+\|\lambda_n^{-\frac{\mu}{2}}A^{1+\frac{\beta}{2}-\alpha}y_n\|=o(1).$$
The proof for Case (i) is complete.

\textbf{Case (ii): When $\frac{1}{2}\le\g \le 1$.} 
One can easily check that the proof of $\|A^{\frac{1}{2}}y_n\|=o(1)$ follows identically to case (i). To establish a contradiction via \eqref{5.8b}, it remains to show that  $\|v_n\|=o(1)$.
We proceed by dividing the proof of
\begin{align}\label{6.273}
    \|v_n\|=o(1)
\end{align}
into two cases. 

\textbf{Case (ii-1): If $\beta>\alpha$.}  Since $\alpha>\frac{1}{2}$, we now still have \eqref{6.27}-\eqref{6.281}. 
On the other hand, multiplying \eqref{3.30d} by $\lambda_n^{2(\mu-1)}$ yields
$$i\lambda_n^{-1+\mu}w_n+\lambda_n^{-2+\mu}Ay_n+\lambda_n^{-2+\mu}A^{\alpha}v_n+\lambda_n^{-2+\mu}A^{\beta}w_n=o(1).$$
Taking the inner product of the above with $A^{\alpha-\beta}v_n$ yields
\begin{align*}
    \lan  i\lambda_n^{-1+\mu}w_n,A^{\alpha-\beta}v_n\ran+\lan \lambda_n^{-2+\mu}Ay_n+\lambda_n^{-2+\mu}A^{\beta}w_n,A^{\alpha-\beta}v_n\ran+\|\lambda_n^{-1+\frac{\mu}{2}}A^{\alpha-\frac{\beta}{2}}v_n\|^2=o(1).
\end{align*}
Note that $\mu<1$, $ \lan i\lambda_n^{-1+\mu}w_n,A^{\alpha-\beta}v_n\ran=o(1)$. 
Then 
\begin{align*}
    \|\lambda_n^{-1+\frac{\mu}{2}}A^{\alpha-\frac{\beta}{2}}v_n\|^2&=-\lan\lambda_n^{-2+\mu}Ay_n+\lambda_n^{-2+\mu}A^{\beta}w_n,A^{\alpha-\beta}v_n\ran+o(1)\\
    &\leq\frac{1}{2}\|\lambda_n^{-1+\frac{\mu}{2}}A^{1-\frac{\beta}{2}}y_n+\lambda_n^{-1+\frac{\mu}{2}}A^{\frac{\beta}{2}}w_n\|^2+ \frac{1}{2}\|\lambda_n^{-1+\frac{\mu}{2}}A^{\alpha-\frac{\beta}{2}}v_n\|^2+o(1).
\end{align*}
Since $\beta>\alpha>\frac{1}{2}$ and $2\alpha-2\beta-\gamma+1>0$, we now still have \eqref{5.16} and \eqref{5.211}.
However, since we have proved $\|A^{\frac{1}{2}}y_n\|=o(1)$,  by interpolation and \eqref{5.211}, we have
$$\|\lambda_n^{-1+\frac{\mu}{2}}A^{1-\frac{\beta}{2}}y_n\|=o(1).$$
By the above and \eqref{5.8c},
\begin{align}\label{6.274}
    \|\lambda_n^{-1+\frac{\mu}{2}}A^{\alpha-\frac{\beta}{2}}v_n\|\leq \|\lambda_n^{-1+\frac{\mu}{2}}A^{1-\frac{\beta}{2}}y_n\|+\|\lambda_n^{-1+\frac{\mu}{2}}A^{\frac{\beta}{2}}w_n\|=o(1).
\end{align}
Taking the inner product of \eqref{3.30b}  with $v_n$ gives
$$i\|v_n\|^2+a\lan\lambda_n^{-\frac{\mu}{2}}A^{\gamma-\alpha+\frac{\beta}{2}}u_n,\lambda_n^{-1+\frac{\mu}{2}}A^{\alpha-\frac{\beta}{2}}v_n\ran-\lan\lambda_n^{-\frac{\mu}{2}}A^{\frac{\beta}{2}}w_n,\lambda_n^{-1+\frac{\mu}{2}}A^{\alpha-\frac{\beta}{2}}v_n\ran=o(1).$$
Therefore, we deduce \eqref{6.273} from \eqref{5.8c}, \eqref{6.281} and \eqref{6.274}.

\textbf{Case (ii-2): If $\beta\le \alpha$.} Note that we still have \eqref{6.27}-\eqref{6.281}.  Let $$p:=\frac{\gamma-\beta}{\gamma+1-2\alpha}.$$
Since $\frac{\gamma}{2}\leq \gamma-\frac{\beta}{2}\leq \gamma+\frac{1}{2}-\alpha$ and $\frac{\b}{2}\leq \alpha-\frac{\beta}{2}\le \frac{1}{2}$, by interpolation we have
\begin{align}\label{7.52}
    \|\lambda_n^{-p}A^{\gamma-\frac{\beta}{2}}u_n\|=O(1),\, \|\lambda_n^{-p}A^{\a-\frac{\beta}{2}}w_n\|=o(1).
\end{align}
By \eqref{5.82},
$$i\|v_n\|^2+a\lan \l_n^{-p} A^{\gamma-\frac{\beta}{2}}u_n, \lambda_n^{-1+p} A^{\frac{\beta}{2}}v_n\ran- \lan \l_n^{-p} A^{\alpha-\frac{\beta}{2}}w_n,  \l_n^{-1+p}A^{\frac{\beta}{2}}v_n\ran=o(1).$$
In view of \eqref{7.52}, it suffices to show
\begin{equation} \label{7.51}
 \|\l_n^{-1+p}A^{\frac{\beta}{2}}v_n\|=o(1),
\end{equation}
from which it follows that $\|v_n\|=o(1)$.

Indeed, acting $A^{\frac{\gamma}{2}-\alpha}$ on \eqref{3.30d}, note that $\frac{1}{2}+\alpha-\b-\frac{\gamma}{2}\ge0$, we have $\|\lambda_n^{-1}A^{1+\frac{\gamma}{2}-\alpha}y_n\|=O(1)$. By interpolation, it then follows that $\|\lambda_n^{-1+\frac{\mu}{2}}A^{1-\frac{\beta}{2}}y_n\|=O(1)$. Moreover, since $\frac{1}{2}\leq 1+\frac{\beta}{2}-\alpha\leq 1-\frac{\b}{2}$, we conclude that
\begin{align}\label{7.54}
    \|\l_n^{-1+p}A^{ 1+\frac{\beta}{2}-\alpha}y_n\|=o(1).
\end{align}
On the other hand, taking the inner produce of \eqref{3.30b} with $\lambda_n^{-\frac{2(1-\gamma)}{1+\gamma-2\alpha}}A^{\beta-\alpha}w_n$ yields
\begin{align*}
     \lan i\lambda_n^{1-\mu} v_n,\lambda_n^{-\frac{2(1-\gamma)}{1+\gamma-2\alpha}}A^{\beta-\alpha}w_n\ran+a\lan\lambda_n^{-\frac{\mu}{2}-\frac{1-\gamma}{1+\gamma-2\alpha}}A^{\g+\frac{\beta}{2}-\alpha}u_n, \lambda_n^{-\frac{\mu}{2}}A^{\frac{\beta}{2}}w_n\ran+\|\l_n^{-\frac{\mu}{2}-\frac{1-\gamma}{1+\gamma-2\alpha}}A^{\frac{\beta}{2}}w_n\|^2=o(1).
\end{align*}
It follows from \eqref{5.8c} and \eqref{6.281} that $$\lan\lambda_n^{-\frac{\mu}{2}-\frac{1-\gamma}{1+\gamma-2\alpha}}A^{\g+\frac{\beta}{2}-\alpha}u_n, \lambda_n^{-\frac{\mu}{2}}A^{\frac{\beta}{2}}w_n\ran,\quad \|\l_n^{-\frac{\mu}{2}-\frac{1-\gamma}{1+\gamma-2\alpha}}A^{\frac{\beta}{2}}w_n\|=o(1),$$
where we also use $-\frac{\mu}{2}-\frac{1-\gamma}{1+\gamma-2\alpha}\leq -\frac{\mu}{2}$. Thus,
\begin{align}\label{7.55}
    \lan i\lambda_n^{1-\mu} v_n,\lambda_n^{-\frac{2(1-\gamma)}{1+\gamma-2\alpha}}A^{\beta-\alpha}w_n\ran=o(1).
\end{align}

Note that $$-1+p=-\frac{\mu}{2}-\frac{1-\gamma}{1+\gamma-2\alpha}.$$
Taking the inner produce of \eqref{3.30d} with $\lambda_n^{-\frac{2(1-\gamma)}{1+\gamma-2\alpha}}A^{\beta-\alpha}v_n$ yields
\begin{align}\label{7.53}
    \lan i\lambda_n^{1-\mu} w_n,\lambda_n^{-\frac{2(1-\gamma)}{1+\gamma-2\alpha}}A^{\beta-\alpha}v_n\ran+\lan\lambda_n^{-\mu}A^{\beta}w_n+\lambda_n^{-\mu}Ay_n, \lambda_n^{-\frac{2(1-\gamma)}{1+\gamma-2\alpha}}A^{\beta-\alpha}v_n\ran+\|\l_n^{-1+p}A^{\frac{\beta}{2}}v_n\|^2=o(1).
\end{align}
Then by \eqref{7.55} we get 
\begin{align*}
    \|\l_n^{-1+p}A^{\frac{\beta}{2}}v_n\|^2&\leq  \lan\lambda_n^{-\mu}A^{\beta}w_n+\lambda_n^{-\mu}Ay_n, \lambda_n^{-\frac{2(1-\gamma)}{1+\gamma-2\alpha}}A^{\beta-\alpha}v_n\ran+o(1)\\
    &\leq\frac{1}{2}\|\lambda_n^{-\frac{\mu}{2}-\frac{1-\gamma}{1+\gamma-2\alpha}}A^{\frac{3\beta}{2}-\alpha}w_n+\lambda_n^{-\frac{\mu}{2}-\frac{1-\gamma}{1+\gamma-2\alpha}}A^{ 1+\frac{\beta}{2}-\alpha}y_n\|^2+\frac{1}{2}\|\l_n^{-\frac{\mu}{2}-\frac{1-\gamma}{1+\gamma-2\alpha}}A^{\frac{\beta}{2}}v_n\|^2+o(1).
\end{align*}
Therefore, 
\begin{align*}
     \|\l_n^{-1+p}A^{\frac{\beta}{2}}v_n\|\leq  \|\l_n^{-\frac{\mu}{2}-\frac{1-\gamma}{1+\gamma-2\alpha}}A^{ \frac{3\beta}{2}-\alpha}w_n\|+ \|\l_n^{-\frac{\mu}{2}-\frac{1-\gamma}{1+\gamma-2\alpha}}A^{ 1+\frac{\beta}{2}-\alpha}y_n\|=o(1),
\end{align*}
where we have used \eqref{5.8c}, \eqref{7.54} and $\beta\leq \alpha$.
\end{proof}

\section{Identical wave speeds}\label{sec:same}
\setcounter{equation}{0}
\setcounter{theorem}{0}
This section is devoted to the proof of Theorem \ref{thm:same}, following the same strategy as in Section~\ref{sec:different}. Note that Lemma \ref{lemma:important} remains valid in the present setting.

Recall the parameter space partition in \eqref{2.5}. Since the expressions for $\mu$ in $\tilde{R}_1$, $\tilde{R}_2$ (when $\beta>\frac{1}{2}$) and $\tilde{R}_4$ coincide with those in $R_1(1)$, $R_2(1)$ and $R_5(1)$ when $a\neq1$ (see Figure \ref{Comparison}), the proof may be omitted in these cases. It remains to prove the result in $\tilde{R}_2$ (when $\beta\leq\frac{1}{2}$) and $\tilde{R}_3$.
 By Lemma \ref{lemma:method}, it suffices to verify \eqref{2.7} holds for $\mu=2(2\alpha-\b)$ in $\tilde{R}_2$ and $\mu=2\beta$ in $\tilde{R}_3$, respectively. By Lemma \ref{lemma:important}, if \eqref{2.7} fails, we have \eqref{5.7}-\eqref{5.8e}. 

Since $\alpha, \beta\leq \frac{1}{2}$ both in $\tilde{R}_2$ and $\tilde{R}_3$, by \eqref{3.30d} we see $$\|\lambda_n^{-1}Ay_n\|=O(1).$$
Then by interpolation, 
$$\|\lambda_n^{-1+\beta}A^{1-\frac{\beta}{2}}y_n\|=O(1).$$
Note that $-\beta\leq -2\alpha+\beta$, it follows from \eqref{5.8c} that
$\lan\lambda_n^{-1}Ay_n,w_n\ran=\lan\lambda_n^{-1+\beta}A^{1-\frac{\beta}{2}}y_n,\lambda_n^{-\beta}A^{\frac{\beta}{2}}w_n\ran=o(1)$. Thus, we deduce from  \eqref{5.83} that
\begin{align}\label{7.65}
    \|A^{\frac{1}{2}}y_n\|=o(1).
\end{align}
Next, we prove that $\|v_n\|=o(1)$. This, together with \eqref{5.8b} and \eqref{7.65} yields a contradiction. We will prove it in $\tilde{R}_2$ (when $\beta\leq\frac{1}{2}$) and $\tilde{R}_3$, respectively.

\textbf{Case (i): Let $(\a,\b)\in \tilde{R}_2$ (when $\beta\leq\frac{1}{2}$) and let $\mu=2(2\alpha-\b)$.}
Taking the inner produce of \eqref{3.30b} with $A^{\alpha-\beta}w_n$ and taking the inner produce of \eqref{3.30d} with $A^{\alpha-\beta}v_n$, respectively, we get
\begin{align}
 \lan i\lambda_n^{1-\mu} v_n,A^{\alpha-\beta}w_n\ran+\lan\lambda_n^{-\mu}Au_n, A^{\alpha-\beta}w_n\ran+\|\l_n^{-\frac{\mu}{2}}A^{\alpha-\frac{\b}{2}}w_n\|^2=o(1).\label{7.61}\\
    \lan i\lambda_n^{1-\mu} w_n,A^{\alpha-\beta}v_n\ran+\lan\lambda_n^{-\mu}Ay_n, A^{\alpha-\beta}v_n\ran+\|\l_n^{-\frac{\mu}{2}}A^{\alpha-\frac{\beta}{2}}v_n\|^2+\lan\lambda_n^{-\mu}A^{\beta}w_n, A^{\alpha-\beta}v_n\ran=o(1).\label{7.62}
\end{align}
Moreover, 
\begin{align}\label{7.68}
\begin{array}{lll}
     \lan\lambda_n^{-\mu}Au_n, A^{\alpha-\beta}w_n\ran&=\lan A^{\frac{1}{2}}u_n, \lambda_n^{-\mu}A^{\frac{1}{2}+\alpha-\beta}w_n\ran\\
    &=\lan A^{\frac{1}{2}}u_n, i\lambda_n^{-\mu+1}A^{\frac{1}{2}+\alpha-\beta}y_n\ran+o(1)\\
    &=-\lan \lambda_n^{-\mu}A^{\frac{1}{2}}v_n, A^{\frac{1}{2}+\alpha-\beta}y_n\ran+o(1)\\
    &=-\lan \lambda_n^{-\mu}A^{\alpha-\beta}v_n, Ay_n\ran+o(1).
\end{array}   
\end{align}
Therefore, adding \eqref{7.61} and \eqref{7.62} and taking the real part of it yields
\begin{align*}
  \|\l_n^{-\frac{\mu}{2}}A^{\alpha-\frac{\beta}{2}}v_n\|^2+\Re\lan\lambda_n^{-\frac{\mu}{2}}A^{\frac{\beta}{2}}w_n,\lambda_n^{-\frac{\mu}{2}} A^{\alpha-\frac{\beta}{2}}v_n\ran=o(1),
\end{align*}
where we use $\alpha-\frac{\beta}{2}\leq \frac{\beta}{2}$.
 We then conclude from \eqref{5.8c} and the above that 
 \begin{align}\label{7.63}
      \|\l_n^{-\frac{\mu}{2}}A^{\alpha-\frac{\beta}{2}}v_n\|=o(1).
 \end{align}
 On the other hand, since $\frac{1}{2}\leq 1+\frac{\beta}{2}-\alpha\leq 1$, by interpolation we have
 \begin{align}\label{7.64}
      \|\l_n^{-1+\frac{\mu}{2}}A^{1+\frac{\beta}{2}-\alpha}u_n\|=O(1).
 \end{align}
By \eqref{5.82}, we have
$$i\|v_n\|^2+\lan  \l_n^{-1+\frac{\mu}{2}}A^{1+\frac{\beta}{2}-\alpha}u_n, \l_n^{-\frac{\mu}{2}}A^{\alpha-\frac{\beta}{2}}v_n\ran- \lan \l_n^{-1+\frac{\mu}{2}} A^{\frac{\beta}{2}}w_n, \l_n^{-\frac{\mu}{2}}A^{\alpha-\frac{\beta}{2}}v_n\ran=o(1),$$
 implying $\|v_n\|=o(1)$ due to \eqref{5.8c}, \eqref{7.63} and \eqref{7.64}. 

\textbf{Case (ii): Let $(\a,\b)\in \tilde{R}_3$ and let $\mu=2\b$.}
Taking the inner produce of \eqref{3.30b} with $w_n$ and taking the inner produce of \eqref{3.30d} with $v_n$, respectively, we get
\begin{align}
 \lan i\lambda_n^{1-\mu} v_n,w_n\ran+\lan\lambda_n^{-\mu}Au_n, w_n\ran+\|\l_n^{-\frac{\mu}{2}}A^{\frac{\a}{2}}w_n\|^2=o(1).\label{7.66}\\
    \lan i\lambda_n^{1-\mu} w_n,v_n\ran+\lan\lambda_n^{-\mu}Ay_n, v_n\ran+\|\l_n^{-\frac{\mu}{2}}A^{\frac{\a}{2}}v_n\|^2+\lan\lambda_n^{-\mu}A^{\beta}w_n, v_n\ran=o(1).\label{7.67}
\end{align}
By a similar analysis as in \eqref{7.68}, we have $ \lan\lambda_n^{-\mu}Au_n, w_n\ran=-\lan \lambda_n^{-\mu}v_n, Ay_n\ran+o(1).$ 
Then, adding \eqref{7.66} and \eqref{7.67}, and taking the real part of it gives
\begin{align*}
  \|\l_n^{-\frac{\mu}{2}}A^{\frac{\a}{2}}v_n\|^2+\|\l_n^{-\frac{\mu}{2}}A^{\frac{\a}{2}}w_n\|^2+\Re\lan\lambda_n^{-\frac{\mu}{2}}A^{\beta-\frac{\a}{2}}w_n,\lambda_n^{-\frac{\mu}{2}} A^{\frac{\a}{2}}v_n\ran=o(1).
\end{align*}
We conclude from $\beta\leq \a$ that 
$$\|\l_n^{-\frac{\mu}{2}}A^{\frac{\a}{2}}v_n\|=o(1).$$
Furthermore, by interpolation, 
 \begin{align*}
      \|\l_n^{-1+\alpha}A^{1-\frac{\a}{2}}u_n\|=O(1).
 \end{align*}
Finally, by \eqref{5.82}, we have
$$i\|v_n\|^2+\lan  \l_n^{-1+\frac{\mu}{2}}A^{1-\frac{\a}{2}}u_n, \l_n^{-\frac{\mu}{2}}A^{\frac{\a}{2}}v_n\ran- \lan\lambda_n^{-1+\frac{\mu}{2}}A^{\beta-\frac{\a}{2}}w_n,\lambda_n^{-\frac{\mu}{2}} A^{\frac{\a}{2}}v_n\ran=o(1),$$
which  implies $\|v_n\|=o(1)$ since $-1+\frac{\mu}{2}\leq-1+\alpha$ and $\beta-\frac{\a}{2}\leq \frac{\beta}{2}$.

\section{Asymptotic expansion of eigenvalues}\label{sec:Asymptotic}
\label{sec:eigen}
This section studies the asymptotic expansion of eigenvalues associated with $\mathcal{A}_{\a,\b,\g}$. The asymptotic forms of these eigenvalues provide important clues about the order of Gevrey class of the semigroup $e^{\mathcal{A}_{\a,\b,\g}t}$ discussed in \Cref{thm:different} and \Cref{thm:same}. Furthermore, based on the asymptotic expansion of the eigenvalues, we can further prove that the orders of Gevrey class previously acquired are optimal. The development of this section largely follows the procedures discussed in previous works \cite{Deng-Han-Kuang-Zhang 2024,  Deng-Han-Kuang-Zhang 2025, Han-Kuang-Liu-Zhang 2024, Han-Kuang-Zhang 2023,  Hao-Kuang-Liu-Yong 2025, Hao-Liu-Yong 2015, Kuang-Liu-Sare 2021, Kuang-Liu-Tebou  2022}, and hence we omit  low-level technical steps of these procedures for brevity.

\subsection{Characteristic equations}
\label{sec:characteristic}
Let $\mu_n\rightarrow \infty$ as $n \rightarrow \infty$ be a sequence of eigenvalues of $A$. The characteristic equation associated with $\mathcal{A}_{\a,\b,\g}$ is
\begin{equation}
\label{eq:characteristic-general}
\lambda_n^4 + \mu_n^\beta \lambda_n^3 + (\mu_n+\m_n^{2\a}+a \mu_n^\g) \lambda_n^2 +a \mu_n^{\b+\g}\lambda_n + a \m_n^{1+\g}=0,
\end{equation}
where we consider $a \in \curly{1,2}$ and the parameter space
\begin{equation}
\label{eq:E}
E = \curly{(\alpha,\beta,\gamma) \mid 0 \le \alpha \le \frac{\gamma+1}{2},\; 0 \le \beta \le 1,\; \frac{1}{2} \le \gamma \le 2}.
\end{equation}
Note that when $\gamma=1$, the choice of $a$ represents whether
the two equations of the coupled system having different wave speeds ($a=2$) or the same wave speeds ($a=1$). For the latter case where $a=1$ and $\gamma=1$, we consider the parameter space 
\begin{equation}
\label{eq:E-tilde}
\tilde{E} := \curly{(\alpha,\beta) \mid 0 \le \alpha \le 1,\; 0 \le \beta \le 1}.
\end{equation}

The goal of \Cref{sec:eigen} is to identify the asymptotic forms of the roots to the following two instantiations of \eqref{eq:characteristic-general}. Specifically,
\begin{itemize}[leftmargin=*]
\item Set $a=2$ in \eqref{eq:characteristic-general}. Let $n\rightarrow\infty$, for all $(\alpha,\beta,\gamma) \in E$, we are interested in finding the roots to 
\begin{equation}
\label{eq:characteristic-different}
\lambda^4 + \mu_n^\beta \lambda^3 + (\mu_n+\m_n^{2\a}+2 \mu_n^\g) \lambda^2 +2 \mu_n^{\b+\g}\lambda + 2 \m_n^{1+\g}=0.
\end{equation}
\item Set $a=1$ and $\gamma=1$ in \eqref{eq:characteristic-general}.  Let $n\rightarrow\infty$, for all $(\alpha,\beta) \in \tilde{E}$, we are interested in finding the roots to 
\begin{equation}
\label{eq:characteristic-same}
\lambda^4 + \mu_n^\beta \lambda^3 + (\mu_n+\m_n^{2\a}+ \mu_n) \lambda^2 +  \mu_n^{\b+1}\lambda +  \m_n^{2}=0.
\end{equation}
\end{itemize}

\begin{figure}[t]
\centering
\begin{subfigure}{0.32\textwidth}
\centering
\includegraphics[width=\textwidth]{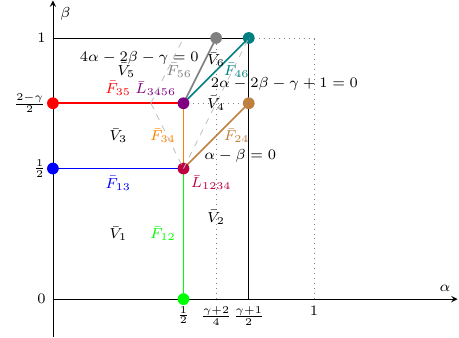}
\caption{$\gamma=\frac{1}{2}$}
\label{fig:gamma-1o2}
\end{subfigure}
\begin{subfigure}{0.32\textwidth}
\centering
\includegraphics[width=\textwidth]{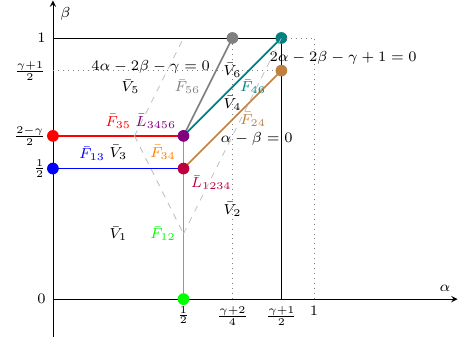}
\caption{$\gamma=\frac{3}{4}$}
\label{fig:gamma-3o4}
\end{subfigure}
\begin{subfigure}{0.32\textwidth}
\centering
\includegraphics[width=\textwidth]{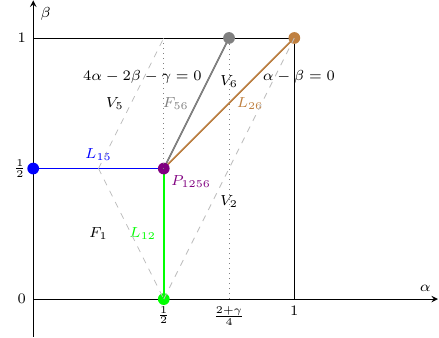}
\caption{$\gamma=1$, $a\ne 1$}
\label{fig:gamma-1}
\end{subfigure}
\begin{subfigure}{0.32\textwidth}
\centering
\includegraphics[width=\textwidth]{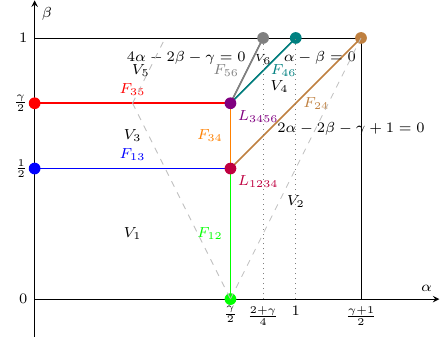}
\caption{$\gamma=\frac{3}{2}$}
\label{fig:gamma-3o2}
\end{subfigure}
\begin{subfigure}{0.32\textwidth}
\centering
\includegraphics[width=\textwidth]{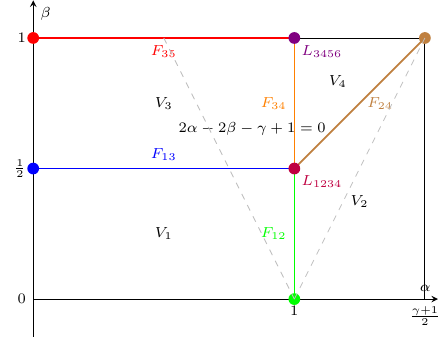}
\caption{$\gamma=2$}
\label{fig:gamma-2}
\end{subfigure}
\begin{subfigure}{0.32\textwidth}
\centering
\includegraphics[width=\textwidth]{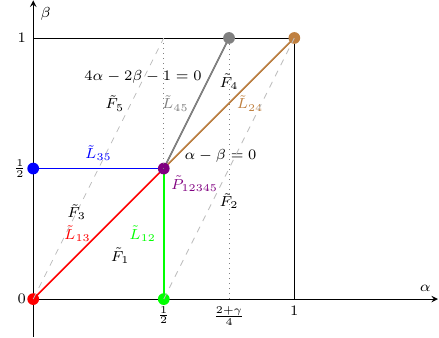}
\caption{$\tilde{E}$}
\label{fig:E-tilde}
\end{subfigure}
\caption{Regions in \Cref{def:partition-E} derived from the partition of $E$ (\Cref{fig:gamma-1o2}-\Cref{fig:gamma-2}) and regions in \Cref{def:partition-E-tilde} derived from the partition of $\tilde{E}$ (\Cref{fig:E-tilde}).  The gray dash lines are the boundaries of the regularity region.}
\label{fig:E-and-E-tilde}
\end{figure}

\subsection{Parameter space partitions}
\label{sec:partition}
Before presenting the roots to the characteristic equations \eqref{eq:characteristic-different} and \eqref{eq:characteristic-same}, we first partition the parameter spaces $E$ and $\tilde{E}$ into several regions, since the roots have distinct asymptotic forms across different regions under these partitions.

Specifically, following the method described in \cite[Chapter 4]{Hao-Kuang-Liu-Yong 2025}, we divide the set $E$ into the following regions.

\begin{definition}
\label[definition]{def:partition-E}
The parameter space $E$ as given in \eqref{eq:E} can be partitioned into disjoint regions as follows. 
\begin{align*}
V_1 = & \curly{ (\a,\beta,\g) \in E \mid
0 \le \a < \frac{\g}{2},\ 
0 \le \beta < \frac{1}{2},\
1 < \gamma\le 2
}, \\
V_2 = & \curly{ (\a,\beta,\g) \in E \mid 
\frac{\g}{2}<\a \le \frac{\g+1}{2},\ 
2\a - 2\beta -\g + 1>0,\
0 \le \beta,\
1 \le \gamma \le 2
}, \\
V_3 = & \curly{ (\a,\beta,\g) \in E \mid 
0 \le \a <\frac{\g }{2},\ 
\frac{1}{2}<\beta <\frac{\g}{2},\ 
1<\gamma\le 2
}, \\
V_4 = & \curly{(\a,\beta,\g) \in E \mid 
\frac{\gamma}{2} < \alpha < \frac{\gamma+1}{2},\ 
 \beta \le 1, 
\beta < \alpha,\ 
2\alpha - 2\beta - \gamma +1 < 0,\
1 < \gamma \le 2
}, \\
V_5 = & \curly{ (\a,\beta,\g) \in E \mid
0 \le \alpha ,\
\frac{\g}{2}<\beta\le 1,\ 
4\alpha - 2 \beta - \gamma <0,\  
1 \le \gamma < 2
}, \\
V_6 = & \curly{ (\a,\beta,\g) \in E \mid
 4 \alpha - 2\beta - \gamma >0,\
 \alpha < \beta \le 1,\
 1 \le \gamma < 2
 }, \\
F_{12} = & \curly{ (\a,\beta,\g) \in E \mid 
\a =  \frac{\g}{2},\
 0 \le \beta < \frac{1}{2},\
1 < \gamma \le 2
}, \\
F_{13} = & \curly{ (\a,\beta,\g) \in E \mid 
0 \le \a< \frac{\gamma}{2},\ 
\beta = \frac{1}{2},\
1 < \gamma\le 2
}, \\
F_{24} = & \curly{ (\a,\beta,\g) \in E \mid 
 \frac{\g}{2} < \alpha \le \frac{1+\gamma}{2},\ 
2\alpha - 2\beta - \gamma + 1 = 0,\
1 < \gamma \le 2
}, \\
F_{34} = & \curly{ (\a,\beta,\g) \in E \mid 
\a =  \frac{\g}{2},\
 \frac{1}{2} < \beta < \frac{\gamma}{2},\
1 < \gamma \le 2
}, \\
F_{35} = & \curly{ (\a,\beta,\g) \in E \mid 
0 \le \a < \frac{\g}{2},\ 
\beta = \frac{\g}{2},\
1<\gamma\le 2
}, \\
F_{46} = & \curly{ (\a,\beta,\g) \in E \mid 
\frac{\gamma}{2} < \a \le 1,\ 
\beta = \alpha,\
1<\gamma < 2
}, \\
F_{56} = & \curly{ (\a,\beta,\g) \in E \mid 
\frac{\gamma}{2} < \a \le \frac{\gamma+2}{4},\ 
4\alpha - 2\beta - \gamma  = 0,\
1 \le \gamma < 2
}, \\
L_{1234} = & \curly{ (\a,\beta,\g) \in E \mid 
\a = \frac{\g}{2},\ 
\beta = \frac{1}{2},\
1 < \gamma \le 2
}, \\
L_{3456} = & \curly{ (\a,\beta,\g) \in E \mid 
\a = \frac{\g}{2},\ 
\beta = \frac{\g}{2},\
1 < \gamma \le 2
}, \\
F_1 = & \curly{ (\a,\beta,\g) \in E \mid
0 \le \a < \frac{1}{2},\ 
0 \le \beta < \frac{1}{2},\
\gamma=1
}, \\
L_{12} = & \curly{ (\a,\beta,\g) \in E \mid 
\a =  \frac{1}{2},\
 0 \le \beta < \frac{1}{2},\
\gamma = 1
}, \\
L_{15} = & \curly{ (\a,\beta,\g) \in E \mid 
0 \le \a< \frac{1}{2},\ 
\beta = \frac{1}{2},\
\gamma=1
}, \\
L_{26} = & \curly{ (\a,\beta,\g) \in E \mid 
 \frac{1}{2} < \alpha \le 1,\ 
\beta = \alpha,\
\gamma=1
}, \\
P_{1256} = & \curly{ (\a,\beta,\g) \in E \mid 
\a = \frac{1}{2},\ 
\beta = \frac{1}{2},\
\gamma=1
}, \\
\bar{V}_1 = & \curly{ (\a,\beta,\g) \in E \mid
0 \le \a < \frac{1}{2},\ 
0 \le \beta < \frac{1}{2},\
\frac{1}{2} \le \gamma <1
}, \\
\bar{V}_2 = & \curly{ (\a,\beta,\g) \in E \mid 
\frac{1}{2}<\a \le \frac{\g+1}{2},\ 
0 \le \beta < \alpha ,\
\frac{1}{2} \le \gamma <1
}, \\
\bar{V}_3 = & \curly{ (\a,\beta,\g) \in E \mid 
0 \le \a <\frac{1}{2},\ 
\frac{1}{2}<\beta <1-\frac{\g}{2},\ 
\frac{1}{2} \le \gamma <1
}, \\
\bar{V}_4 = & \curly{(\a,\beta,\g) \in E \mid 
\frac{1}{2} < \alpha \le \frac{\gamma+1}{2},\ 
\beta > \alpha,\ 
2\alpha - 2\beta - \gamma +1 > 0,\
\frac{1}{2} \le \gamma <1
}, \\
\bar{V}_5 = & \curly{ (\a,\beta,\g) \in E \mid
0 \le \alpha ,\
1-\frac{\g}{2}<\beta\le 1,\ 
4\alpha - 2 \beta - \gamma <0,\  
\frac{1}{2} \le \gamma <1
}, \\
\bar{V}_6 = & \curly{ (\a,\beta,\g) \in E \mid
4 \alpha - 2\beta - \gamma >0,\
2\alpha - 2\beta - \gamma + 1 <0,\
\beta \le 1,\
\frac{1}{2} \le \gamma <1
}, \\
\bar{F}_{12} = & \curly{ (\a,\beta,\g) \in E \mid 
\a =  \frac{1}{2},\
 0 \le \beta < \frac{1}{2},\
\frac{1}{2} \le \gamma <1
}, \\
\bar{F}_{13} = & \curly{ (\a,\beta,\g) \in E \mid 
0 \le \a< \frac{1}{2},\ 
\beta = \frac{1}{2},\
\frac{1}{2} \le \gamma <1
}, \\
\bar{F}_{24} = & \curly{ (\a,\beta,\g) \in E \mid 
 \frac{1}{2} < \alpha \le \frac{1+\gamma}{2},\ 
\beta = \alpha,\
\frac{1}{2} \le \gamma <1
}, \\
\bar{F}_{34} = & \curly{ (\a,\beta,\g) \in E \mid 
\a =  \frac{1}{2},\
 \frac{1}{2} < \beta < 1-\frac{\gamma}{2},\
\frac{1}{2} \le \gamma <1
}, \\
\bar{F}_{35} = & \curly{ (\a,\beta,\g) \in E \mid 
0 \le \a < \frac{1}{2},\ 
\beta = 1-\frac{\g}{2},\
\frac{1}{2} \le \gamma <1
}, \\
\bar{F}_{46} = & \curly{ (\a,\beta,\g) \in E \mid 
\frac{1}{2} < \a \le \frac{1+\gamma}{2},\ 
2\alpha - 2\beta - \gamma + 1 = 0,\
\frac{1}{2} \le \gamma <1
}, \\
\bar{F}_{56} = & \curly{ (\a,\beta,\g) \in E \mid 
\frac{1}{2} < \a \le \frac{\gamma+2}{4},\ 
4\alpha - 2\beta - \gamma  = 0,\
\frac{1}{2} \le \gamma <1
}, \\
\bar{L}_{1234} = & \curly{ (\a,\beta,\g) \in E \mid 
\a = \frac{1}{2},\ 
\beta = \frac{1}{2},\
\frac{1}{2} \le \gamma <1
}, \\
\bar{L}_{3456} = & \curly{ (\a,\beta,\g) \in E \mid 
\a = \frac{1}{2},\ 
\beta = 1-\frac{\g}{2},\
\frac{1}{2} \le \gamma <1
}.
\end{align*}
\end{definition}

\noindent Similarly, we partition the set $\tilde{E}$ into the regions described below.
\begin{definition}
\label[definition]{def:partition-E-tilde}
The parameter space $\tilde{E}$ as given in \eqref{eq:E-tilde} can be partitioned into disjoint regions as follows. 
\begin{align*}
\tilde{F}_1 = & \curly{(\alpha,\beta) \in \tilde{E} \mid \beta<\alpha<\frac{1}{2},\; 0 < \beta < \frac{1}{2}},\\
\tilde{F}_2 = & \curly{ (\a,\beta) \in \tilde{E} \mid 
\frac{1}{2}<\a \le 1,\ 
0 \le \beta < \alpha
},\\
\tilde{F}_3 = & \curly{(\alpha,\beta) \in \tilde{E} \mid 0 \le \alpha<\beta,\; 0 < \beta < \frac{1}{2}},\\
\tilde{F}_5 = & \curly{ (\a,\beta) \in \tilde{E}  \mid
0 \le \alpha ,\
\frac{1}{2}<\beta\le 1,\ 
4\alpha - 2 \beta - 1 <0
}, \\
\tilde{F}_4 = & \curly{ (\a,\beta) \in \tilde{E}  \mid
4 \alpha - 2\beta - 1 >0,\
\alpha < \beta \le 1}, \\
\tilde{L}_{13} = & \curly{(\alpha,\beta) \in \tilde{E} \mid 0 \le \alpha<\frac{1}{2},\; \beta=\alpha},\\
\tilde{L}_{12} = & \curly{ (\a,\beta) \in \tilde{E} \mid 
\a =  \frac{1}{2},\
 0 \le \beta < \frac{1}{2}
},\\
\tilde{L}_{35} = & \curly{ (\a,\beta)  \in \tilde{E} \mid 
0 \le \a< \frac{1}{2},\ 
\beta = \frac{1}{2}},\\
\tilde{L}_{24} = & \curly{ (\a,\beta) \in \tilde{E} \mid 
 \frac{1}{2} < \alpha \le 1,\ 
\beta=\alpha
},\\
\tilde{L}_{45} = & \curly{ (\a,\beta) \in \tilde{E} \mid 
\frac{1}{2} < \a \le \frac{3}{4},\ 
4\alpha - 2\beta - 1  = 0}, \\
\tilde{P}_{12345} = & \curly{ (\a,\beta) \in \tilde{E} \mid 
\a = \frac{1}{2},\ 
\beta = \frac{1}{2}
}.
\end{align*}
\end{definition}

The partitions of $E$ and $\tilde{E}$ given in \Cref{def:partition-E} and \Cref{def:partition-E-tilde} are plotted in \Cref{fig:E-and-E-tilde}. In \Cref{def:partition-E} and \Cref{def:partition-E-tilde}, the regions are named with letters $V ( \bar{V},\tilde{V})$, $F( \bar{F},\tilde{F})$, $L( \bar{L},\tilde{L})$, or $P( \tilde{P})$, representing that the regions are polyhedrons, planes, lines, and points in the parameter space, respectively. The regions named with $V$, $F$, $L$, and $P$ are in $E$ where $\gamma \in [1,2]$. The regions named with $\bar{V}$, $\bar{F}$, and $\bar{L}$ are in $E$ where $\gamma \in [\frac{1}{2},1)$. Finally, the regions named with $\tilde{V}$, $\tilde{F}$, $\tilde{L}$, and $\tilde{P}$ belong to $\tilde{E}$.  These regions form disjoint subsets whose union is $E$ and $\tilde{E}$.

\subsection{Asymptotic forms of roots}
\label{sec:roots}
With the partitions of $E$ and $\tilde{E}$ defined in the foregoing section, we now present the asymptotic forms of the roots to the characteristic equations \eqref{eq:characteristic-different} and \eqref{eq:characteristic-same} in \Cref{thm:roots-different} and \Cref{thm:roots-same}, respectively. 

\begin{theorem}
\label{thm:roots-different}
Let $\lambda_{n,i}$'s be the four roots to \eqref{eq:characteristic-different}, where $i \in \curly{1,2,3,4}$. They have the asymptotic forms given in \Cref{tab:roots-different} across the regions of the parameter space $E$ partitioned according to \Cref{def:partition-E}.
\end{theorem}

\begin{theorem}
\label{thm:roots-same}
Let $\tilde{\lambda}_{n,i}$'s be the four roots to \eqref{eq:characteristic-same}, where $i \in \curly{1,2,3,4}$. They have the asymptotic forms given in \Cref{tab:roots-same} across the regions of the parameter space $\tilde{E}$ partitioned according to \Cref{def:partition-E-tilde}.
\end{theorem}

We highlight a few implications of \Cref{thm:roots-different} and \Cref{thm:roots-same}.
 First,  note that the union of $V_6$, $F_{56}$, $F_{46}$, $L_{3456}$, $L_{26}$,  $P_{1256}$,  $\bar{V}_6$, $\bar{F}_{56}$, $\bar{F}_{46}$, and $\bar{L}_{3456}$ correspond to the analytic region $R_1(\gamma)$, and the union of $\tilde{F}_4$, $\tilde{L}_{24}$,   $\tilde{L}_{45}$, and $\tilde{P}_{12345}$ correspond to the analytic region $\tilde{R}_1$, we omit the discussion of the roots in these regions as they are not needed to reason about the order of Gevrey class.

Secondly, the eigenvalues in the two theorems inform the formation of the regularity regions $R(\gamma)$ and $\tilde{R}$. In particular, from \Cref{tab:roots-different} and \Cref{tab:roots-same}, we can identify subsets of the regions in \Cref{def:partition-E} and \Cref{def:partition-E-tilde} that induce sequences of eigenvalues with a vertical asymptote (i.e., when the exponents of the real parts of the eigenvalues are smaller than or equal to 0, and the imaginary goes to infinity). Such subsets are excluded from regularity analysis as discussed in \Cref{sec:intro} and \Cref{sec:pre}, and the remaining subsets form $R(\gamma)$ and $\tilde{R}$. Furthermore, within $R(\gamma)$ and $\tilde{R}$, the two theorems provide Gevrey order candidates, as mentioned in \Cref{remark:eigenvalues}. In fact, based on \Cref{lemma:method}, these candidates are the ratios of the exponents of the imaginary parts over the exponents of the real parts of the complex conjugate roots of the regions in \Cref{tab:roots-different} and \Cref{tab:roots-same}. Note that when there are two pairs of complex conjugate roots in a region, the one with higher ratio is the candidate.  The regions that share the same candidate ratio shape $R_2(\gamma)$, $R_3(\gamma)$, $R_4(\gamma)$, and $R_5(\gamma)$  in \Cref{thm:different} and $\tilde{R}_2$, $\tilde{R}_3$, and  $\tilde{R}_4$ in \Cref{thm:same}.

Thirdly, the results reported in the two theorems cover regions of the parameter space that are both exponentially stable and polynomially stable. As such, these results extend those previously reported in \cite{Hao-Kuang-Liu-Yong 2025}, where only the roots associated with the polynomially stable regions are discussed. 

Finally, with the asymptotic forms of the eigenvalues reported in these two theorems, one can also show that the orders of Gevrey class obtained in the previous sections are optimal. Indeed, 
\begin{theorem}
\label{thm:optimality}
The order of Gevrey class described in \Cref{thm:different} and \Cref{thm:same} are optimal in the sense that the semigroup is not of Gevrey class order $\frac{1}{\mu+\varepsilon}$ for any $\varepsilon>0$.
\end{theorem}
\begin{proof}
Without loss of generality, we assume $(\a,\b,\g)\in R_2(\g)$.
By Table \ref{tab:roots-different}, for any $\varepsilon>0$, there always exists a sequence of eigenvalues such that
    $$\overline{\underset{n\to \infty}{\lim}}\frac{\Re \lambda_n}{|\Im \lambda_n|^{\mu+\varepsilon}}=0.$$
   By \cite[Corollary 2.2]{Hao-Liu-Yong 2015}, this implies that $e^{\mathcal{A_{\a,\b,\g}}t}$ is not of Gevrey class order $\frac{1}{\mu+\varepsilon}$.
\end{proof}

\begin{table}[H]
\centering
\caption{Asymptotic Forms of the Roots to \eqref{eq:characteristic-different}}
\bgroup
\resizebox{\textwidth}{!}{
\begin{tabular}{crr}
\hline
Region & $\lambda_{n,1}$ and $\lambda_{n,2}$ & $\lambda_{n,3}$  and $\lambda_{n,4}$\\
\hline
$V_{1}$ & 
$-\frac{1}{4}\m_n^{2\a+\beta-\g} (1+o(1)) \pm i\sqrt{2} \m_n^{\frac{\g}{2}}(1+o(1))$ & 
$-\frac{1}{2}\m_n^{\beta} (1+o(1)) \pm i \m_n^{\frac{1}{2}} (1+o(1))$\\
$V_{2}$ & 
$-\frac{1}{2}\m_n^{\beta}(1+o(1)) \pm i\m_n^{\a}(1+o(1))$ & 
$- \m_n^{-2\a+\beta+\g}(1+o(1)) \pm i\sqrt{2}\m_n^{-\a+\frac{\g}{2}+\frac{1}{2}}(1+o(1))$ \\
$V_{3}$ & 
$-\frac{1}{4}\m_n^{2\a+\beta-\g}(1+o(1)) \pm i\sqrt{2}\m_n^{\frac{\g}{2}}(1+o(1))$ & 
$-\m_n^{1-\beta}(1+o(1))$ and $-\m_n^{\beta}(1+o(1))$\\
$V_{4}$ & 
$-\frac{1}{2}\mu_n^{\beta}(1+o(1)) \pm i\mu_n^{\alpha}(1+o(1))$  & 
$-\mu_n^{1-\beta}(1+o(1))$ and 
$-2\mu_n^{-2\alpha+\beta+\gamma}(1+o(1))$\\
$V_{5}$ & 
$-\m_n^{1-\beta}(1+o(1))$ and
$-\m_n^{\beta}(1+o(1))$ & 
$-\frac{1}{2}\m_n^{2\a-\beta}(1+o(1)) \pm i\sqrt{2}\m_n^{\frac{\g}{2}}(1+o(1))$\\
$V_{6}$ &  $-\mu_n^{2\alpha-\beta}(1+o(1))$
and 
$-\mu_n^{\beta}(1+o(1))$ & 
$-2\mu_n^{-2\alpha+\beta+\gamma}(1+o(1))$  and  $-\mu_n^{1-\beta}(1+o(1))$ \\
$F_{12}$ & 
$-\frac{1}{6}\mu_n^{\beta}(1+o(1)) \pm i\sqrt{3}\mu_n^{\gamma/2}(1+o(1))$ & 
$-\frac{1}{3}\mu_n^{\beta}(1+o(1)) \pm i\frac{\sqrt{6}}{3}\mu_n^{\frac{1}{2}}(1+o(1))$ \\
$F_{13}$ & 
$-\frac{1}{4}\m_n^{2\a+\beta-\g}(1+o(1))\pm i \sqrt{2} \m_n^{\frac{\g}{2}}(1+o(1))$ & 
$-\frac{1}{2}\m_n^{\beta}(1+o(1)) \pm i\frac{\sqrt{3}}{2} \m_n^{\frac{1}{2}}(1+o(1))$ \\
$F_{24}$ & 
$-\frac{1}{2}\mu_n^{\beta}(1+o(1)) \pm i\mu_n^{\alpha}(1+o(1))$ & 
$(-1 \pm i)\mu_n^{1-\beta}(1+o(1))$ \\
$F_{34}$ & 
$-\frac{1}{6}\mu_n^{\beta}(1+o(1)) \pm i\sqrt{3}\mu_n^{\frac{\gamma}{2}}(1+o(1))$ & 
$-\mu_n^{1-\beta}(1+o(1))$ and 
$-\frac{2}{3}\mu_n^{\beta}(1+o(1))$\\
$F_{35}$ & 
$- \mu_n^{1-\beta }(1+o(1))$ and
$-\m_n^{\beta}(1+o(1))$ & 
$-\frac{1}{6}\m_n^{2\a-\beta}(1+o(1)) \pm i \sqrt{2} \m_n^{\frac{\g}{2}}(1+o(1))$\\
$L_{1234}$ & 
$-\frac{1}{6}\mu_n^{\frac{1}{2}}(1+o(1)) \pm i\sqrt{3}\mu_n^{\frac{\gamma}{2}}(1+o(1))$ & 
$(-\frac{1}{3}\pm i\frac{\sqrt{5}}{3})\mu_n^{\frac{1}{2}}(1+o(1))$\\
$F_{1}$ & 
$-\mu_n^{2\alpha+\beta-1}(1+o(1)) \pm i\sqrt{2}\mu_n^{\frac{1}{2}}(1+o(1))$ & 
$-\frac{1}{2}\mu_n^{\beta}(1+o(1)) \pm i\mu_n^{\frac{1}{2}}(1+o(1))$ \\
$L_{12}$ & 
$-\frac{1}{4}\mu_n^{\beta}(1+o(1))\pm i\sqrt{2+\sqrt{2}}\mu_n^{\frac{1}{2}}(1+o(\
1))$ & $-\frac{1}{4}\mu_n^{\beta}(1+o(1))\pm i\sqrt{2-\sqrt{2}}\mu_n^{\frac{1}{2}}(1+o(\
1))$\\
$L_{15}$ & 
$-\frac{1}{3}\mu_n^{2\alpha-\frac{1}{2}}(1+o(1)) \pm i\sqrt{2}\mu_n^{\frac{1}{2}}(1+o(1))$ & 
$-\frac{1}{2}\mu_n^{\frac{1}{2}}(1+o(1)) \pm \frac{\sqrt{3}}{2}i\mu_n^{\frac{1}{2}}(1+o(1))$ \\
$L_{26}$ & 
$-\frac{1}{2}\mu_n^{\beta}(1+o(1)) \pm i\frac{\sqrt{3}}{2} \mu_n^{\alpha}(1+o(1))$ & 
$(-1 \pm i)\mu_n^{1-\beta}(1+o(1))$ \\
$\bar{V}_{1}$ & 
$-\frac{1}{2}\mu_n^{\beta}(1+o(1)) \pm i\mu_n^{\frac{1}{2}}(1+o(1))$ & 
$-\mu_n^{2\alpha+\beta+\gamma-2}(1+o(1)) \pm  i\sqrt{2}\mu_n^{\frac{\gamma}{2}}(1+o(1))$\\
$\bar{V}_{2}$ & 
$-\frac{1}{2}\m_n^{\beta}(1+o(1)) \pm i\m_n^{\a}(1+o(1))$ & 
$-\m_n^{-2\a+\beta+\g}(1+o(1))\pm i\sqrt{2}\m_n^{-\a+\frac{\g}{2}+\frac{1}{2}}(1+o(1))$ \\
$\bar{V}_{3}$ & 
$-\m_n^{1-\beta}(1+o(1))$  and 
$-\m_n^{\beta}(1+o(1))$ & 
$-\m_n^{2\a+\beta+\g-2}(1+o(1)) \pm i\sqrt{2}\m_n^{\frac{\gamma}{2}}(1+o(1))$ \\
$\bar{V}_{4}$ & 
$-\m_n^{2\a-\beta}(1+o(1))$  and 
$-\m_n^{\beta}(1+o(1))$ & 
$-\m_n^{-2\a+\beta+\g}(1+o(1)) \pm i\sqrt{2}\m_n^{-\a+\frac{\g}{2}+\frac{
1}{2}}(1+o(1))$ \\
$\bar{V}_{5}$ & 
$-\m_n^{1-\beta}(1+o(1))$  and  
$-\m_n^{\beta}(1+o(1))$ & 
$-\frac{1}{2}\m_n^{2\a-\beta}(1+o(1)) \pm i\sqrt{2}\m_n^{\frac{\gamma}{2}}(1+o(1))$ \\
$\bar{F}_{12}$ & 
$-\frac{1}{2}\m_n^{\beta}(1+o(1)) \pm i\sqrt{2}\m_n^{\frac{1}{2}}(1+o(1))$ & 
$-\frac{1}{4}\m_n^{\beta+\g-1}(1+o(1)) \pm i\m_n^{\frac{\gamma}{2}}(1+o(1))$\\
$\bar{F}_{13}$ & 
$-\m_n^{\frac{1}{2}}(1+o(1)) \pm i\frac{\sqrt{3}}{2}\m_n^{\frac{1}{2}}(1+o(1)\
)$ & 
$-\m_n^{2\a+\g-\frac{3}{2}}(1+o(1))\pm i \sqrt{2}  \m_n^{\frac{\gamma}{2}}(1+o(1))$ \\
$\bar{F}_{24}$ & 
$-\frac{1}{2}\m_n^{\a}(1+o(1))\pm i \frac{\sqrt{3}}{2} \m_n^{\a}(1+o(1))$ & 
$-\m_n^{\g-\a}(1+o(1)) \pm i\sqrt{2}\m_n^{-\a+\frac{\g}{2}+\frac{1}{2}}(1+o(1))$ \\
$\bar{F}_{34}$ & 
$-2\mu_n^{1-\beta}(1+o(1))$ and 
$-\mu_n^{\beta}(1+o(1))$ & 
$-\frac{1}{4}\mu_n^{\beta+\gamma-1}(1+o(1)) \pm i\mu_n^{\gamma/2}(1+o(1))$\\
$\bar{F}_{35}$ & 
$-\m_n^{\frac{\gamma}{2}}(1+o(1))$  and
$-\m_n^{1-\frac{\g}{2}}(1+o(1))$ & 
$-\frac{1}{3}\m_n^{2\a+\frac{\g}{2}-1}(1+o(1)) \pm i\sqrt{2}\m_n^{\frac{\gamma}{2}}(1+o(1))$\\
$\bar{L}_{1234}$ & 
$-\frac{1}{2}\mu_n^\frac{1}{2}(1+o(1)) \pm i\frac{\sqrt{7}}{2}\mu_n^\frac{1}{2}(1+o(1))$ & 
$-\frac{1}{4}\mu_n^{\gamma-\frac{1}{2}}(1+o(1)) \pm i\mu_n^{\frac{\gamma}{2}}(1+o(1))$ \\
\hline
\end{tabular}
}
\egroup
\label{tab:roots-different}
\end{table}

\begin{table}[H]
\caption{Asymptotic Forms of the Roots to \eqref{eq:characteristic-same}}
\centering
\bgroup
\resizebox{\textwidth}{!}{
\begin{tabular}{crr}
\hline
Region & $\tilde{\lambda}_{n,1}$ and $\tilde{\lambda}_{n,2}$ & $\tilde{\lambda}_{n,3}$  and $\tilde{\lambda}_{n,4}$\\
\hline
$\tilde{F}_{1}$ & $-\frac{1}{4}\mu_n^{\beta}(1+o(1))-i\mu_n^{\frac{1}{2}}(1+o(1))$ &
$-\frac{1}{4}\mu_n^{\beta}(1+o(1))+i\mu_n^{\frac{1}{2}}(1+o(1))$\\
$\tilde{F}_{2}$ & 
$-\frac{1}{2}\mu_n^{\beta}(1+o(1))\pm i\mu_n^{\alpha}(1+o(1))$ & 
$-\frac{1}{2}\mu_n^{-2\alpha+\beta+1}(1+o(1))\pm i\mu_n^{1-\alpha}(1+o(1))$\\
$\tilde{F}_3$ & $\curly{-\frac{1}{2} \mu_n^{2 \alpha -\beta }, -\frac{1}{2}\mu_n^{\beta}} (1+o(1))+i \mu_n^{\frac{1}{2}} (1+o(1))$
& $\curly{-\frac{1}{2}\mu_n^{\beta}, -\frac{1}{2} \mu_n^{2 \alpha -\beta }} (1+o(1))-i \mu_n^{\frac{1}{2}} (1+o(1))$\\
$\tilde{F}_5$ & $-\mu_n^{1-\beta }(1+o(1))$ and $-\mu_n^{\beta }(1+o(1))$ & 
$-\frac{1}{2}\mu_n^{2\alpha-\beta}(1+o(1))\pm i \mu_n^{\frac{1}{2}}(1+o(1))$
\\
$\tilde{L}_{12}$ & $-\frac{5+\sqrt{5}}{20} \mu_n^{\beta}(1+o(1)) \pm i \frac{\sqrt{5}+1}{2} \mu_n^{\frac{1}{2}}(1+o(1))$ &
$-\frac{5-\sqrt{5}}{20} \mu_n^{\beta}(1+o(1)) \pm i \frac{\sqrt{5}-1}{2} \mu_n^{\frac{1}{2}}(1+o(1))$
\\
$\tilde{L}_{35}$ & 
$-\frac{1}{2}\mu_n^{2\alpha-\frac{1}{2}}(1+o(1))\pm i \mu_n^{\frac{1}{2}}(1+o(1))
$ & $-\frac{1}{2}\mu_n^{\frac{1}{2}}(1+o(1))\pm i\frac{\sqrt{3}}{2} \mu_n^{\frac{1}{2}}(1+o(
1))$
\\
$\tilde{L}_{13}$ & $-\frac{1}{4}\mu_n^{\beta}(1+o(1))-i\mu_n^{\frac{1}{2}}(1+o(1))$ &
$-\frac{1}{4}\mu_n^{\beta}(1+o(1))+i\mu_n^{\frac{1}{2}}(1+o(1))$
\\
\hline
\end{tabular}
}
\egroup
\label{tab:roots-same}
\end{table}

\section{Applications}\label{sec:applications}
\setcounter{equation}{0}
\setcounter{theorem}{0}
\rm

In this section, we give several examples of two coupled second order PDEs where $a$ is a positive constant, and apply the main results in this paper to obtain the regularity of these systems. 
Let $\Omega$ be a bounded domain in $R^{n}$ with smooth boundary $\partial\Omega.$
\vspace{3mm}

\noindent Example 1.
\be
\left\{ \begin{array}{ll}  u_{tt}(x,t) = -a \Delta^2 u(x,t) + b(-\Delta)^\alpha y_t(x,t), & x\in\Omega, \quad t>0, \\
     y_{tt}(x,t)= \Delta y(x,t) - b(-\Delta)^\alpha u_t(x,t) + k\Delta y_t(x,t), & x\in\Omega, \quad t>0, \\
     u(x,t)=\Delta u(x,t)=y(x,t) = 0, & x\in \partial\Omega,\; t>0, \end{array} \right.
\ee
where $0\le \alpha \le \frac{3}{2}$, and $a,k>0$, $b\ne 0$. When $\alpha=1$, this system models a type III thermo-elastic plate and was investigated in \cite{Liu-Quintanilla 2010}. They showed that the system is analytic. The general case $\alpha\in [0, \frac{3}{2}]$ was studied in \cite{Zelati-Delloro-Pata 2013}. They showed that the system is exponentially stable only if $\alpha \ge \frac{1}{2}$. 

Let $\cH := (H_0^1(\Omega)\cap H^2(\Omega))\times L^2(\Omega)\times H_0^1(\Omega)\times L^2(\Omega)$. Define $Af=-\Delta f$  with ${\mathcal D}(A) = \{f\in  H^1(\Omega)\mid Af  \in L^2(\Omega),\; f= 0 \; \mbox{on}\;  \partial\Omega \}$}.  Then the system corresponds to
our abstract system with $\gamma = 2,\; \beta=1$ and $\alpha\in [0,\frac{3}{2}]$. 
By Theorem \ref{thm:different}, we conclude that the system is  (see Figure \ref{fig:r0}e)
$$ 
\left\{ \begin{array}{ll} 
 \mbox{not differentiable} & \mbox{when} \quad \alpha\in [0,\frac{1}{2}]; \\
\mbox{a Gevrey class}\; \delta> \frac{1}{2\alpha -1} & \mbox{when} \quad \alpha \in (\frac{1}{2}, 1); \\
\mbox{analytic} & \mbox{when} \quad \alpha =1; \\
\mbox{a Gevrey class}\; \delta> \alpha & \mbox{when} \quad \alpha \in (1, \frac{3}{2}]. 
\end{array} \right.
$$

\vspace{0.5 cm}

\noindent Example 2.
\be
\left\{ \begin{array}{ll}  
     u_{tt}(x,t)= a\Delta u(x,t)  + b (-\Delta)^\alpha y_t(x,t), & x\in\Omega, \quad t>0, \\ y_{tt}(x,t) = - \Delta^2 y(x,t) - b(-\Delta)^\alpha u_t(x,t) + k\Delta y_t(x,t), & x\in\Omega, \quad t>0, \\
     u(x,t)=y(x,t)=\Delta y(x,t) = 0, & x\in \partial\Omega,\; t>0, \end{array} \right.
\ee
where $0\le \alpha \le \frac{3}{4}$, and $a,k>0$, $b\ne 0$. Let $\cH := H_0^1(\Omega)\times L^2(\Omega)\times (H_0^1(\Omega)\cap H^2(\Omega))\times L^2(\Omega)$. Define $A$ by $Af(x)= \Delta^2 f(x)$  with ${\mathcal D}(A) = \{f\in  H^2(\Omega)\mid Af  \in L^2(\Omega),\; f=\Delta f = 0 \; \mbox{on}\;  \partial\Omega \}$. Then the system corresponds to 
the abstract system with $\beta = \gamma=\frac{1}{2}$ and $\alpha\in [0,\frac{3}{4}]$. Since it does not intersect with the region $R(\frac{1}{2})$, the system is not differentiable
for all $\alpha\in [0,\frac{3}{4}]$  (see Figure \ref{fig:r0}a}).

The results from Examples 1 and 2 show that for a coupled plate-wave system, the solution regularity is  weaker when damping is applied to the plate equation than when it is applied to the wave equation.

\vspace{0.5 cm}

\noindent Example 3.
\be
\left\{ \begin{array}{ll}  u_{tt}(x,t) = -a \Delta^2 u(x,t) - b\Delta y_t(x,t), & x\in\Omega, \quad t>0, \\
     y_{tt}(x,t)= \Delta y(x,t)  +  b\Delta u_t(x,t) - k(-\Delta)^\beta y_t(x,t), & x\in\Omega, \quad t>0, \\
     u(x,t)=\Delta u(x,t)=y(x,t) = 0, & x\in \partial\Omega,\; t>0, \end{array} \right.
\ee
where $0\le \beta \le 1$, and $a,k>0$, $b\ne 0$. This system was investigated in \cite{Suarez-Mendes 2021}. They showed that the system is analytic when $\beta=1$; the system is a Gevrey class of $\delta> \frac{2+\beta}{\beta}$, which is clearly not sharp since it is not equal to 1 when $\beta=1$. 

Let $\cH := (H_0^1(\Omega)\cap H^2(\Omega))\times L^2(\Omega)\times H_0^1(\Omega)\times L^2(\Omega)$. Define $Af=-\Delta f$ with ${\mathcal D}(A) = \{f\in  H^1(\Omega)\mid Af  \in L^2(\Omega),\; f = 0 \; \mbox{on}\;  \partial\Omega \}$.  Then the system corresponds to our abstract system with $\gamma = 2,\; \alpha=1$ and $\beta\in [0,1]$.  By Theorem \ref{thm:different}, we conclude that the system is (see Figure \ref{fig:r0}e).
$$ 
\left\{ \begin{array}{ll} 
\mbox{not differentiable} & \mbox{when} \quad \beta=0; \\
\mbox{a Gevrey class}\; \delta> \frac{1}{\beta} & \mbox{when} \quad \beta \in (0,1); \\
\mbox{analytic} & \mbox{when} \quad \beta=1. 
\end{array} \right.
$$

\vspace{0.5 cm}

\noindent Example 4.
\be
\left\{ \begin{array}{ll}  
     u_{tt}(x,t)= a\Delta u(x,t)  - b \Delta y_t(x,t), & x\in\Omega, \quad t>0, \\ y_{tt}(x,t) = - \Delta^2 y(x,t) + b\Delta u_t(x,t) - k(-\Delta)^\beta y_t, & x\in\Omega, \quad t>0, \\
     u(x,t)=y(x,t)=\Delta y(x,t) = 0, & x\in \partial\Omega,\; t>0. \end{array} \right.
\ee
Let $\cH := H_0^1(\Omega)\times L^2(\Omega)\times (H_0^1(\Omega)\cap H^2(\Omega))\times L^2(\Omega)$. Define $A$ by $Af = \Delta^2 f $ with ${\mathcal D}(A) = \{f\in  H^2(\Omega)\mid Af  \in L^2(\Omega),\; f=\Delta f = 0 \; \mbox{on}\;  \partial\Omega \}$. Then the system corresponds to 
the abstract system with $\alpha = \gamma=\frac{1}{2}$ and $\beta\in [0,1]$.  By Theorem \ref{thm:different} we conclude that the system is  (see Figure  \ref{fig:r0}a)
$$ 
\left\{ \begin{array}{ll} 
\mbox{not differentiable} & \mbox{when} \quad \beta\in [0, \frac{1}{2}]; \\
\mbox{a Gevrey class}\; \delta> \frac{1}{4\beta - 2} & \mbox{when} \quad \beta \in (\frac{1}{2}, \frac{3}{4}); \\
\mbox{analytic} & \mbox{when} \quad \beta=\frac{3}{4}; \\
\mbox{a Gevrey class}\; \delta> \frac{1}{4(1-\beta)} & \mbox{when} \quad \beta \in (\frac{3}{4}, 1); \\
\mbox{not differentiable} & \mbox{when} \quad \beta=1,
\end{array} \right.
$$

\vspace{5 mm}

\noindent Example 5.
\be
\left\{ \begin{array}{ll}  u_{tt}(x,t) = -a \Delta^2 u(x,t) - b\Delta y_t(x,t), & x\in\Omega, \quad t>0, \\
     y_{tt}(x,t)= -\Delta^2 y(x,t)  +  b\Delta u_t(x,t) - k(\Delta^2)^\beta y_t(x,t)], & x\in\Omega, \quad t>0, \\
     u(x,t)=y(x,t)=\Delta y(x,t) = 0, & x\in \partial\Omega,\; t>0, \end{array} \right.
\ee
where $0\le \beta \le 1$, and $a,k>0$, $b\ne 0$. This system was investigated in \cite{Han-Liu 2019} for two special cases: (i) $a=1,\beta=\frac{1}{2}$ and (ii) $a=1,\beta=1$. They showed that the system is analytic when $\beta=\frac{1}{2}$; the system is not differentiable when $\beta=1$. 

Let $\cH := (H_0^1(\Omega)\cap H^2(\Omega))\times L^2(\Omega)\times (H_0^1\cap H^2(\Omega))\times L^2(\Omega)$. Define $A$ by $Af =\Delta^2f $ with ${\mathcal D}(A) = \{f\in  H^2(\Omega)\mid Af  \in L^2(\Omega),\; f=\Delta f = 0 \; \mbox{on}\;  \partial\Omega \}$. Then the system corresponds to 
our abstract system with $\gamma = 1$, $\alpha=\frac{1}{2}$, and $\beta\in [0,1]$. Thus, by Theorem \ref{thm:same}, we have that the system is  (see Figure 3b)
$$ 
\left\{ \begin{array}{ll} 
\mbox{not differentiable} & \mbox{when} \quad \beta=0; \\
\mbox{a Gevrey class}\; \delta> \frac{1}{2\beta} & \mbox{when} \quad \beta \in (0, \frac{1}{2}); \\
\mbox{analytic} & \mbox{when} \quad \beta=\frac{1}{2}; \\
\mbox{a Gevrey class}\; \delta> \frac{1}{2(1-\beta)} & \mbox{when} \quad \beta \in (\frac{1}{2}, 1); \\
\mbox{not differentiable} & \mbox{when} \quad \beta=1;
\end{array} \right.
$$
However, by Theorem \ref{thm:different}, the above conclusions remain valid even when $a \ne 1$ (see Figure 3a). This case is somewhat special, as one would generally expect the regularity to improve when the wave speeds are the same. 

\vspace{5 mm}
\noindent Example 6.
\be
\left\{ \begin{array}{ll}  u_{tt}(x,t) = -a \Delta^2u(x,t) - b(-\Delta)^\frac{1}{2} y_t(x,t), & x\in\Omega, \quad t>0, \\
     y_{tt}(x,t)= -\Delta^2 y(x,t)  +  b(-\Delta)^\frac{1}{2} u_t(x,t) - k(-\Delta)^\beta y_t(x,t)], & x\in\Omega, \quad t>0, \\
     u(x,t)=y(x,t)=\Delta y(x,t) = 0, & x\in \partial\Omega,\; t>0, \end{array} \right.
\ee
where $0\le \beta \le 1$, and $a,k>0$, $b\ne 0$. 

Let $\cH := (H_0^1(\Omega)\cap H^2(\Omega))\times L^2(\Omega)\times (H_0^1(\Omega)\cap H^2(\Omega))\times L^2(\Omega)$. Define $A$ by $Af =\Delta^2f $  with ${\mathcal D}(A) = \{f\in  H^2(\Omega)\mid Af  \in L^2(\Omega),\; f=\Delta f = 0 \; \mbox{on}\;  \partial\Omega \}$. Then the system corresponds to 
our abstract system with $\gamma = 1$, $\alpha=\frac{1}{4}$, and $\beta\in [0,1]$. 

If $a\ne 1$, by Theorem \ref{thm:different}, the associated semigroup is not differentiable for all $\beta\in [0, 1]$
(see Figure 3a). 

If $a=1$, by Theorem \ref{thm:same}, the system is  (see Figure 3b)
$$ 
\left\{ \begin{array}{ll}
\mbox{not differentiable} & \mbox{when} \quad \beta=0; \\
\mbox{a Gevrey class} \; \delta> \frac{1}{2\beta}& \mbox{when} \quad \beta\in (0,\frac{1}{4}); \\
\mbox{a Gevrey class}\; \delta> \frac{1}{1-2\beta} & \mbox{when} \quad \beta \in [\frac{1}{4}, \frac{1}{2}); \\
\mbox{not differentiable} & \mbox{when} \quad \beta\in [\frac{1}{2},1];
\end{array} \right.
$$

\section{Conclusion}\label{sec:conclusion}

In this paper, we conducted a comprehensive regularity analysis of the coupled hyperbolic system with indirect damping \eqref{1.1}. The system generates a strongly continuous semigroup of contractions $e^{\mathcal{A}_{\a,\b,\g}t}$ on the Hilbert space $D(A^\frac{\gamma}{2})\times H \times D(A^\frac{1}{2})\times H$.  While the asymptotic stability of the same system was previously established in \cite{Hao-Kuang-Liu-Yong 2025}, we here refine and extend the regularity characterization by considering two distinct cases based on the wave speeds.

(i) \textbf{The case of different wave speeds, i.e., $(a,\gamma)\neq (1,1)$.} 
The parameter space  $E= [0, \frac{\gamma+1}{2}] \times [0,1] \times [\frac{1}{2}, 2]$ was partitioned in \cite{Hao-Kuang-Liu-Yong 2025} into five disjoint regions  $S_1(\gamma), S_2(\gamma), \cdots, S_5(\gamma)$, where the system is exponentially stable in $S_1(\gamma)$; polynomially stable with optimal decay rates in $S_2(\gamma)-S_4(\gamma)$; and strongly stable in $ S_5(\gamma)$. In the present work, we further subdivide the interior of the exponentially stable region $S_1(\gamma)$, together with the boundary $\beta=1$, into five parts $R_1(\gamma), R_2(\gamma),\cdots,R_5(\gamma)$. We prove that the  $C_0$-semigroup generated by \eqref{1.1} is analytic in $R_1(\gamma)$, and belongs to Gevrey classes of different orders in $ R_2(\gamma)$–$R_5(\gamma)$. In contrast, for $(\alpha, \beta, \gamma) \notin \cup_{i=1}^5 R_i$, the semigroup is not differentiable, as there exists a sequence of system eigenvalues that approach the imaginary axis asymptotically, violating the spectral condition required for differentiable semigroups. The corresponding results are summarized in Table~\ref{tab:stability-regularity1}.

\begin{table}[h]
\centering
\caption{Stability and regularity properties over parameter regions when $(a,\gamma)\neq (1,1)$.}
\renewcommand{\arraystretch}{1.2}
\resizebox{\textwidth}{!}{
\begin{tabular}{c|c|c}
\hline
\textbf{Region} & \textbf{Stability} & \textbf{Regularity} \\
\hline
\multirow{6}{*}{%
  \begin{tabular}[c]{@{}c@{}}
  $S_1(\gamma) \cap R_1(\gamma)$ \\
  $S_1(\gamma) \cap R_2(\gamma)$ \\
  $S_1(\gamma) \cap R_3(\gamma)$ \\
  $S_1(\gamma) \cap R_4(\gamma)$ \\
  $S_1(\gamma) \cap R_5(\gamma)$ \\
  $S_1(\gamma) \setminus \cup_{i=1}^5 R_i(\gamma)$
  \end{tabular}
} 
& \multirow{6}{*}{Exponentially stable} 
&  Analytic \\
&   &Gevrey class $\delta > \frac{\gamma}{2(2\alpha - \beta)}$ \\
&  &Gevrey class $\delta > \frac{\gamma}{2(2\alpha + \beta - (\gamma \vee (2 - \gamma)))}$ \\
&  & Gevrey class $\delta > \frac{\alpha}{\beta}$ \\
&  & Gevrey class $\delta > \frac{-2\alpha + \gamma + 1}{2(-2\alpha + \beta + \gamma)}$ \\
& &Not differentiable \\
\hline
$S_2(\gamma)$ & Polynomially stable of order $\frac{\gamma}{2(\beta - 2\alpha)}$ & \multirow{4}{*}{Not differentiable} \\
$S_3(\gamma)$ & Polynomially stable of order $\frac{\gamma}{2(|\gamma - 1| + 1 - \beta - 2\alpha)}$ &  \\
$S_4(\gamma)$ & Polynomially stable of order $\frac{\gamma + 1 - 2\alpha}{2(2\alpha - \beta - \gamma)}$ & \\
$S_5(\gamma)$ & Strongly stable & \\
\hline
\end{tabular}
}
\label{tab:stability-regularity1}
\end{table}

(ii) \textbf{The case of identical wave speeds, i.e., $(a, \gamma) = (1,1)$.} 
In this setting, the parameter space $\tilde{E} = [0,1] \times [0,1]$ was divided into the exponentially stable region $ \tilde{S}_1 $, the polynomially stable regions $\tilde{S}_2, S_4(1)$, and the strongly stable region $,S_5(1)$ in \cite{Hao-Kuang-Liu-Yong 2025}. We further partition the interior of the region $\tilde{S}_1$, together with the boundary $\beta=1$, into four subregions. We show that the semigroup is analytic in  $\tilde{R}_1$and Gevrey of various orders in  $\tilde{R}_2–\tilde{R}_4 $. For the same reason, one sees that the semigroup is not differentiable when $(\alpha, \beta, \gamma) \notin \cup_{i=1}^4 \tilde{R}_i$. The results are detailed in Table~\ref{tab:stability-regularity2}.
\begin{table}[h]
\centering
\caption{Stability and regularity properties over parameter regions when $(a,\gamma)=(1,1)$.}
\renewcommand{\arraystretch}{1.2}
\begin{tabular}{c|c|c}
\hline
\textbf{Region} & \textbf{Stability} & \textbf{Regularity} \\
\hline
\multirow{5}{*}{%
  \begin{tabular}[c]{@{}c@{}}
  $\tilde{S}_1\cap \tilde{R}_1$ \\
  $\tilde{S}_1 \cap \tilde{R}_2$ \\
  $\tilde{S}_1 \cap \tilde{R}_3$ \\
  $\tilde{S}_1 \cap \tilde{R}_4$ \\
  $\tilde{S}_1 \setminus \cup_{i=1}^4 \tilde{R}_i$
  \end{tabular}
} 
& \multirow{5}{*}{Exponentially stable} 
&  Analytic \\
&   &Gevrey class $\delta > \frac{1}{2(2\alpha-\beta)}$ \\
&  &Gevrey class $\delta > \frac{1}{2\beta}$  \\
&  & Gevrey class $\delta > \frac{1-\alpha}{-2\alpha+\beta+1}$ \\
& &Not differentiable \\
\hline
$\tilde{S}_2$ & Polynomially stable of order $\frac{1}{2(\beta - 2\alpha)}$ &  \multirow{3}{*}{Not differentiable} \\
$S_4(1)$ & Polynomially stable of order $\frac{1 - \alpha}{2\alpha - \beta - 1}$ & \\
$S_5(1)$ & Strongly stable & \\
\hline
\end{tabular}
\label{tab:stability-regularity2}
\end{table}

 Comparing to the case of different wave speeds, the identical wave speed setting yields better stability and regularity results, see Remark \ref{7.19}. Moreover, by analyzing the asymptotic expansion of the generator’s eigenvalues in \Cref{sec:Asymptotic}, we see the Gevrey orders given in the above tables are optimal.

\end{document}